\documentclass[a4paper,twoside,11pt]{amsart}
\pdfoutput=1
\usepackage{amssymb}
\usepackage{mathtools}
\usepackage[utf8]{inputenc} 
\usepackage[all]{xy}
\usepackage{microtype}
\usepackage[english]{babel}
\usepackage{mathdots}

\usepackage{stmaryrd}
\usepackage{enumitem}
\usepackage{array}
\usepackage[breaklinks]{hyperref}
\usepackage{mathrsfs}

\numberwithin{equation}{subsection}
\theoremstyle{plain}
\newtheorem{thm}[equation]{Theorem}
\newtheorem{cor}[equation]{Corollary}
\newtheorem{lem}[equation]{Lemma}
\newtheorem{prop}[equation]{Proposition}
\newtheorem{defn}[equation]{Definition}
\newtheorem{remark}[equation]{Remark}

\newcommand\id{\mathrm{id}}

\newcommand\NN{\mathbb{N}}
\newcommand\RR{\mathbb{R}}
\newcommand\GG{\mathbb{G}}
\newcommand\QQ{\mathbb{Q}}
\newcommand\ZZ{\mathbb{Z}}
\newcommand\CC{\mathbb{C}}

\newcommand\DD{\mathbb{D}}
\newcommand{\FF}{\mathbb{F}}
\newcommand{\AF}{\mathbb{A}}
\newcommand{\GL}{\mathrm{GL}}

\newcommand{\GU}{\mathrm{GU}}
\newcommand{\U}{\mathrm{U}}
\newcommand{\GSp}{\mathrm{GSp}}
\newcommand{\Sp}{\mathrm{Sp}}

\newcommand{\xyhookrightarrow}{\,\ar@{^{(}->}[r]}
\newcommand{\Mloc}{M^{\mathrm{loc}}}
\newcommand{\Me}{M}
\newcommand{\too}{\xrightarrow}

\newcommand{\Adm}{\mathrm{Adm}(\mu)}

\newcommand{\pair}[2]{\left< #1,#2 \right>}
\newcommand{\paird}{\left<\cdot,\cdot\right>}
\newcommand{\pairt}[2]{\left( #1,#2 \right)}
\newcommand{\pairtd}{\left(\cdot,\cdot\right)}

\newcommand{\val}{\mathrm{val}}

\newcommand{\mf}{\mathfrak}
\newcommand{\mc}{\mathcal}

\DeclareMathOperator{\Spec}{Spec}
\DeclareMathOperator{\Int}{Int}
\DeclareMathOperator{\Res}{Res}
\DeclareMathOperator{\Lf}{L}
\DeclareMathOperator{\Lp}{L^+}

\DeclareMathOperator{\ch}{char}
\DeclareMathOperator{\Gal}{Gal}

\DeclareMathOperator{\Hom}{Hom}
\DeclareMathOperator{\Aut}{Aut}
\DeclareMathOperator{\Isom}{Isom}

\DeclareMathOperator{\End}{End}

\DeclareMathOperator{\im}{im}

\DeclareMathOperator{\rk}{rk}

\DeclareMathOperator{\tr}{tr}

\DeclareMathOperator{\Lie}{Lie}

\DeclareMathOperator{\diag}{diag}

\renewcommand{\det}{\operatorname{det}}
\newcommand{\fp}{F_\mc P}
\newcommand{\fpn}{(F_0)_{\mc P_0}}
\newcommand{\aff}{\widetilde}
\newcommand{\res}{\overline}
\newcommand{\inv}{\ast}
\newcommand{\Perm}{\mathrm{Perm}}
\newcommand{\pd}{{\mathrm{p}\text{-}\mathrm{div}}}
\begin{document}
\author{Philipp Hartwig}
\address{Universität Duisburg-Essen\\
Institut für experimentelle Mathematik\\
Ellernstr. 29\\
45326 Essen\\
Germany}
\email{philipp.hartwig@uni-due.de}
\urladdr{http://www.esaga.uni-due.de/philipp.hartwig}
\title[Kottwitz-Rapoport and $p$-rank strata]{Kottwitz-Rapoport and $p$-rank strata in the reduction of Shimura varieties of PEL type}
\begin{abstract}
    We study the reduction of certain integral models of Shimura varieties of PEL type 
    with Iwahori level structure. On these spaces we have the Kottwitz-Rapoport 
    and the $p$-rank stratification. We show that the $p$-rank 
    is constant on a KR stratum, generalizing a result of Ng\^o and Genestier. We
    prove an abstract, uniform formula for the $p$-rank on a KR stratum.
    In the symplectic and in the unitary case
    we derive explicit formulas for its value. We apply these formulas to the question of
    the density of the ordinary locus and to the question of the dimension of the
    $p$-rank 0 locus.
\end{abstract}
\maketitle
\section{Introduction}
Fix a rational prime $p\neq 2$ and a PEL datum $\mc B=(B,\inv,V,\pairtd,J)$ with 
auxiliary data $\mc B_p=(\mc O_B,\mc L)$, see Section \ref{SecPelData}. The 
datum $\mc B$ gives rise to a reductive group $G$ over $\QQ$ and a conjugacy 
class $h$ of homomorphisms $\Res_{\CC/\RR}\GG_m\to G_\RR$.  Fix a compact open 
subgroup $C^p\subset G(\AF_f^p)$.  From $C^p$ and $\mc B_p$ one obtains a 
compact open subgroup $C\subset G(\AF_f)$ and thus a Shimura datum $(G,h,C)$. In 
\cite[Section 6]{rz}, Rapoport and Zink construct from $\mc B,\ \mc B_p$ and 
$C^p$ an integral model $\mc A=\mc A_{C^p}$ of the Shimura variety associated 
with $(G,h,C)$. Concretely $\mc A$ is defined as a moduli space of abelian 
schemes with additional structure. 

In order to study properties of the scheme $\mc A$, Rapoport and Zink introduce 
the so-called local model $\Mloc$. It is defined purely in terms of linear 
algebra and therefore easier to investigate than $\mc A$. The schemes $\mc A$ 
and $\Mloc$ are related via an intermediate object $\widetilde{\mc A}$ fitting 
into the so called \emph{local model diagram}
\begin{equation*}
    \xymatrix{
        &\ar[dl]_{\widetilde{\varphi}}\widetilde{\mc 
        A}\ar[dr]^{\widetilde{\psi}}&\\
        \mc A&&\Mloc.
    }
\end{equation*}
\'Etale locally on $\mc A$, there is a section $s:\mc A\to \widetilde{\mc A}$ of 
$\widetilde{\varphi}$ such that the composition
$\mc A\too{s}\widetilde{\mc A}\too{\widetilde{\psi}} \Mloc$ is \'etale.  
Consequently $\mc A$ inherits any property from $\Mloc$ which is local for the 
\'etale topology.  In particular, questions about singularities of $\mc A$ or 
the flatness of $\mc A$ can equivalently be studied for $\Mloc$. Let us mention 
that
recently a purely group-theoretic definition of the local model was given in 
\cite{pz} by Pappas and Zhu, providing an intriguing new perspective on the 
local model diagram.

The PEL datum also gives rise to an affine smooth group scheme $\Aut(\mc L)$, 
and $\Aut(\mc L)$ acts on both $\tilde{A}$ and $\Mloc$. The map 
$\widetilde{\varphi}$ is an $\Aut(\mc L)$-torsor, while the map 
$\widetilde{\psi}$ is $\Aut(\mc L)$-equivariant.
Denote by $\FF$ an algebraic closure of $\FF_p$. Via the local model diagram, 
the decomposition of $\Mloc(\FF)$ into $\Aut(\mc L)(\FF)$-orbits induces the 
\emph{Kottwitz-Rapoport} (or KR) \emph{stratification}
\begin{equation*}
    \mc A(\FF)=\coprod_{x\in \Aut(\mc L)(\FF)\backslash \Mloc(\FF)}\mc A_{x},
\end{equation*}
which was first introduced by Ng\^o and Genestier in \cite{genestier_ngo}.
The $\Aut(\mc L)(\FF)$-orbits on $\Mloc(\FF)$ admit the following interesting 
description.
In all cases considered thus far (cf.\ the discussion in \cite[\textsection 
3.3]{prs}), the special fiber $\Mloc_\FF$ of $\Mloc$ embeds into an \emph{affine 
flag variety $\mc F$}, which is defined
as a moduli space of lattice chains over the ring of formal power series 
$\FF[\![u]\!]$. In analogy with the Bruhat decomposition of the classical flag 
variety, indexed by
the finite Weyl group $W$, the affine flag variety admits the
\emph{Iwahori decomposition} $\mc F(\FF)=\coprod_{x\in\widetilde{W}}\mc C_x$ 
into Schubert cells $\mc C_x$, indexed by the extended affine Weyl group 
$\widetilde{W}$. It then turns out that $\Mloc_\FF\subset \mc F$ is a disjoint 
union of Schubert cells and that the decomposition
$\Mloc(\FF)=\coprod_{\mc C_x\subset \Mloc(\FF)}\mc C_x$ coincides with the
decomposition of $\Mloc(\FF)$ into $\Aut(\mc L)(\FF)$-orbits. As in the case of 
the Bruhat decomposition, many properties of the Iwahori decomposition are 
easily expressed by combinatorial properties of the corresponding index element 
in $\widetilde{W}$.  Notably, the dimension of $\mc C_x$ is given by the length 
$\ell(x)$ of $x$ in  $\widetilde{W}$, and the closure relation between Schubert 
cells is expressed by the Bruhat order on $\widetilde{W}$.
We conclude that the same statements hold for the KR stratification on
$\mc A(\FF)$.

Let us explain in detail one case in which this convenient combinatorial 
behavior of the KR stratification
was fruitfully exploited. For $B=\QQ$ and a complete lattice chain $\mc L$, the 
moduli problem $\mc A$ specializes to the Siegel moduli space $\mc A_I$ of 
principally polarized abelian varieties with Iwahori level structure. In 
\cite{gy2}, Görtz and Yu compute the dimension of the $p$-rank 0 locus in $\mc 
A_I$, and this computation was later generalized in \cite{hamacher} by Hamacher 
to the case of all $p$-rank strata. The method is the same in both cases: 
Determine all KR strata contained in a given $p$-rank stratum and compute the 
maximum of their dimensions. For this method to work one of course needs to know 
that a $p$-rank stratum is indeed the union of the KR strata contained in it.  
Thus both papers depend crucially on the result \cite[Th\'eor\`eme 
4.1]{genestier_ngo} of Ng\^o and Genestier, which states that indeed the 
$p$-rank is constant on a KR stratum in $\mc A_I$, and also provides an explicit 
formula for the $p$-rank on a given KR stratum.

The subject of this paper is to generalize the result of Ng\^o and Genestier on 
the relationship between the KR and the $p$-rank stratification to more general 
PEL data. Let us give an outline of the structure of this paper and of the 
results that we have obtained.

In Sections \ref{SecPelData} through \ref{SecLocalModelGen}, we recall the 
construction of the local model diagram. We then show the following result.
\begin{thm}
    Let $\mc B$ be an arbitrary PEL datum. If $\mc L$ is \emph{complete} (in the 
    sense of Definition \ref{DefnCompleteness}), the $p$-rank is constant on a 
    KR stratum.
\end{thm}
Before being able to state our next result, we need some more notation.
Denote by $\mc O_K=W(\FF)$ the Witt ring of $\FF$, by $K$ the fraction field of 
$\mc O_K$ and by $\sigma$ the Frobenius on $K$.
Denote by $\DD$ the diagonalizable affine group with character group $\QQ$ over 
$K$.  For $b\in G(K)$ denote by $\nu_b:\DD\to G_K$ the corresponding Newton map.  
By definition, the group $G_K$ acts on $V_K$ and thus $\nu_b$ gives rise to a 
representation of $\DD$ on $V_K$. Consider the corresponding weight 
decomposition $V_K=\oplus_{\chi\in\QQ} V_\chi$ and define
\begin{equation*}
    \nu_{b,0}:=\dim_K V_0.
\end{equation*}
Denote by $I$ the stabilizer of $\mc L\otimes\mc O_K$ in $G(K)$. For $b\in G(K)$ 
and $x\in I\backslash G(K)/I$ we denote by $X_x(b)=\{g\in G(K)/I\mid 
    g^{-1}b\sigma(g)\in IxI\}$ the affine Deligne-Lusztig variety associated 
    with $b$ and $x$.

By interpreting the KR stratification in terms of the relative position of $\mc 
L\otimes\mc O_K$ to its image under Frobenius, we show that the
KR strata are precisely the fibers of a canonical map $\gamma:\mc A(\FF)\to 
I\backslash G(K)/I$. Denote by $\Perm\subset I\backslash G(K)/I$ the image of 
$\gamma$, and for $x\in \Perm$ by $\mc A_x=\gamma^{-1}(x)$ the corresponding KR 
stratum.
\begin{thm}
    Let $x\in\Perm$ and let $b\in G(K)$.  Assume that $X_x(b)\neq \emptyset$.  
    Then the $p$-rank on $\mc A_{x}$ is constant with value $\nu_{b,0}$.
\end{thm}
In Sections \ref{SecSym} through \ref{SecUni3} we turn to the aforementioned 
interpretation of the KR stratification in terms of the affine flag variety.  
Section \ref{SecSym} deals with the case of the symplectic group. Section 
\ref{SecUni} (resp.\ \ref{SecUni2}, resp.\ \ref{SecUni3}) deals with the case of 
a unitary group associated with a ramified (resp.\ inert, resp.\  split) 
quadratic extension. Let us note that the embedding of $\Mloc_\FF$ into an 
affine flag variety has a long history. In particular we want to emphasize that 
we have greatly profited from the expositions by Pappas and Rapoport in 
\cite{pr2}, \cite{pr}, \cite{pr3}, and by Smithling in 
\cite{smithling_unitary_odd}, \cite{smithling_unitary_even}. We have decided to 
repeat part of their discussions, on the one hand for the convenience of the 
reader, and on the other to provide several proofs and details that have been 
omitted in loc.\ cit.

Our discussion is quite similar in all cases. We begin with describing in detail 
the PEL datum at hand, including the Hodge structure and the resulting 
determinant morphism. We proceed by making explicit the definition of the local 
model and investigate its base-change to $\FF$. We then recall the definition of 
the affine flag variety in terms of lattice chains and prove in detail that it 
can also be described as a suitable quotient of loop groups. We conclude the 
discussion of the local model by embedding it into the affine flag variety and 
prove that the $\Aut(\mc L)$-orbits on $\Mloc(\FF)$ are precisely the Schubert 
cells contained in $\Mloc(\FF)$. We then prove an explicit formula for the 
$p$-rank on a KR stratum.

Using these explicit formulas and the aforementioned combinatorial structure of 
the KR stratification, we prove the following geometric results.
\subsection{Density of the ordinary locus}
It is an interesting question whether the ordinary locus lies dense in $\mc 
A_\FF$. In the case of hyperspecial level structure, it has been studied in 
detail by Wedhorn in \cite{wedhorn}. 

We focus on the case of Iwahori level structure. In the corresponding Siegel 
case $\mc A_I$, the result \cite[Corollaire 4.3]{genestier_ngo} of Ng\^o and 
Genestier answers this question affirmatively.
On the other hand, Stamm obtains in \cite{stamm} a negative answer in the 
two-dimensional Hilbert-Blumenthal case. The following result generalizes these 
two results and explains the general pattern, thereby answering a question by 
M.~Rapoport.
\begin{thm}[Corollary \ref{CorDense}]
    Assume that $\mc B$ is the symplectic PEL datum associated with a totally 
    real extension $F/\QQ$ (see Section \ref{SecPELSym} for a detailed 
    description of $\mc B$). Assume (without loss of generality) that there is 
    only a single prime of $\mc O_F$ dividing $p$. Then the ordinary locus lies 
    dense in $\mc A_\FF$ if and only if $p$ is totally ramified in $\mc O_F$.
\end{thm}
\subsection{Dimension of the \texorpdfstring{$p$}{p}-rank 0 locus}
As mentioned above, Görtz and Yu use \cite[Th\'eor\`eme 4.1]{genestier_ngo} to 
compute the dimension of the $p$-rank 0 locus in $\mc A_I$, see \cite[Theorem 
8.8]{gy2}. By copying their approach and using our formula for the $p$-rank on a 
KR stratum in the split unitary case, we obtain the following result.
\begin{thm}[Theorem \ref{ThmPRank0Locus}]
    Assume that $\mc B$ is the unitary PEL datum of signature $(r,n-r)$ 
    associated with an imaginary quadratic extension of $\QQ$ in which $p$ 
    splits (see Section \ref{SecPELUni3} for a detailed description of $\mc B$).
    Denote by $\mc A^{(0)} \subset \mc A(\FF)$ the subset where the $p$-rank of 
    the underlying abelian variety is equal to $0$. Then
    \begin{equation*}
        \dim \mc A^{(0)}=\min\bigl((r-1)(n-r),r(n-r-1)\bigr).
    \end{equation*}
\end{thm}
\subsection{The Hilbert-Blumenthal case}
As an illustrative example, we look in Section \ref{SecHilbBlum} at the case of 
the
Hilbert-Blumenthal modular varieties. Without any additional work, we obtain the 
following result.
\begin{thm}[Theorem \ref{TheoremHB}]
    Let $g\geq 2$ and let $\mc A$ be the Hilbert-Blumenthal modular variety with 
    $\Gamma_0(p)$-level structure associated with a totally real extension of 
    degree $g$ of $\QQ$. Denote by $\mc A^{(0)} \subset \mc A_{\FF}$ and $\mc 
    A^{(g)} \subset \mc A_{\FF}$ the subsets where the $p$-rank of the 
    underlying abelian variety is equal to $0$ and $g$, respectively. Then
    \begin{equation*}
        \mc A_{\FF}=\mc A^{(0)} \amalg \mc A^{(g)}.  \end{equation*}
    The ordinary locus $\mc A^{(g)}$ is the union of only two KR strata $\mc 
    A_{x_1}$ and $\mc A_{x_2}$. Consequently we have
    \begin{equation*}
        \mc A_{\FF}=\overline{\mc A}_{x_1}\cup \overline{\mc A}_{x_2}\cup \mc 
        A^{(0)}.
    \end{equation*}
    Here $\overline{\mc A}_{x}$ denotes the closure of the KR stratum $\mc 
    A_{x}$ in $\mc A_{\FF}$.

    Each of $\mc A_{\FF}, \overline{\mc A}_{x_1}, \overline{\mc A}_{x_2}$ and 
    $\mc A^{(0)}$
    is equidimensional of dimension $2^g$.

    Furthermore, we have
    \begin{equation*}
        \overline{\mc A}_{x_1}\cap \overline{\mc A}_{x_2}\subset \mc A^{(0)}.
    \end{equation*}
\end{thm}
Taking $g=2$, we recover the result \cite[Theorem 2 (p.\ 408)]{stamm} of Stamm.
\subsection*{Acknowledgments}
It is my pleasure to express my gratitude to my advisor U.~Görtz. He is the one 
who introduced me to this area of mathematics, who suggested the topic of this 
paper and who has provided invaluable support through countless hours of 
stimulating discussions.

I also want to thank T.~Wedhorn for suggesting the point of view taken in 
Section \ref{SecFormula}, namely to obtain a formula for the $p$-rank on a KR 
stratum by looking at the Newton stratification, and M.~Rapoport for helpful 
comments and suggestions.

This work was supported by the SFB/TR45 ``Periods, moduli spaces and arithmetic 
of algebraic varieties'' of the DFG (German Research Foundation).
\subsection*{Notation}
We fix once and for all a rational prime $p\neq 2$ and an algebraic closure 
$\FF$ of $\FF_p$.

Let $n\in\NN_{\geq 1}$.
\begin{itemize}
    \item For elements $x_1,\dots,x_n$ of some set and $k_1,\dots k_n\in\NN$, we 
        denote by $(x_1^{(k_1)},\dots,x_n^{(k_n)})$ the tuple
        \begin{equation*}
            (\underset{k_1\text{-times}}{\underbrace{x_1,\dots,x_1}}, ,\dots, 
            \underset{k_n\text{-times}}{\underbrace{x_n,\dots,x_n}}).
        \end{equation*}
        For a tuple $x\in\ZZ^n$, we denote by $x(i)$ its $i$-th entry.
    \item For an element $w$ of $S_n$, the symmetric group on $n$ letters, we 
        denote by $A_w=(\delta_{iw(j)})_{ij}$ the corresponding permutation 
        matrix.
    \item We write
        \begin{equation*}
            \quad\widetilde{J}_{2n}=
            \begin{pmatrix}
                0&\widetilde{I_n}\\
                -\widetilde{I_n}&0
            \end{pmatrix},\quad\text{where}\quad
            \widetilde{I}_n=\text{anti-diag}(1,\dots,1).
        \end{equation*}
\end{itemize}
Let $R$ be a ring and let $R\to R'$ be an $R$-algebra.
\begin{itemize}
    \item We denote the dual of various objects over $R$ by a superscript 
        $\cdot^\vee=\cdot^{\vee,R}$.
    \item We often denote the base-change from $R$ to $R'$ by a subscript 
    $\cdot_{R'}$.  \item If $G$ is a functor on the category of $R'$-algebras, 
        we denote by $\Res_{R'/R} G$ the functor on the category of $R$-algebras 
        with $(\Res_{R'/R} G)(S)=G(S\otimes_R R')$.
    \item If $F$ is a functor on the category of $R(\!(u)\!)$-algebras (resp.\ 
        $R[\![u]\!]$-algebras), we denote by $\Lf F=\Lf_u F$ (resp.\ $\Lp 
        F=\Lf^+_u F$) the functor on the category of $R$-algebras with $\Lf 
        F(S)=F(S(\!(u)\!))$ (resp.\  $\Lp F(S)=F(S[\![u]\!])$).
    \item For $\lambda\in\ZZ^n$, we write
        $u^\lambda=\diag(u^{\lambda(1)}, \dots, u^{\lambda(n)})\in 
        \GL_n(R(\!(u)\!))$.
\end{itemize}
\section{The general case}\label{SecGeneral}
We assume that the reader is familiar with at least the definitions of 
\cite[3.1-3.27]{rz} and \cite[6.1-6.9]{rz}. The required results 
on orders in semisimple algebras can all be found in Reiner's excellent 
\cite{reiner}.  In Sections \ref{SecPelData} through \ref{SecLocalModelGen} we 
recall from \cite{rz} the general setup of integral models of PEL-type Shimura 
varieties and their local models.
\subsection{PEL data}\label{SecPelData}
A \emph{PEL datum} consists of the following objects.
\begin{enumerate}
    \item A finite-dimensional semisimple $\QQ$-algebra $B$.
    \item A positive\footnote{By this we mean that the involution on $B\otimes 
            \RR$ arising
        from $\inv$ via base-change is a positive involution in the sense of 
        \cite[\textsection 2]{kottwitz_points}.} involution $\inv$ on $B$.
    \item A finitely generated left $B$-module $V$. We assume that $V\neq 0$.
    \item A symplectic form $\pairtd:V\times V\to \QQ$ on the underlying 
        $\QQ$-vector space of $V$, such that for all $v,w\in V$ and all $b\in B$ 
        the relation
        \begin{equation*}
            \pairt{bv}{w}=\pairt{v}{b^\inv w}
        \end{equation*}
        is satisfied.
    \item An element $J\in \End_{B\otimes \RR}(V\otimes \RR)$ with $J^2+1=0$
        such that the bilinear form $\pairt{\cdot}{J\cdot}_\RR:V_\RR\times 
        V_\RR\to \RR$
        is symmetric and positive definite.
\end{enumerate}
We also fix the following data.
\begin{enumerate}[label=(\alph*)]
    \item A $\ZZ$-order $\mc O_B$ in $B$ such that $\mc O_B\otimes \ZZ_p$ is a 
        maximal $\ZZ_p$-order in $B\otimes\QQ_p$. We assume that $\mc O_B\otimes 
        \ZZ_p$ is stable under $\inv$.
    \item A self-dual multichain $\mc L$ of $\mc O_B\otimes \ZZ_p$-lattices in 
        $V\otimes\QQ_p$.
\end{enumerate}
Denote by $G$ the group on the category of $\QQ$-algebras with
\begin{equation*}
    G(R)=\left\{g\in \GL_{B\otimes R}(V\otimes R)\mid \exists c=c(g)\in R^\times
    \left(
    \begin{aligned}
        &\forall x,y\in V\otimes R\\
        &\pairt{gx}{gy}_R=c\pairt{x}{y}_R
    \end{aligned}
    \right)
\right\}.
\end{equation*}
Note that for $g\in G(R)$, the unit $c(g)\in R^\times$ is indeed uniquely 
determined in view of the assumption $V\neq 0$ and the perfectness of $\pairtd$, 
justifying the notation. We also denote by $c:G\to\GG_{m,\QQ}$ the resulting 
morphism.

Let $\Lambda\in \mc L$. We deviate slightly from the notation of \cite{rz} in 
writing $\Lambda^\vee=\{x\in V_{\QQ_p}\mid \pairt{x}{\Lambda}_{\QQ_p} \subset\mc 
\ZZ_p\}$ (in loc.\ cit.\ the notation $\Lambda^\ast$ is used instead).
We denote by $\pairtd_{\Lambda}:\Lambda\times \Lambda^\vee\to \ZZ_p$ the 
restriction of $\pairtd_{\QQ_p}$. It is a perfect pairing and induces an 
isomorphism $\Lambda^\vee\too{\sim} 
\Hom_{\ZZ_p}(\Lambda,\ZZ_p)=\Lambda^{\vee,\ZZ_p}$ of $\mc 
O_B\otimes\ZZ_p$-modules, justifying the notation. For $\Lambda\subset \Lambda'$ 
in $\mc L$ we denote by $\rho_{\Lambda',\Lambda}:\Lambda\to \Lambda'$ the 
inclusion. For $b\in (B\otimes \QQ_p)^\times$ in the normalizer of $\mc 
O_B\otimes\ZZ_p$, denote by $\Lambda^b$ the $\mc O_B\otimes\ZZ_p$-module obtained from $\Lambda$ by restriction 
of scalars with respect to the morphism $\mc O_B\otimes\ZZ_p \to \mc 
O_B\otimes\ZZ_p,\ x\mapsto b^{-1}xb$, and let 
$\vartheta_{\Lambda,b}:\Lambda^b\to b\Lambda$ be the isomorphism given by 
multiplication with $b$. Then $(\Lambda,\rho_{\Lambda',\Lambda}, \vartheta_{\Lambda,b},\pairtd_\Lambda)$ 
is a polarized multichain of $\mc O_B\otimes \ZZ_p$-modules of type $(\mc L)$ 
which, by abuse of notation, we also denote by $\mc L$.

Let $B\otimes\QQ_p=B_1\times \dots\times B_m$ be the decomposition into simple 
factors. It induces a decomposition
\begin{equation}\label{EqOBecomp}
    \mc O_B\otimes\ZZ_p=\mc O_{B_1}\times\dots\times\mc O_{B_m}
\end{equation}
and each $O_{B_i}$ is a maximal $\ZZ_p$-order in $B_i$.

We also get a decomposition $V\otimes\QQ_p=V_1\times \dots\times V_m$ into left 
$B_i$-modules $V_i$.
Denote by $\mc L_i$ the projection of $\mc L$ to $V_i$. It is a chain of $\mc 
O_{B_i}$-lattices in $V_i$. For $\Lambda\in \mc L$ we denote by 
$\Lambda=\Lambda_1\times\dots\times\Lambda_m,\ \Lambda_i\in\mc L_i$ the 
corresponding decomposition.

Denote by $V_{\CC,\pm i}$ the $(\pm i)$-eigenspace of $J_\CC$. Complex 
conjugation induces an isomorphism $V_{\CC,i}\to V_{\CC,-i}$ and consequently
\begin{equation}\label{EqDimOfVi}
    \dim_\CC V_{\CC,i}=\dim_\CC V_{\CC,-i}=\frac{1}{2}\dim_\QQ V.
\end{equation}

Let us quickly recall from \cite[3.23]{rz} the determinant morphism. See 
\cite[2.3]{diss} for a more detailed discussion. Let $R$ be a ring, $A$ a (not 
necessarily commutative) $R$-algebra and let $M$ be a left $A$-module which is 
finite locally free as an $R$-module. Denote by $V=V_A$ the functor on the 
category of $R$-algebras with
$V(S)=A\otimes_R S$. We define a morphism $\det_{M,A}=\det_M:V\to 
\mathbb{A}^1_R$ on $S$-valued points by
\begin{align*}
    \det_M(S):A\otimes_R S\to S,\quad x\mapsto \det_{S}(M_S\too{x \cdot } M_S).
\end{align*}
For $x\in A$ denote by $\chi_R(x|M)$ the characteristic polynomial of 
$M\too{x \cdot } M$ over $R$.
Below we will phrase the determinant condition using characteristic polynomials.  
This is warranted by the following statement.
\begin{prop}
    Let $A$ be a (not necessarily commutative) $R$-algebra and let $M$ and $N$ 
    be $A$-modules which are finite locally free over $R$. Let $A_0\subset A$ be 
    a generating set of $A$ as an $R$-module. Then $\det_M=\det_N$ if and only 
    if for all $a\in A_0$ we have $\chi_R(a|M)=\chi_R(a|N)$ .
\end{prop}
\begin{proof}
    Clear by the existence of Amitsur's formula \cite[Theorem A]{amitsur}, 
    which, in a suitable sense,
    expresses the characteristic polynomial of a linear combination of 
    endomorphisms in terms of the characteristic polynomials of the summands.
\end{proof}
As $V_{\CC,-i}$ is a $B\otimes\CC$-module we get a morphism 
$\det_{V_{\CC,-i}}:V_{B\otimes\CC}\to \AF^1_\CC$. Consider the reflex field 
$E=\QQ(\tr_\CC(b\otimes 1|V_\CC);\;b\in B)$. The morphism $\det_{V_{\CC,-i}}$ is 
defined over $\mc O_E$. Fix a place $\mc Q$ of $\mc O_E$ lying over $p$.
\subsection{Polarized \texorpdfstring{$\mc L$}{L}-sets of abelian varieties}\label{SecSelfDual}
\begin{defn}\label{DefnLSet}
    Let $R$ be an $\mc O_{E_\mc Q}$-algebra. A
    \emph{polarized $\mc L$-set of abelian varieties over $R$} is a pair 
    $(A,\lambda)$, where $A=(A_\Lambda,\varrho_{\Lambda',\Lambda})$ is an $\mc 
    L$-set of abelian varieties over $R$ in the sense of \cite[Definition 
    6.5]{rz}, and where $\lambda:A\to A^\vee$ is a principal polarization in the 
    sense of \cite[Definition 6.6]{rz}.  We say that $(A,\lambda)$ is \emph{of 
        determinant $\det_{V_{\CC,-i}}$} if
    for all $\Lambda\in \mc L$ we have an equality
    \begin{equation*}
        \det_{\Lie A_\Lambda}=\det_{V_{\CC,-i}}\otimes_{\mc O_E} R
    \end{equation*}
    of morphisms $V_{\mc O_B\otimes R}\to \AF^1_R$.

    We denote by $\mc A$ the functor on the category of $\mc O_{E_\mc 
    Q}$-algebras with $\mc A(R)$ the set of isomorphism classes of polarized
    $\mc L$-sets of abelian varieties of determinant $\det_{V_{\CC,-i}}$ over 
    $R$.
\end{defn}
\begin{remark}
    After additionally imposing a suitable level structure away from $p$ in the 
    definition of $\mc A$, we may (and will) assume that $\mc A$ is representable by a
    quasi-projective scheme over $\mc O_{E_\mc Q}$, see \cite[Definition 
    6.9]{rz} and the discussion following it. We have decided not to include 
    this level structure in our notation as it is of no importance for the 
    question of the $p$-rank on a KR stratum.
\end{remark}
Let $R$ be a ring. For an abelian scheme $A/R$, we denote by $H^{dR}_1(A/R)$ the 
first de Rham cohomology of $A$. It is part of a canonical short exact sequence
\begin{equation}\label{EqHodgedeRham}
    0\to \omega_{A^\vee}\to H_1^{dR}(A/R)\to \Lie(A)\to 0,
\end{equation}
where $\omega_{A^\vee}\subset H_1^{dR}(A/R)$ denotes the Hodge filtration.
All terms of \eqref{EqHodgedeRham} are finite locally free $R$-modules. We 
have $\rk_R H_1^{dR}(A/R)=2\dim_R A$ and $\rk_R \Lie(A)=\rk_R\omega_{A^\vee}=\dim_R A$.
\begin{defn}\label{DefnIntermediate}
    We denote by $\widetilde{\mc A}$ the functor on the category of $\mc 
    O_{E_\mc Q}$-algebras with $\widetilde{\mc A}(R)$ the set of isomorphism 
    classes of pairs $(A,\gamma)$, where $A$ is a polarized $\mc L$-set of 
    abelian varieties of determinant $\det_{V_{\CC,-i}}$ over $R$ and 
    \begin{equation*}
        \gamma:H^{dR}_{1}(A)\too{\sim} \mc L\otimes R
    \end{equation*}
    is an isomorphism of polarized multichains of $\mc O_B\otimes R$-modules of 
    type $(\mc L)$. 
    
    Denote by $\widetilde{\varphi}:\widetilde{\mc A}\to \mc A$ the morphism 
    given on $R$-valued points by $\widetilde{\mc A}(R)\to \mc A(R),\ 
    (A,\gamma)\mapsto A$.
\end{defn}
$\Aut(\mc L)$ acts from the left on $\widetilde{\mc A}$
via $g\cdot (A,\gamma)=(A,g\circ \gamma)$ and $\widetilde{\varphi}$ is 
invariant for this action.
\begin{prop}[{\cite[Theorem 2.2]{pappas_arithmetic}}]\label{PropTorsor}
    The morphism $\widetilde{\varphi}:\widetilde{\mc A}\to \mc A$ is an
    $\Aut(\mc L)$-torsor for the \'etale topology. In particular 
    $\widetilde{\varphi}(\FF)$ is an $\Aut(\mc L)(\FF)$-torsor in the 
    set-theoretic sense.
\end{prop}
\subsection{The local model diagram and the KR stratification}\label{SecLocalModelGen}
We will use the following obvious variant of \cite[Definition 3.27]{rz}.
\begin{defn}\label{DefnLocalModelGen}
    The local model $\Mloc$ is the functor on the category of $\mc O_{E_\mc 
    Q}$-algebras with $\Mloc(R)$ the set of tuples $(t_\Lambda)_{\Lambda\in\mc 
    L}$ of $\mc O_B\otimes R$-submodules $t_\Lambda \subset \Lambda_R$ 
    satisfying the following conditions for all $\Lambda \subset \Lambda'$ in 
    $\mc L$.
    \begin{enumerate}
        \item\label{DefnLocalModelGen-Functoriality} We have 
            $\rho_{\Lambda',\Lambda,R}(t_\Lambda)\subset t_{\Lambda'}$, so that 
            we get a commutative diagram
            \begin{equation*}
                \xymatrix{
                    t_\Lambda\ar[r]\ar[d]&t_{\Lambda'}\ar[d]\\
                    \Lambda_{R}\ar[r]^{\rho_{\Lambda',\Lambda,R}}&\Lambda'_R.
                }
            \end{equation*}
        \item\label{DefnLocalModelGen-Projectivity} The quotient 
            $\Lambda_{R}/t_\Lambda$ is a finite locally free $R$-module.
        \item\label{DefnLocalModelGen-Determinant} We have an equality
            \begin{equation*}
                \det_{\Lambda_R/t_\Lambda}=\det_{V_{\CC,-i}}\otimes_{\mc O_E}R
            \end{equation*}
            of morphisms $V_{\mc O_{B}\otimes R}\to \AF^1_R$.
        \item\label{DefnLocalModelGen-DualityCondition} Under the pairing 
            $\pairtd_{\Lambda,R}:\Lambda_{R}\times \Lambda^\vee_{R}\to R$, the
            submodules $t_\Lambda$ and $t_{\Lambda^\vee}$ pair to zero.
        \item\label{DefnLocalModelGen-Periodicity} We have
            $\vartheta_{\Lambda,b,R}(t_\Lambda)= t_{b\Lambda}$
            for all $b\in (B\otimes \QQ_p)^\times$ that normalize $\mc 
            O_B\otimes\ZZ_p$.
    \end{enumerate}
\end{defn}
\begin{remark}
    We have added the natural condition 
    \ref{DefnLocalModelGen}(\ref{DefnLocalModelGen-Periodicity}), which
    seems to be missing from \cite[Definition 3.27]{rz}.
\end{remark}
\begin{remark}\label{RemLocalModelIsScheme}
    By definition, $\Mloc$ is a closed subscheme of a finite product of 
    Grassmannians. In particular $\Mloc$ is a projective scheme over $\Spec 
    \mc O_{E_\mc Q}$.
\end{remark}
\begin{remark}\label{RemarkDecompLocalModel}
    Let $R$ be an $\mc O_{E_\mc Q}$-algebra and $(t_\Lambda)_{\Lambda}\in 
    \Mloc(R)$. For $\Lambda\in\mc L$ the decomposition 
    $\Lambda=\Lambda_1\times\dots\times \Lambda_m$ induces a decomposition 
    $t_{\Lambda}=t_{\Lambda,1}\times\dots\times t_{\Lambda,m}$ into $\mc 
    O_{B_i}\otimes R$-submodules $t_{\Lambda,i}\subset \Lambda_{i,R}$. Let 
    $i\in\{1,\dots,m\}$ and let $\Lambda\subset \Lambda'$ in $\mc L$ with
    $\Lambda_i=\Lambda'_i$. From condition 
    \ref{DefnLocalModelGen}(\ref{DefnLocalModelGen-Functoriality})  we 
    conclude that $t_{\Lambda,i}\subset t_{\Lambda',i}$. From condition 
    \ref{DefnLocalModelGen}(\ref{DefnLocalModelGen-Determinant}) we conclude 
    that $t_{\Lambda,i}$ and $t_{\Lambda',i}$ both have the same rank over 
    $R$.  Thus $t_{\Lambda,i}=t_{\Lambda',i}$ in view of 
    \ref{DefnLocalModelGen}(\ref{DefnLocalModelGen-Projectivity}).  
    Consequently we may unambiguously write $t_{\Lambda_i}=t_{\Lambda,i}$. 

    We conclude that the family $(t_{\Lambda})_{\Lambda\in\mc L}$ is determined by the tuple of 
    families 
    \begin{equation*}
        \bigl((t_{\Lambda_1})_{\Lambda_1\in\mc 
        L_1},\dots,(t_{\Lambda_m})_{\Lambda_m\in \mc L_m}\bigr).
    \end{equation*}
    All conditions of Definition \ref{DefnLocalModelGen} with the exception of condition 
    (\ref{DefnLocalModelGen-DualityCondition})
    translate into independent conditions on the individual $(t_{\Lambda_i})$.
\end{remark}
\begin{defn}
    Denote by $\widetilde{\psi}:\widetilde{\mc A}\to \Mloc$ the morphism given on $R$-valued 
    points by
    \begin{align*}
        \widetilde{\mc A}(R)&\to \Mloc(R),\\
        ((A_\Lambda),(\gamma_\Lambda))&\mapsto 
        (\gamma_\Lambda(\omega_{A_\Lambda^\vee}))_\Lambda.
    \end{align*}
\end{defn}
$\Aut(\mc L)$ acts from the left on $\Mloc$ via $(\varphi_\Lambda)\cdot 
(t_{\Lambda})=(\varphi_\Lambda(t_\Lambda))$ and $\widetilde{\psi}$ is equivariant for this action. 
\begin{defn}
    The diagram
    \begin{equation*}
        \xymatrix{
            &\ar[dl]_{\widetilde{\varphi}}\widetilde{\mc A}\ar[dr]^{\widetilde{\psi}}&\\
            \mc A&&\Mloc
        }
    \end{equation*}
    of $\mc O_{E_\mc Q}$-schemes is called the \emph{local model diagram}.
\end{defn}
\begin{remark}[{\cite[Chapter 3]{rz}, cf.\ \cite[Theorem 2.2]{pappas_arithmetic}}]\label{RemLocalModelDiag}
    The morphisms $\widetilde{\varphi}$ and $\widetilde{\psi}$ are smooth of the 
    same relative dimension. There is, \'etale locally on $\mc A$, 
    a section $s:\mc A\to \widetilde{\mc A}$ of $\widetilde{\varphi}$, such that 
    the composition $\widetilde{\psi}\circ s:\mc A\to \Mloc$ is \'etale.
\end{remark}
Consider the decomposition
\begin{equation*}
    \Mloc(\FF)=\coprod_{x\in \Aut(\mc L)(\FF)\backslash \Mloc(\FF)}\Mloc_x
\end{equation*}
into $\Aut(\mc L)(\FF)$-orbits.  \begin{remark}\label{RemKRStrataLocallyClosed}
    Let $x\in \Aut(\mc L)(\FF)\backslash \Mloc(\FF)$. The subset
    $\Mloc_x \subset \Mloc(\FF)$ is locally closed, and we equip it with the 
    reduced scheme structure. By \cite[Theorem 3.16]{rz} the $\FF$-group 
    $\Aut(\mc L)_\FF$ is smooth and affine. Thus $\Mloc_x$ is a smooth 
    quasi-projective variety over $\FF$.
\end{remark}
For $x\in \Aut(\mc L)(\FF)\backslash \Mloc(\FF)$, we define $\widetilde{\mc 
A}_x=\widetilde{\psi}(\FF)^{-1}(\Mloc_x)$ and $\mc 
A_x=\widetilde{\varphi}(\FF)(\widetilde{\mc A}_x)$. It follows from Proposition 
\ref{PropTorsor} that the $\mc A_x$ are pairwise disjoint and cover $\mc A(\FF)$
as $x$ runs through $\Aut(\mc L)(\FF)\backslash \Mloc(\FF)$.
\begin{defn}\label{DefnKR}
    The decomposition \begin{equation*}
        \mc A(\FF)=\coprod_{x\in \Aut(\mc L)(\FF)\backslash \Mloc(\FF)}\mc A_x
    \end{equation*}
    is called the \emph{Kottwitz-Rapoport} (or KR) \emph{stratification} on $\mc A$.
\end{defn}
\begin{remark}\label{RemarkKRasScheme}
    By Remark \ref{RemLocalModelDiag} there is, \'etale locally on $\mc A_\FF$, 
    an \'etale morphism $\beta:\mc A_\FF\to \Mloc_\FF$ with $\mc 
    A_{x}=\beta^{-1}(\Mloc_x)$ for $x\in \Aut(\mc L)(\FF)\backslash \Mloc(\FF)$. Hence the subset $\mc A_{x} \subset \mc A(\FF)$ 
    is locally closed, and after equipping it with the reduced scheme structure, 
    $\mc A_{x}$ is a smooth variety over $\FF$.
\end{remark}
\subsection{The \texorpdfstring{$p$}{p}-rank on a KR stratum}\label{SecPRankGen}
\begin{lem}\label{LemModulesOverSimpleAlgebras}
    Let $A/\QQ_p$ be a finite simple algebra and let $\mc O_A\subset A$ be a 
    maximal $\ZZ_p$-order. Then all simple left $\mc O_A$-modules and all simple 
    right $\mc O_A$-modules have the same finite cardinality.
\end{lem}
\begin{proof}
    By \cite[Theorem 17.3]{reiner} there exist a finite division algebra 
    $D/\QQ_p$, an integer $n\in \NN$ and an isomorphism $A\simeq M^{n\times 
    n}(D)$ inducing an isomorphism $\mc O_A\simeq M^{n\times n}(\mc O_D)$. Here 
    $\mc O_D\subset D$ denotes the unique maximal $\ZZ_p$-order, see 
    \cite[Theorem 12.8]{reiner}.  

    Denote by $\mf p\subset \mc O_D$ the unique maximal ideal and by $k=\mc 
    O_D/\mf p$ the corresponding residue field, see \cite[Theorem 13.2]{reiner}. 
    By loc.\ cit.\ every simple left (resp.\ right) $\mc O_D$-module is 
    isomorphic to $k$. Hence by Morita equivalence (see \cite[\S\S 16]{reiner}) 
    every simple left (resp.\ right) $M^{n\times n}(\mc O_D)$-module is 
    isomorphic to $k^n=M^{n\times 1}(k)$ (resp.\ $k^n=M^{1\times n}(k)$).
\end{proof}
\begin{defn}\label{DefnCompleteness}
    The multichain $\mc L$ is called \emph{complete} if for any two neighbors 
    $\Lambda\subset\Lambda'$ in $\mc L$, the quotient $\Lambda'/\Lambda$ is a 
    simple $\mc O_B\otimes \ZZ_p$-module.
\end{defn}
For a finite commutative group scheme $G/\FF$, we denote by $G^{e,u}$ the 
\'etale unipotent and by $G^{i,m}$ the infinitesimal multiplicative part of $G$. 
Let $\rk_{e,u}(G):=\rk(G^{e,u})$ and $\rk_{i,m}(G):=\rk(G^{i,m})$.
\begin{lem}\label{LemKernelVeryBoringGen}
    Assume that $\mc L$ is complete. Let 
    $(A_\Lambda,\varrho_{\Lambda',\Lambda})$ be an $\mc L$-set of abelian 
    varieties over $\FF$ and let $\Lambda\subset\Lambda'$ be neighbors in $\mc 
    L$.  Then $K=\ker \varrho_{\Lambda',\Lambda}$ is either \'etale unipotent or 
    infinitesimal multiplicative or infinitesimal unipotent.
\end{lem}
\begin{proof}
    The decomposition \eqref{EqOBecomp} induces a decomposition 
    $K=K_1\times\dots\times K_m$ into finite locally free group schemes $K_i$ 
    with actions $\mc O_{B_i}\to \End(K_i)$.

    As $\Lambda$ and $\Lambda'$ are neighbors, there is a unique 
    $i_0\in\{1,\dots,m\}$ with $\Lambda_{i_0}\subsetneq \Lambda'_{i_0}$, and as 
    $\mc L$ is complete we know that $\Lambda'_{i_0}/\Lambda_{i_0}$ is a
    simple left $\mc O_{B_{i_0}}$-module. Let 
    $N=|\Lambda'_{i_0}/\Lambda_{i_0}|$. By the definition of an $\mc L$-set of 
    abelian varieties we know that $K_i=0$ for $i\neq i_0$ and that $G:=K_{i_0}$ 
    has rank $N$ over $\FF$.

    The action $\mc O_{B_{i_0}}\to \End G$ induces on $G(\FF)$ the structure of 
    a left $O_{B_{i_0}}$-module, and as $|G(\FF)|\leq \rk G= N$, Lemma 
    \ref{LemModulesOverSimpleAlgebras} implies $|G(\FF)|\in\{0,N\}$. As
    $|G(\FF)|=\rk(G^{e,u})$, we conclude that $G^{e,u}\in \{0,G\}$. 

    Denote by $D(G)$ the Cartier dual of $G$. We also obtain on $D(G)(\FF)$ the 
    structure of a right $O_{B_{i_0}}$-module and we analogously obtain that 
    $D(G)^{e,u}\in \{0,D(G)\}$. As $D(G)^{e,u}=D(G^{i,m})$, it follows that 
    $G^{i,m}\in\{0,G\}$.
\end{proof}
\begin{defn}\label{DefnPRank}
    Let $A/\FF$ be an abelian variety. Denote by $[p]_A:A\to A$ the 
    multiplication by $p$ and by $A[p]$ the kernel of $[p]_A$. The integer 
    $\log_p \rk_{e,u}A[p]$ is called the \emph{$p$-rank} of $A$.
\end{defn}
\begin{prop}\label{ProppRankFirstVersionGen}
    Assume that $\mc L$ is complete. Let 
    $(A_\Lambda,\varrho_{\Lambda',\Lambda})$ be an $\mc L$-set of abelian 
    varieties over $\FF$.
    Let $\Lambda\in\mc L$ and choose a sequence 
    $p^{-1}\Lambda=\Lambda^{(0)}\supsetneq \Lambda^{(1)}\supsetneq 
    \dots\supsetneq \Lambda^{(k)}=\Lambda$ of neighbors 
    $\Lambda^{(j-1)}\supsetneq \Lambda^{(j)}$ in $\mc L$. Define 
    \begin{equation*}
        J_{e,u}=\{j\in\{1,\dots,k\}\mid \ker \varrho_{ \Lambda^{(j-1)}, 
        \Lambda^{(j)}}\text{ is \'etale}\}.
    \end{equation*}
    The $p$-rank of $A_\Lambda$ is equal to
    \begin{equation*}
        \sum_{j\in J_{e,u}} \log_p |\Lambda^{(j-1)}/\Lambda^{(j)}|.
    \end{equation*}
\end{prop}
\begin{proof}
    By the definition of an $\mc L$-set of abelian varieties, there is a 
    periodicity isomorphism 
    $\theta_{p^{-1}\Lambda,p}:A_{p^{-1}\Lambda}\too{\sim}A_{\Lambda}$ such that
    \begin{equation*}
        [p]_{A_\Lambda}=\theta_{p^{-1}\Lambda,p}\circ \prod_{j=1}^k \varrho_{ 
            \Lambda^{(j-1)}, \Lambda^{(j)}}.
    \end{equation*}
    This implies
    \begin{gather*}
        \rk_{e,u}A_\Lambda[p]=\prod_{j=1}^k \rk_{e,u}\ker 
        \varrho_{\Lambda^{(j-1)}, \Lambda^{(j)}}.
    \end{gather*}
    Lemma \ref{LemKernelVeryBoringGen} and the definition of an $\mc L$-set of 
    abelian varieties yield
    \begin{gather*}
        \rk_{e,u}\ker \varrho_{ \Lambda^{(j-1)}, \Lambda^{(j)}}=
        \begin{cases}
            |\Lambda^{(j-1)}/\Lambda^{(j)}|&\text{if }j\in J_{e,u},\\
            1&\text{otherwise.}
        \end{cases}
    \end{gather*}
\end{proof}
\begin{prop}\label{PropEquivCondGen}
    Let $A=(A_\Lambda,\varrho_{\Lambda',\Lambda})\in\mc A(\FF)$, choose a lift 
    $\widetilde{A}\in \widetilde{\mc A}(\FF)$ of $A$ under
    $\widetilde{\varphi}(\FF)$ and let $(t_\Lambda)=\widetilde{\psi}(\FF)(\widetilde{A})\in 
    \Mloc_x$.  Let $\Lambda\subset \Lambda'$ in $\mc L$.
    Then
    \begin{equation}\label{EqEquivCondGen}
        \begin{aligned}
            &\phantom{{}\Leftrightarrow{}}\ker \varrho_{ \Lambda', 
            \Lambda}\text{ is multiplicative}\\
            &\Leftrightarrow \rho_{\Lambda',\Lambda,\FF}(t_\Lambda)=t_{\Lambda'}
        \end{aligned}
    \end{equation}
    and
    \begin{equation}
        \label{EqEquivCondGenII}
        \begin{aligned}
            &\phantom{{}\Leftrightarrow{}}\ker \varrho_{ \Lambda', 
            \Lambda}\text{ is \'etale}\\
            &\Leftrightarrow \Lambda'_\FF=\im 
            \rho_{\Lambda',\Lambda,\FF}+t_{\Lambda'}.
        \end{aligned}
    \end{equation}
\end{prop}
\begin{proof}
    In view of the definition of $\widetilde{\psi}$, the stated equivalences 
    amount to well-known characterizations of the respective conditions on $\ker 
    \varrho_{ \Lambda', \Lambda}$ in terms of the Hodge filtration
    inside the de Rham cohomology.
\end{proof}
\begin{cor}\label{CorEquivEtale}
    Let $x\in \Aut(\mc L)(\FF)\backslash \Mloc(\FF)$ and
    $(A_\Lambda,\varrho_{\Lambda',\Lambda}),(A'_\Lambda,\varrho'_{\Lambda',\Lambda})\in 
    \mc A_x$. Let $\Lambda\subset \Lambda'$ in $\mc L$.  Then $\ker 
    \varrho_{\Lambda', \Lambda}$ is \'etale if and only if $\ker 
    \varrho'_{\Lambda', \Lambda}$ is \'etale.
\end{cor}
\begin{proof}
    For $(t_\Lambda)\in \Mloc(\FF)$, the condition 
    $\Lambda'_\FF=\im \rho_{\Lambda',\Lambda,\FF}+t_{\Lambda'}$ is clearly invariant under the 
    $\Aut(\mc L)(\FF)$-action on $\Mloc(\FF)$. The claim therefore follows from 
    \eqref{EqEquivCondGenII}.
\end{proof}
\begin{thm}\label{ThmPRankConstant}
    Assume that $\mc L$ is complete. Let $x\in \Aut(\mc L)(\FF)\backslash 
    \Mloc(\FF)$ and
    $(A_\Lambda,\varrho_{\Lambda',\Lambda}),(A'_\Lambda,\varrho'_{\Lambda',\Lambda})\in 
    \mc A_x$. Let $\Lambda,\Lambda'\in\mc L$. Then the $p$-ranks of $A_\Lambda$ 
    and $A'_{\Lambda'}$ coincide. In other words, the $p$-rank is constant on a 
    KR stratum.
\end{thm}
\begin{proof}
    The $p$-rank of an abelian variety is an isogeny invariant by \cite[p.\ 
    147]{mav}, so that it suffices to treat the case $\Lambda=\Lambda'$.  The 
    statement then follows from
    Proposition \ref{ProppRankFirstVersionGen} and Corollary 
    \ref{CorEquivEtale}.
\end{proof}
\subsection{A formula for the \texorpdfstring{$p$}{p}-rank on a KR stratum}\label{SecFormula}
Denote by $K$ the completion of the maximal unramified extension of $\QQ_p$ and 
by $\mc O_K$ the valuation ring of $K$. We identify the residue field of 
$K$ with $\FF$. We denote by $\sigma$ the Frobenius automorphism on $K$, 
inducing the usual Frobenius $\FF\to\FF,x\mapsto x^p$ on the residue field.  
By abuse of notation we also denote by $\sigma$ the morphism 
$G(\sigma):G(K)\to G(K)$.
\subsubsection{A \texorpdfstring{$p$}{p}-divisible group analogue}
Let $Y/\FF$ be a $p$-divisible group.  We denote by $\DD(Y)$ the covariant 
Dieudonn\'e module of $Y$, see for example \cite{demazure}. It is a free $\mc 
O_K$-module, equipped with a $\sigma$-linear endomorphism $F_\DD$ and a 
$\sigma^{-1}$-linear endomorphism $V_\DD$, satisfying $F_\DD\circ 
V_\DD=F_\DD\circ V_\DD=p$. Let $Y,Y'$ be $p$-divisible groups over $\FF$ and let 
$\lambda:Y\to (Y')^\vee$ be a morphism.  It induces a pairing 
$\pairtd_\lambda:\DD(Y)\times\DD(Y')\to \mc O_K$ satisfying
\begin{equation}\label{EqFVandPair}
    \pairt{F_\DD x}{y}_\lambda=\sigma\pairt{x}{V_\DD y}_\lambda,\ 
    x\in\DD(Y),y\in\DD(Y').
\end{equation}
We denote by $(\res{\DD}(Y),F_{\res{\DD}},V_{\res{\DD}})$ the reduction of 
$\DD(Y), F_\DD, V_\DD)$ modulo $p$.

If $A/\FF$ is an abelian variety, let $\DD(A)=\DD(A[p^\infty])$.
\begin{prop}[{\cite[Theorem 5.11]{Oda}}]\label{PropOda}\leavevmode
    \begin{enumerate}
        \item Let $A/\FF$ be an abelian variety. There is a canonical isomorphism 
            \begin{equation}\label{EqOda}
                \iota=\iota_A:\res{\DD}(A) \too{\sim} H^{dR}_1(A), 
            \end{equation}
            inducing an isomorphism of short exact sequences
            \begin{equation}\label{EqDiagOda}
                \begin{gathered}
                    \xymatrix{
                        0\ar[r]& \im V_{\res{\DD}}\ar[r]\ar[d]^\simeq& \res{\DD}(A)\ar[r]\ar[d]^\simeq_{\iota}& \Lie(A[p^\infty])\ar[d]^\simeq\ar[r]& 0\\
                        0\ar[r]& \omega_{A^\vee}\ar[r]& H_1^{dR}(A)\ar[r]& \Lie(A)\ar[r] &0.
                    }
                \end{gathered}
            \end{equation}
        \item Let $A=(A_\Lambda)\in\mc A(\FF)$. For $\Lambda\in\mc L$ let 
            $\iota_\Lambda=\iota_{A_\Lambda}:\res{\DD}(A_\Lambda)\to 
            H^{dR}_1(A_\Lambda)$ be the isomorphism from \eqref{EqOda}.
            The resulting morphism
            \begin{equation}\label{EqOdaChain}
                \iota=(\iota_\Lambda)_\Lambda:(\res{\DD}(A_\Lambda))_\Lambda\too{\sim} (H^{dR}_{1}(A_\Lambda))_\Lambda
            \end{equation}
            is an isomorphism of polarized multichains of $\mc O_B\otimes \FF$-modules of type $(\mc L)$.
    \end{enumerate}
\end{prop}
In complete analogy with \cite[Definition 6.5]{rz} and Definition \ref{DefnLSet} we have the notion of 
a polarized $\mc L$-set $Y=(Y_\Lambda, \varrho_{\Lambda',\Lambda},\lambda_\Lambda)$ of $p$-divisible 
groups of determinant $\det_{V_{\CC,-i}}$ over $\FF$. We further obtain a 
set $\mc A_{\pd}(\FF)$ of isomorphism classes of such $Y$. In analogy with 
Definition \ref{DefnIntermediate}, denote by $\widetilde{\mc A}_\pd(\FF)$ the set 
of isomorphism classes of pairs $(Y,\gamma)$, where $Y$ is as above and $\gamma:\res{\DD}(Y)\too{\sim} \mc L\otimes \FF$
is an isomorphism of polarized multichains of $\mc O_B\otimes \FF$-modules of 
type $(\mc L)$. We have the canonical map
$\widetilde{\varphi}_\pd(\FF):\widetilde{\mc A}_\pd(\FF)\to \mc A_\pd(\FF),\ (Y,\gamma)\mapsto Y$, and the morphism
$\widetilde{\psi}_\pd(\FF):\widetilde{\mc A}_\pd(\FF)\to \Mloc(\FF),\ (Y,\gamma)\mapsto \gamma(\im V_{\res{\DD}})$.
\begin{equation*}
    \xymatrix{
        &\ar[dl]_{\widetilde{\varphi}_\pd}\widetilde{\mc A}_\pd\ar[dr]^{\widetilde{\psi}_\pd}&\\
        \mc A_\pd&&\Mloc
    }
\end{equation*}
By Proposition \ref{PropOda} we obtain maps $\delta:\mc A(\FF)\to \mc A_\pd(\FF),\ A\mapsto A[p^\infty]$ and 
$\widetilde{\delta}:\widetilde{\mc A}(\FF)\to \widetilde{\mc A}_\pd(\FF),\ (A,\gamma)\mapsto 
(A[p^\infty],\gamma\circ \iota)$, and the following diagrams commute.
\begin{equation*}
    \xymatrix{
        \widetilde{\mc A}(\FF)\ar[d]_{\widetilde{\varphi}(\FF)}\ar[r]^-{\widetilde{\delta}}& \widetilde{\mc A}_\pd(\FF)\ar[d]^{\widetilde{\varphi}_\pd(\FF)}\\
        \mc A(\FF)\ar[r]^-\delta&\mc A_\pd(\FF),
    }\quad
    \xymatrix{
        \widetilde{\mc A}(\FF)\ar[d]_{\widetilde{\psi}(\FF)}\ar[r]^-{\widetilde{\delta}}& \widetilde{\mc A}_\pd(\FF)\ar[d]^{\widetilde{\psi}_\pd(\FF)}\\
        \Mloc(\FF)\ar@{=}[r]&\Mloc(\FF).
    }
\end{equation*}
In absolute analogy with Definition \ref{DefnKR}, we obtain a decomposition 
\begin{equation*}
    \mc A_\pd(\FF)=\coprod_{x\in \Aut(\mc L)(\FF)\backslash \Mloc(\FF)}\mc A_{\pd,x}, 
\end{equation*}
which we call the \emph{KR stratification} on $\mc A_\pd(\FF)$. For $x\in \Aut(\mc 
L)(\FF)\backslash \Mloc(\FF)$ we have $\mc A_x=\delta^{-1}(\mc A_{\pd,x})$. We 
define as in Definition \ref{DefnPRank} the $p$-rank of a $p$-divisible group 
over $\FF$. The proof of Theorem \ref{ThmPRankConstant} then carries over without any 
changes to show the following statement.
\begin{thm}\label{ThmPRankDivisible}
    Assume that $\mc L$ is complete. Then the $p$-rank is constant on a KR 
    stratum in $\mc A_\pd(\FF)$.
\end{thm}
\subsubsection{The map \texorpdfstring{$\alpha:\mc A_\pd(\FF)\to B_I(G)$}{alpha}}\label{SecAlpha}
\begin{lem}\label{LemIsomOverK}
    Let $\mc M$ and $\mc M'$ be polarized multichains of $\mc O_B\otimes\mc 
    O_K$-modules of type $(\mc L)$. Then the canonical map $\Isom(\mc M,\mc 
    M')(\mc O_K)\to \Isom(\mc M,\mc M')(\FF)$ is surjective. In particular 
    $\mc M$ and $\mc M'$ are isomorphic.
\end{lem}
\begin{proof}
    Let $\mc I=\Isom(\mc M,\mc M')$. Clearly $\mc I$ is representable by an 
    affine scheme over $\mc O_K$. By \cite[Theorem 3.16]{rz} we know that 
    the base-change $\mc I\otimes_{\mc O_K} \mc O_K/p^n$  is in particular 
    formally smooth over $\mc O_K/p^n$ for every $n\in\NN$. This easily 
    implies the surjectivity of $\mc I(\mc O_K)\to \mc I(\FF)$ in view of $\mc I(\mc O_K)=\lim_{n\in\NN}\mc I(\mc 
    O_K/p^n)$. The second claim follows, as $\mc I(\FF)\neq \emptyset$
    by loc.\ cit.
\end{proof}

For $g\in G(K)$, we denote by $g\cdot (\mc L\otimes\mc O_K)$ the tuple 
$(g(\Lambda\otimes\mc O_K))_{\Lambda\in\mc L}$ of $\mc O_B\otimes\mc 
O_K$-submodules of $V\otimes K$. Let $I=\{g\in G(K)\mid g\cdot (\mc L\otimes\mc 
O_K)=\mc L\otimes\mc O_K\}=\{g\in G(K)\mid \forall \Lambda\in\mc L:\ g(\Lambda\otimes\mc 
O_K)=\Lambda\otimes\mc O_K\}$ and $I_0=\{g\in I\mid c(g)=1\}$.  
\begin{lem}\label{LemIvsI0}
    We have $I=\mc O_K^\times I_0$. In particular, any $g\in I$ satisfies 
    $c(g)\in \mc O_K^\times$.
\end{lem}
\begin{proof}
    Let $g\in I$ and  $\Lambda\in \mc L$. Then $\Lambda$ is in particular an 
    $\mc O_K$-lattice in the $K$-vector space $V_K$ and the fact that $g$ 
    restricts to an automorphism of $\Lambda$ implies that $\det(g)\in \mc 
    O_K^\times$.  The equation $\det(g)^2=c(g)^{\dim_\QQ V}$ then yields 
    $c(g)\in \mc O_K^\times$. As $\mc O_K$ is strictly Henselian of residue 
    characteristic different from $2$, there is an $x\in \mc O_K^\times$ with 
    $x^2=c(g)$. Then $x^{-1}g\in I_0$, as desired.
\end{proof}
\begin{lem}\label{LemIvsIprime}
    Let $g\in I_0$. Then $g$ restricts to an
    automorphism $g_\Lambda:\Lambda\otimes\mc O_K\to \Lambda\otimes\mc O_K$ for 
    each $\Lambda\in \mc L$. The assignment $g\mapsto (g_\Lambda)_\Lambda$ 
    defines an isomorphism $I_0\too{\sim} \Aut(\mc L)(\mc O_K)$.
\end{lem}
\begin{proof}
    If $(\varphi_\Lambda)\in \Aut(\mc L)(\mc O_K)$, then 
    $\varphi_\Lambda\otimes_{\mc O_K}K$ is an element of $G(K)$ which is 
    independent of $\Lambda$. This provides an inverse to the map in question.
\end{proof}
\begin{prop}[{\cite[3.23]{rz}}]\label{PropDieudonneYieldsPolChainGen}
    Let $Y=(Y_\Lambda,\varrho_{\Lambda',\Lambda},\lambda_\Lambda)\in\mc 
    A_\pd(\FF)$.
    For $\Lambda\in \mc L$ let $\mc E_\Lambda:\DD(Y_\Lambda)\times 
    \DD(Y_{\Lambda^\vee})\to \mc O_K$ be the pairing induced by 
    $\lambda_\Lambda$.  Then $(\DD(Y_\Lambda))_\Lambda$, equipped with the 
    pairings $(\mc E_\Lambda)_\Lambda$, is a polarized multichain of $\mc 
    O_B\otimes \mc O_K$-modules of type $(\mc L)$.
\end{prop}
Define equivalence relations $\sim$ and $\sim_I$ on $G(K)$ by \begin{gather*}
    x\sim y:\Leftrightarrow \exists g\in G(K):\,y=gx\sigma(g)^{-1},\\
    x\sim_I y:\Leftrightarrow \exists i\in I:\,y=ix\sigma(i)^{-1}
\end{gather*}
and denote by $B(G)=G(K)/\sim$ and $B_I(G)=G(K)/\sim_I$ the corresponding 
quotients. For an element $b\in G(K)$, we denote by $[b]$ its equivalence class 
in $B(G)$.

Let $(Y_\Lambda)\in \mc A_\pd(\FF)$. By Lemma \ref{LemIsomOverK} there is an 
isomorphism $\varphi=(\varphi_\Lambda)_\Lambda:(\DD(Y_\Lambda))_\Lambda\too{\sim}\mc 
L\otimes\mc O_K$ of polarized multichains of $\mc O_B\otimes \mc O_K$-modules of 
type $(\mc L)$. Let $\Lambda\in\mc L$. Then $\mf F_\Lambda=\varphi_\Lambda\circ 
F_\DD\circ \varphi_\Lambda^{-1}$ is a $\sigma$-linear endomorphism of $\Lambda\otimes\mc O_K$. 
By functoriality of the Dieudonn\'e module, the base-change $\mf 
F_\Lambda\otimes_{\mc O_K}K:V\otimes K\to V\otimes K$ is independent of 
$\Lambda$ and we simply denote it by $\mf F$. In the same way, we obtain from 
the morphisms $V_\DD$ on the Dieudonn\'e modules a $\sigma^{-1}$-linear 
endomorphism $\mf V$ of $V\otimes K$ .

Let $b=\mf F\circ (\id_V\otimes \sigma)^{-1}$. Then $b$ is a $B\otimes K$-linear 
endomorphism of $V\otimes K$, and in view of \eqref{EqFVandPair} we have $b\in 
G(K)$, with $c(b)=p$. If $\varphi':(\DD(Y_\Lambda))_\Lambda\too{\sim}\mc 
L\otimes\mc O_K$ is another isomorphism, we have $\varphi'=i\circ \varphi$ for 
some $i\in I_0$, and the resulting element $b'\in G(K)$ will satisfy 
$b'=ib\sigma(i)^{-1}$. In this way we obtain a well-defined map $\alpha:\mc 
A_\pd(\FF)\to B_I(G)$.

The canonical projection $G(K)\to I\backslash G(K)/I$ factors through $B_I(G)$, 
so that we obtain a map $B_I(G)\too{\mathrm{can.}} I\backslash G(K)/I$.
Denote by $\gamma$ the composition $\mc A_\pd(\FF)\too{\alpha}B_I(G)\too{\mathrm{can.}} I\backslash G(K)/I$.
Denote by $\Perm\subset I\backslash G(K)/I$ the image of $\gamma$. By abuse of 
notation, we also denote by $\gamma:\mc A_\pd(\FF)\to \Perm$ the induced map. 

Denote by $B_I(G)_\Perm \subset B_I(G)$ the preimage of $\Perm$ under the 
canonical map $B_I(G)\too{\mathrm{can.}} I\backslash G(K)/I$. By abuse of 
notation, we also denote by $\alpha:\mc A_\pd(\FF)\to B_I(G)_\Perm$ the induced 
map. The situation is visualized by the following commutative diagram, in which 
the square is cartesian.
\begin{equation*}
    \xymatrix{
        & B_I(G)\ar[r]^-{\mathrm{can.}}&I\backslash G(K)/I\\
        \mc A_\pd(\FF)\ar@/_2pc/[rr]_\gamma\ar[r]^\alpha&B_I(G)_\Perm\ar[r]^-{\mathrm{can.}} \ar@{}[u]|{\bigcup}&\Perm\ar@{}[u]|{\bigcup}
        }
\end{equation*}
The following two results show that the map $\alpha:\mc A_\pd(\FF)\to B_I(G)_\Perm$ 
is very close to being a bijection. Our proofs are based on the discussion of the Siegel 
case by Hoeve in \cite[Chapter 7]{hoeve}.
\begin{prop}\label{PropAlphaSurjective}
        The map $\alpha:\mc A_\pd(\FF)\to B_I(G)_\Perm$ is surjective.
\end{prop}
\begin{proof}
    Let $\res{b}\in B_I(G)_\Perm$ and pick any representative $b\in G(K)$ of 
    $\res{b}$. By \eqref{EqFVandPair} and Lemma \ref{LemIvsI0} there is a unit 
    $v\in \mc O_K^\times$ with $c(b)=vp$. Using Lang's Lemma in combination with 
    an approximation argument, we find a $u\in \mc O_K^\times$ with 
    $v=(u\sigma(u)^{-1})^2$.  After replacing $b$ by $u^{-1}b\sigma(u)$, we may 
    assume that $c(b)=p$.

    Define $\mf F=b\circ(\id_V\otimes \sigma)$ and $\mf V=p\mf F^{-1}$. Let $\Lambda\in \mc L$. Then
    $\Lambda\otimes\mc O_K$ is stable under $\mf F$ and 
    $\mf V$, and we denote by $\mf F_\Lambda$ and $\mf V_\Lambda$ the 
    induced endomorphisms of $\Lambda\otimes\mc O_K$. Dieudonn\'e theory implies that the 
    chain $((\Lambda\otimes\mc O_K,\mf F_\Lambda, \mf V_\Lambda),\rho_{\Lambda',\Lambda})$ is of the form 
    $\DD(Y)$ for an $\mc L$-set $Y=(Y_\Lambda,\varrho_{\Lambda',\Lambda})$ of 
    $p$-divisible groups of determinant $\det_{V_{\CC,-i}}$ over $\FF$.

    From $c(b)=p$, we obtain
    \begin{equation*}
        \pairt{\mf F_\Lambda x}{y}_{\Lambda,\mc O_K}=\sigma\pairt{x}{\mf 
            V_{\Lambda^\vee} y}_{\Lambda,\mc O_K},\ x\in\Lambda\otimes\mc 
        O_K,y\in\Lambda^\vee\otimes\mc O_K,
    \end{equation*}
    compare \eqref{EqFVandPair}. Dieudonn\'e theory then implies that 
    $\pairtd_{\Lambda,\mc O_K}$ is induced by an isomorphism 
    $\lambda_\lambda:Y_\Lambda\to Y_{\Lambda^\vee}^\vee$, and the tuple 
    $\lambda=(\lambda_\Lambda)$ provides us with a polarization of $Y$.  The 
    isomorphism class of $(Y,\lambda)$ is the desired preimage of $\res{b}$ 
    under $\alpha$.
\end{proof}
\begin{prop}
    Let $Y=(Y_\Lambda,\varrho_{\Lambda',\Lambda},\lambda_\Lambda)$ and $Y'=(Y'_\Lambda,\varrho'_{\Lambda',\Lambda},\lambda'_\Lambda)$
    be two points of $\mc A_\pd(\FF)$. Then $\alpha(Y)=\alpha(Y')$ if and only if there exist both
    an isomorphism $\phi=(\phi_\Lambda):(Y_\Lambda,\varrho_{\Lambda',\Lambda})\to 
    (Y'_\Lambda,\varrho'_{\Lambda',\Lambda})$
    of $\mc L$-sets of $p$-divisible groups over $\FF$ and a unit 
    $u\in\ZZ_p^\times$ such that the following diagram commutes for all 
    $\Lambda\in\mc L$.
    \begin{equation*}
        \xymatrix{
            Y_\Lambda\ar[r]^{\phi_\Lambda}\ar[d]_{u\lambda_\Lambda}&Y'_{\Lambda}\ar[d]^{\lambda'_{\Lambda}}\\
            Y_{\Lambda^\vee}^\vee &Y_{\Lambda^\vee}'^\vee \ar[l]_{\phi_{\Lambda^\vee}^\vee}.
        }
    \end{equation*}
\end{prop}
\begin{proof}
    Choose isomorphisms $\varphi:\DD(Y)\too{\sim}\mc L\otimes\mc O_K$ and 
    $\varphi':\DD(Y')\too{\sim}\mc L\otimes\mc O_K$, and denote by $b$ and 
    $b'$ the resulting elements of $G(K)$ as above. If $\alpha(Y)=\alpha(Y')$, there is 
    an $i\in I$ with $b'=ib\sigma(i)^{-1}$. We have seen above that 
    $c(b)=p=c(b')$, so that $c(i)=\sigma(c(i))$. By Lemma \ref{LemIvsI0} and 
    \cite[Lemma 1.2]{kottwitzI}, the element $u:=c(i)$ lies in $\ZZ_p^\times$. As in 
    Lemma \ref{LemIvsIprime} we consider $i$ as a tuple of endomorphisms 
    $(\Lambda\otimes \mc O_K\to \Lambda\otimes \mc O_K)_{\Lambda\in \mc L}$.  
    The composition $\varphi'^{-1}\circ i\circ \varphi$ corresponds under Dieudonn\'e 
    theory to the desired morphism $\phi$. Similarly for the converse.
\end{proof}
\subsubsection{The map \texorpdfstring{$\gamma$}{gamma} and the KR stratification}
\begin{prop}\label{PropInTermsOfGamma}
    Two points $Y,Y'\in \mc A_\pd(\FF)$  lie in the same KR stratum if and only 
    if $\gamma(A)=\gamma(A')$.
\end{prop}
\begin{proof}
    The map $G(K)\to G(K),\ g\mapsto p\sigma^{-1}(g^{-1})$ descends to a 
    well-defined bijection $\tau:I\backslash G(K)/I\to I\backslash G(K)/I$, and 
    it suffices to show the corresponding statement for the composition
    $\mc A_\pd(\FF)\too{\gamma} I\backslash G(K)/I\too{\tau} I\backslash G(K)/I$ 
    instead of $\gamma$.

    Let $Y=(Y_\Lambda)_\Lambda\in\mc A_\pd(\FF)$.
    Choose an isomorphism $\varphi:(\DD(Y_\Lambda))_\Lambda\too{\sim}\mc 
    L\otimes\mc O_K$ and denote by $\mf F$ and $\mf V$ the resulting 
    endomorphisms of $V\otimes K$, as above. Let $b\in G(K)$ with $\mf F=b\circ 
    (\id_V\otimes \sigma)$. We have $\mf V=p\sigma^{-1}(b^{-1})\circ 
    (\id_V\otimes\sigma^{-1})$, so that $\mf V(\Lambda\otimes\mc 
    O_K)=p\sigma^{-1}(b^{-1})(\Lambda\otimes\mc O_K),\ \Lambda\in \mc L$.

    For $\Lambda\in\mc L$, denote by $\pi_\Lambda:\Lambda\otimes\mc O_K\to 
    \Lambda\otimes\FF$ the canonical projection. For a tuple 
    $M=(M_\Lambda)_{\Lambda\in \mc L}$ of $\mc O_B\otimes\mc O_K$-submodules 
    $p\Lambda\otimes\mc O_K\subset M_\Lambda\subset \Lambda\otimes\mc O_K$ further write
    $\pi(M)=(\pi_\Lambda(M_\Lambda))_\Lambda$. Note that $M$ and $\pi(M)$ 
    mutually determine each other.

    By definition the KR stratum that $Y\in \mc A_\pd(\FF)$ 
    lies in is given by the $\Aut(\mc L)(\FF)$-orbit of the point
    $\pi\bigl((\mf V(\Lambda\otimes\mc O_K))_\Lambda\bigr)\in \Mloc(\FF)$. 
    By Lemma \ref{LemIsomOverK}, Lemma \ref{LemIvsI0} and Lemma 
    \ref{LemIvsIprime}, this orbit is equal to
    \begin{equation*}
        \pi\bigl(\{ip\sigma^{-1}(b^{-1})\cdot(\mc L\otimes\mc O_K)\mid i\in I\}\bigr).
    \end{equation*}
    We conclude by noting that for an element $g\in G(K)$, the set 
    $\{ig\cdot(\mc L\otimes\mc O_K)\mid i\in I\}$ and the 
    image of $g$ in $I\backslash G(K)/I$ mutually determine each other.
\end{proof}
\begin{defn}\label{DefnIndexKR}
    For $x\in \Perm$, we denote by $\mc A_{\pd,x}=\gamma^{-1}(x)$ and $\mc 
    A_{x}=\delta^{-1}(\mc A_{\pd,x})$ the corresponding KR stratum in 
    $\mc A_\pd(\FF)$ and $\mc A(\FF)$, respectively.
\end{defn}
\begin{remark}[{Compare \cite[11.3]{hoeve}}]\label{RemarkVvsF}
    The normalization of Definition \ref{DefnIndexKR} amounts to indexing the KR 
    stratification by the relative position of $\mc L\otimes\mc O_K$ to its 
    image under \emph{Frobenius}. This seems to be the natural normalization in the 
    current context, see in particular Remark \ref{RemarkADLvsIntersection} 
    below. In the context of affine flag varieties however, to be explained in 
    the following sections, the natural normalization seems to be by the 
    relative position of $\mc L\otimes\mc O_K$ to its image under 
    \emph{Verschiebung}, see for example Remark \ref{RemVvsFSym}. This amounts to replacing the map $\gamma$ by the 
    composition
    $\mc A(\FF)\too{\gamma} I\backslash G(K)/I\too{g\mapsto 
        p\sigma^{-1}(g^{-1})} I\backslash G(K)/I$,
    as we have for example done in the proof of Proposition 
    \ref{PropInTermsOfGamma}.

    Let us note that for the question of the $p$-rank on a KR stratum both 
    normalizations yield the same results.
\end{remark}
Denote by $r_p:\mc A_\pd(\FF)\to\NN$ the map with $r_p((Y_\Lambda)_\Lambda)$ 
equal to the common $p$-rank of the $Y_\Lambda$.
In view Proposition \ref{PropInTermsOfGamma}, Theorem \ref{ThmPRankDivisible} amounts 
to the following statement.
\begin{thm}\label{ThmPRankConstant2}
    Assume that $\mc L$ is complete. The map $r_p:\mc A_\pd(\FF)\to \NN$ factors 
    through $\mc A_\pd(\FF)\too{\gamma}\Perm$.
\end{thm}
\subsubsection{The Newton stratification}
The canonical projection $G(K)\to B(G)$ factors through $B_I(G)$, so that we get 
a map $B_I(G)\too{\mathrm{can.}} B(G)$. Denote by $\beta$ the composition $\mc 
A_\pd(\FF)\too{\alpha}B_I(G)\too{\mathrm{can.}} B(G)$.
\begin{defn}
    Let $b\in B(G)$. We define $\mc N_{\pd,b}:=\beta^{-1}(b)\subset \mc A_\pd(\FF)$ and $\mc 
    N_{b}:=(\beta\circ \delta)^{-1}(b)\subset \mc A(\FF)$, and call it the \emph{Newton stratum}
    associated with $b$ in $\mc A_\pd(\FF)$ and $\mc A(\FF)$, respectively.
\end{defn}
Denote by $\DD$ the diagonalizable affine group with character group $\QQ$ over 
$K$. Let $b\in G(K)$. We denote by $\nu_b:\DD\to G_K$ the corresponding Newton 
map, defined in \cite[4.2]{kottwitzI}.\footnote{Note that the discussion in 
    loc.\  cit.\ still remains valid for not necessarily connected reductive 
groups over $\QQ_p$.}
The morphism $\nu_b$ makes $V_K$ into a representation of $\DD$ and we consider 
the corresponding weight decomposition $V_K=\oplus_{\chi\in\QQ} V_\chi$. We 
define
\begin{equation*}
    \nu_{b,0}:=\dim_K V_0.
\end{equation*}
If $g\in G(K)$, we know that $\nu_{gb\sigma(g)^{-1}}=\Int(g)\circ \nu_b$, where 
$\Int(g):G(K)\to G(K),\ h\mapsto ghg^{-1}$, and consequently 
$\nu_{b,0}=\nu_{gb\sigma(g)^{-1},0}$. Thus the map $G(K)\to\NN,\ b\mapsto 
\nu_{b,0}$ factors through $B(G)$, and we also denote by $B(G)\to\NN,\ b\mapsto 
\nu_{b,0}$ the resulting map.
\begin{prop}\label{PropPRankNewton}
    The map $r_p:\mc A_\pd(\FF)\to \NN$ factors as 
    \begin{equation*}
        \mc A_\pd(\FF)\too{\beta} B(G) \too{b\mapsto \nu_{b,0}}\NN.
    \end{equation*}
    In other words, for $b\in B(G)$ the $p$-rank on $\mc N_b$ is constant with 
    value $\nu_{b,0}$.
\end{prop}
\begin{proof}
    This follows from the fact that the isotypical component of slope 0 in 
    $\DD(G)\otimes_{\mc O_K} K$ comes precisely from the \'etale part of $G$, 
    see for instance \cite[IV]{demazure}.
\end{proof}
\textbf{Assume from now on that $\mc L$ is complete.}
We can summarize the above discussion in the following commutative diagram, with 
the dotted arrow coming from Theorem \ref{ThmPRankConstant2}.
\begin{equation}\label{EqCommDiagram}
    \begin{gathered}
        \xymatrix{
            \mc A_\pd(\FF)\ar@/^2pc/[rr]^\gamma\ar@/_4pc/[drr]_{r_p}\ar@{>>}[r]^-\alpha\ar[dr]_\beta&B_I(G)_\Perm\ar[r]^-{\mathrm{can.}}\ar[d]_{\mathrm{can.}} &\Perm\ar@{.>}[d]\\
            &B(G)\ar[r]^-{b\mapsto \nu_{b,0}}&\NN&\\
        }
    \end{gathered}
\end{equation}
\begin{defn}\label{DefnAffDL}
    Let $b\in G(K)$ and $x\in I\backslash G(K)/I$. The
    \emph{affine Deligne-Lusztig variety associated with $b$ and $x$} is defined 
    by
    \begin{equation*}
        X_x(b)=\{g\in G(K)/I\mid g^{-1}b\sigma(g)\in IxI\}.
    \end{equation*}
\end{defn}
\begin{prop}\label{PropNonemptinessInTermsOfFibers}
    Let $x\in \Perm$ and $b\in G(K)$. Then the following equivalence holds.
    \begin{equation*}
        X_x(b)\neq \emptyset \Leftrightarrow \mc A_{\pd,x}\cap \mc N_{\pd,[b]}\neq\emptyset.
    \end{equation*}
\end{prop}
\begin{proof}
    Follows from the definitions and Proposition \ref{PropAlphaSurjective}.
\end{proof}
\begin{remark}\label{RemarkADLvsIntersection}
    Let $x\in \Perm$ and let $b\in G(K)$.
    Although not established in full generality, it is expected that the 
    following equivalence holds, see \cite[Proposition 12.6]{haines}.
    \begin{equation}\label{EqEquivNonempty}
        X_x(b)\neq \emptyset \Leftrightarrow \mc A_x\cap \mc N_{[b]}\neq 
        \emptyset.
    \end{equation}
    The difficulty in proving this equivalence lies in the construction of a suitable 
    $\mc L$-set of abelian varieties with prescribed $\mc L$-set of
    $p$-divisible groups. For recent progress in the \emph{unramified} case due 
    to Viehmann and Wedhorn see \cite{vw}
\end{remark}
\begin{thm}\label{ThmFormulaI}
    Let $x\in \Perm$ and let $b\in G(K)$. Assume that $X_x(b)\neq \emptyset$.  
    Then the $p$-rank on $\mc A_{\pd,x}$ (and a fortiori on $\mc A_x$) is 
    constant with value $\nu_{[b],0}$.
\end{thm}
\begin{proof}
    Clear from Proposition \ref{PropNonemptinessInTermsOfFibers}, Theorem 
    \ref{ThmPRankConstant2} and Proposition \ref{PropPRankNewton}.
\end{proof}
\begin{remark}
    We expect that $I\subset G(K)$ is an Iwahori subgroup. This is true in
    the situations to be studied in the following sections, and would provide
    a more natural view on the set $I\backslash G(K)/I$ occurring above as we 
    could then identify it with a suitable extended affine Weyl group, see 
    \cite[Appendix]{pr}. Proving this statement seems to require a
    case-by-case analysis. The case of a ramified unitary group has been studied 
    in \cite[\textsection 1.2]{pr3}, and the case of the orthogonal group has 
    been investigated in \cite[\textsection 4.3]{smithling_orthogonal_even}.
\end{remark}
\subsection{A combinatorial lemma}\label{SecComputing}
The following combinatorial result explains the relationship between the 
abstract formula of Theorem \ref{ThmFormulaI} and the more concrete formulas of 
the following sections.

Let $n\in\NN$ and consider the canonical semidirect product 
$\widetilde{W}:=S_n\ltimes \ZZ^n$.
Let $\Xi$ be a finite cyclic group of order $f$ with generator $\sigma$.
We have the shift $\prod_{\xi\in\Xi}\widetilde{W}\to 
\prod_{\xi\in\Xi}\widetilde{W},\ (x_\xi)_\xi\mapsto (x_{\sigma^{-1}\xi})_\xi$.  
By abuse of notation, we simply denote it by $\sigma$.
\begin{lem}\label{LemCombinatoricsInSemiDirect}
    Let $(w_\xi)_\xi\in \prod_{\xi\in\Xi} S_n$ and $(\lambda_\xi)_\xi\in 
    \prod_{\xi\in\Xi}\ZZ^n$.  Assume that for all $\xi\in \Xi$ and all $1\leq 
    i\leq n$, the following statement holds.
    \begin{equation}\label{EqInclusionConditionConcrete}
        \lambda_\xi(i)\geq 0\quad\text{and}\quad (\lambda_\xi(i)=0 \Rightarrow 
        w_\xi(i)\leq i).
    \end{equation}
    Let $x=(w_\xi u^{\lambda_\xi})_\xi\in \prod_{\xi\in\Xi}\widetilde{W}$.  
    Choose $N\in\NN_{\geq 1}$ such that $\prod_{k=0}^{Nf-1} \sigma^k(x)\in 
    \prod_{\xi\in\Xi}\ZZ^n$.  Consider the element
    \begin{equation*}
        \nu=(\nu_\xi)_\xi:=\frac{1}{Nf}\prod_{k=0}^{Nf-1} \sigma^k(x)
    \end{equation*}
    of $\prod_{\xi\in\Xi}\QQ_{\geq 0}$.
    Then for each $1\leq i\leq n$, the following statements are equivalent.
    \begin{enumerate}
        \item $\exists \xi\in \Xi:\ \nu_\xi(i)=0$.
        \item $\forall \xi\in \Xi:\ \nu_\xi(i)=0$.
        \item $\forall \xi\in \Xi:\ (w_\xi(i)=i\wedge \lambda_\xi(i)=0)$.\qed
    \end{enumerate}
\end{lem}
\section{The symplectic case}\label{SecSym}
\subsection{Number fields}\label{SecNF}
We first fix some notation concerning (extensions of) number fields. Let $K/\QQ$ 
be a number field. We will always denote by $\mc O_{K}$ the ring of integers of 
$K$. If $\mc P$ is a nonzero prime of $\mc O_{K}$, we will always denote by 
$k_{\mc P}=\mc O_{K}/\mc P$ its residue field and by $\rho_{\mc P}:\mc O_{K}\to 
k_{\mc P}$ the corresponding residue morphism. We further denote by $K_\mc P$ 
the completion of $K$ with respect to $\mc P$ and by $\mc O_{K_\mc P}$ the 
valuation ring of $K_\mc P$.

Let $K_0/\QQ$ be a number field and assume that 
$p\mc O_{K_0}=\mc P_0^{e_0}$ for a single prime $\mc P_0$ of $\mc O_{K_0}$ and 
some $e_0\in\NN$.
Denote by $\Sigma_0$ the set of all embeddings $K_0\hookrightarrow \CC$.
Fix a finite Galois extension $L/\QQ$ with $K_0\subset L$ and write
$G=\Gal(L/\QQ)$ and $H_0=\Gal(L/K_0)$. Fix a prime $\mc Q$ of $\mc O_L$ lying 
over $\mc P_0$ and denote by $G_\mc Q\subset G$ the corresponding decomposition 
group.  
\begin{lem}\label{LemResidueMapinNonGaloisCase}
    There is a unique map $\gamma_0=\gamma_{\mc P_0}:\Sigma_0\to \Gal(k_{\mc 
    P_0}/\FF_p)$ satisfying
    \begin{equation}\label{EqCharPropofGamma}
        \forall \sigma\in\Sigma_0\forall a\in\mc O_{K_0}:\quad \rho_\mc 
        Q(\sigma(a))=\gamma_0(\sigma)(\rho_{\mc P_0}(a)).
    \end{equation}
    It is surjective and all its fibers have cardinality $e_0$.  
\end{lem}
\begin{proof}
    Left to the reader.
\end{proof}
Let $K/K_0$ be a quadratic extension. Denote by $\inv$ the non-trivial element of 
$\Gal(K/K_0)$. Assume that $\mc P_0\mc O_{K}=\mc P_+\mc P_-$ for two distinct 
primes $\mc P_+,\mc P_-$ of $\mc O_{K}$, say $\mc Q\cap \mc O_{K}=\mc P_+$. 
Consequently $\mc P_-=\mc P_+^\inv$.  Denote by $\alpha:G\to \Sigma$ the 
restriction map. Fix a lift $\tau_\inv\in G$ of $\inv$ under $\alpha$. 
Define subsets $\Sigma_\pm\subset \Sigma$ by $\Sigma_+=\alpha(G_\mc QH)$ 
and $\Sigma_-=\alpha(G_\mc Q\tau_\inv H)$. Then $\Sigma=\Sigma_+\amalg \Sigma_-$. 
We identify $k_{\mc P_\pm}$ with $k_{\mc P_0}$ via the isomorphism induced by 
the inclusion $\mc O_{K_0}\subset \mc O_K$
\begin{lem}\label{LemGammaSplit}
    There are unique maps $\gamma_\pm:\Sigma_\pm\to \Gal(k_{\mc P_0}/\FF_p)$ satisfying
    \begin{equation}\label{EqRedOfSigmaUnify}
        \forall \sigma\in\Sigma_\pm\forall a\in\mc O_K:\ \rho_\mc 
        Q(\sigma(a))=\gamma_0(\sigma|_{K_0})(\rho_{\mc P_\pm}(a)).
    \end{equation}
\end{lem}
\begin{proof}
    Left to the reader.
\end{proof}
\subsection{The PEL datum}\label{SecPELSym}
Let $g,n\in\NN_{\geq 1}$.
We start with the PEL datum consisting of the following objects.
\begin{enumerate}
	\item A totally real field extension $F/\QQ$ of degree $g$.  
	\item The identity involution $\id_F$ on $F$.
	\item A $2n$-dimensional $F$-vector space $V$.
    \item The symplectic form $\pairtd:V\times V\to \QQ$ on the underlying 
        $\QQ$-vector space of $V$ constructed as follows: Fix once and for all a 
        symplectic form $\pairtd':V\times V\to F$ and a basis $\mf E'=(e'_1,\dots,e'_{2n})$ of $V$ such that $\pairtd'$ is described 
        by the matrix $\widetilde{J}_{2n}$ with respect to $\mf E'$. 
        Define $\pairtd=\tr_{F/\QQ}\circ\pairtd'$.  
    \item The $F\otimes\RR$-endomorphism $J$ of $V\otimes\RR$ described by the 
        matrix $-\widetilde{J}_{2n}$ with respect to $\mf E'$.
\end{enumerate}
\begin{remark}
    Denote by $\GSp_{\pairtd'}$ the $F$-group of symplectic similitudes with 
    respect to $\pairtd'$, and by $c:\GSp_{\pairtd'}\to \GG_m$ the factor of 
    similitude. Then the reductive $\QQ$-group $G$ associated with the above PEL 
    datum fits into the following cartesian diagram.
\begin{equation*}
    \xymatrix{
        G\ar[d]_c \xyhookrightarrow & \Res_{F/\QQ} \GSp_{\pairtd'}\ar[d]^c\\
        \GG_{m,\QQ}\xyhookrightarrow& \Res_{F/\QQ}\GG_{m,F}.
    }
\end{equation*}
\end{remark}
We assume that $p\mc O_F=\mc P^e$ for a single prime $\mc P$ of $\mc O_F$.  
Denote by $f=[k_\mc P:\FF_p]$ the corresponding inertia degree, so that $g=ef$.  
We have $F\otimes\QQ_p=\fp$ and $\mc O_F\otimes\ZZ_p=\mc O_{F_\mc P}$. Fix once 
and for all a uniformizer $\pi$ of $\mc O_{F}\otimes \ZZ_{(p)}$.

Denote by $\mf C=\mf C_{\mc O_{F_\mc P}\mid \ZZ_p}$
the inverse different of the extension $\fp/\QQ_p$. Fix a generator $\delta$ of 
$\mf C$ over $\mc O_{F_\mc P}$ and define a basis $(e_1,\ldots,e_{2n})$ of 
$V_{\QQ_p}$ over $\fp$ by $e_i=e'_i,\  e_{n+i}=\delta e'_{n+i},\ 1\leq i\leq n$.

Let $0\leq i< 2n$. We denote by $\Lambda_i$ the $\mc O_{F_\mc P}$-lattice in 
$V_{\QQ_p}$ with basis
\begin{equation*}
	\mf E_i=(\pi^{-1}e_1,\ldots,\pi^{-1}e_i,e_{i+1},\ldots,e_{2n}).
\end{equation*}
For $k\in\ZZ$ we further define $\Lambda_{2nk+i}=\pi^{-k}\Lambda_i$ and we 
denote by $\mf E_{2nk+i}$ the corresponding basis obtained from $\mf E_i$.
Then $\mc L=(\Lambda_i)_i$ is a complete chain of $\mc O_{F_\mc P}$-lattices in 
$V$. For $i\in\ZZ$, the dual lattice 
$\Lambda_i^\vee:=\{x\in V_{\QQ_p}\mid \pairt{x}{\Lambda_i}_{\QQ_p}\subset\ZZ_p\}$
of $\Lambda_i$ is given by $\Lambda_{-i}$.  Consequently the chain $\mc L$ is 
self-dual.

Let $i\in\ZZ$. We denote by $\rho_i:\Lambda_{i}\to \Lambda_{i+1}$ the 
inclusion, by $\vartheta_i:\Lambda_{2n+i}\to \Lambda_i$ the isomorphism given by 
multiplication with $\pi$ and by $\pairtd_i:\Lambda_i\times \Lambda_{-i}\to 
\ZZ_p$ the restriction of $\pairtd_{\QQ_p}$. Then 
$(\Lambda_{i},\rho_{i},\vartheta_{i},\pairtd_{i})_i$ is a polarized chain of $\mc 
O_{\fp}$-modules of type $(\mc L)$, which, by abuse of notation, we also 
denote by $\mc L$.

Denote by $\paird_i:\Lambda_i\times \Lambda_{-i}\to \mc O_{F_\mc P}$ the 
restriction of the pairing $\delta^{-1}\pairtd'_{\QQ_p}$. It is the perfect pairing 
described by the matrix $\widetilde{J}_{2n}$ with respect to the bases $\mf E_i$ 
and $\mf E_{-i}$.
\subsection{The determinant morphism}
Denote by $\Sigma$ the set of all embeddings $F\hookrightarrow \CC$. The canonical isomorphism
\begin{equation}\label{Fdecomp}
    F\otimes\CC = \prod_{\sigma\in\Sigma}\CC
\end{equation}
induces a decomposition $V\otimes\CC = \prod_{\sigma\in\Sigma} V_\sigma$ into $\CC$-vector 
spaces $V_\sigma$, and the morphism $J_\CC$ decomposes into the product of 
$\CC$-linear maps $J_\sigma:V_\sigma\to V_\sigma$. Each $J_\sigma$ induces a 
decomposition $V_\sigma=V_{\sigma,i}\oplus V_{\sigma,-i}$, where $V_{\sigma,\pm 
i}$ denotes the $\pm i$-eigenspace of $J_\sigma$. From the explicit description 
of $J$ in terms of $\mc B$ above one sees that both $V_{\sigma,i}$ and 
$V_{\sigma,-i}$ have dimension $n$ over $\CC$.

The $(-i)$-eigenspace $V_{-i}$ of $J_\CC$ is given by 
$V_{-i}=\prod_{\sigma\in\Sigma} 
V_{-i,\sigma}$. As $\dim_\CC V_{-i,\sigma}\allowbreak=n$ for all $\sigma$, there is an 
isomorphism $V_{-i}\simeq (\prod_{\sigma}\CC)^n$ of $\prod_{\sigma}\CC$-modules 
and hence the $\mc O_F\otimes \CC$-module corresponding to $V_{-i}$ under 
\eqref{Fdecomp} is isomorphic to $\mc O_F^n\otimes \CC$. In particular, the 
morphism $\det_{V_{-i}}:V_{\mc O_F\otimes\CC}\to \AF^1_\CC$ is defined over 
$\ZZ$.
\subsection{The local model}
For the chosen PEL datum, Definition \ref{DefnLocalModelGen} amounts to the 
following.
\begin{defn}\label{DefnLocalModel}
    The local model $\Mloc$ is the functor on the category of
    $\ZZ_p$-algebras with $\Mloc(R)$ the set of tuples $(t_i)_{i\in \ZZ}$ of
    $\mc O_{F}\otimes R$-submodules $t_i\subset \Lambda_{i,R}$
    satisfying the following conditions for all $i\in\ZZ$.
    \renewcommand\theenumi {\alph{enumi}}
    \begin{enumerate}
        \item\label{DefnLocalModel-Functoriality}
            $\rho_{i,R}(t_i)\subset t_{i+1}$.
        \item\label{DefnLocalModel-Projectivity}
            The quotient $\Lambda_{i,R}/t_i$ is a finite locally free $R$-module.
        \item\label{DefnLocalModel-Determinant} 
            We have an equality 
            \begin{equation*}
                \det_{\Lambda_{i,R}/t_i}=\det_{V_{-i}}\otimes R
            \end{equation*}
            of morphisms $V_{\mc O_{F}\otimes R}\to \AF^1_R$.
        \item\label{DefnLocalModel-DualityCondition} Under the pairing 
            $\pairtd_{i,R}:\Lambda_{i,R}\times \Lambda_{-i,R}\to R$, the 
            submodules $t_i$ and $t_{-i}$ pair to zero.
        \item\label{DefnLocalModel-Periodicity}
            $\vartheta_i(t_{2n+i})=t_i$.
    \end{enumerate}
\end{defn}
\subsection{The special fiber of the local model}
For $i\in\ZZ$, denote by $\res{\Lambda}_i$ the $\FF[u]/u^e$-module 
$(\FF[u]/u^e)^{2n}$ and by $\res{\mf E}_i$ its canonical basis. Denote by 
$\res{\paird}_i:\res{\Lambda}_i\times \res{\Lambda}_{-i}\to \FF[u]/u^e$ the 
pairing described by the matrix $\widetilde{J}_{2n}$ with respect to $\res{\mf E}_i$ 
and $\res{\mf E}_{-i}$.  Denote by $\res{\vartheta}_i:\res{\Lambda}_{2n+i}\to 
\res{\Lambda}_i$ the identity morphism. For $k\in\ZZ$ and $0\leq i<2n$, let $\res{\rho}_{2n+i}:\res{\Lambda}_{2n+i}\to 
\res{\Lambda}_{2n+i+1}$ be the morphism described  by the matrix 
$\mathrm{diag}(1^{(i)},u,1^{(2n-i-1)})$ with respect to $\res{\mf E}_{2nk+i}$ 
and $\res{\mf E}_{2nk+i+1}$.
\begin{defn}\label{DefnSpecialLocalModel}
    Let $\Me^{e,n}$ be the functor on the category of
    $\FF$-algebras with $\Me^{e,n}(R)$ the set of tuples $(t_i)_{i\in \ZZ}$ of
    $R[u]/u^e$-submodules $t_i\subset \res{\Lambda}_{i,R}$
    satisfying the following conditions for all $i\in\ZZ$.
    \renewcommand\theenumi {\alph{enumi}}
    \begin{enumerate}
        \item $\res{\rho}_{i,R}(t_i)\subset t_{i+1}$.
        \item The quotient $\res{\Lambda}_{i,R}/t_i$ is finite locally free over $R$.
        \item\label{DefnSpecialLocalModel-Determinant}
            For all $p\in R[u]/u^e$, we have
            \begin{equation*}
                \chi_R(p|\res{\Lambda}_{i,R}/t_i)=\bigl(T-p(0)\bigr)^{ne}
            \end{equation*}
            in $R[T]$.
        \item $t_i^{\perp,\res{\paird}_{i,R}}=t_{-i}$.
        \item $\res{\vartheta}_i(t_{2n+i})=t_i$.
    \end{enumerate}
\end{defn}
Denote by $\mf S$ the set of all embeddings $k_{\mc P}\hookrightarrow 
\FF$. Our choice of uniformizer $\pi$ induces a canonical isomorphism
\begin{equation}
    \label{Eqdecompspecialfiber}
    \mc O_F\otimes\FF=\prod_{\sigma\in\mf S} \FF[u]/(u^e).
\end{equation}
Let $i\in\ZZ$. From 
\eqref{Eqdecompspecialfiber} we obtain an isomorphism
\begin{equation}\label{EqLambdaInSpecialFiber}
    \Lambda_{i,\FF}=\prod_{\sigma\in\mf S}\res{\Lambda}_{i}
\end{equation}
by identifying the basis $\mf E_{i,\FF}$ with the product of the bases $\res{\mf 
E}_i$. Under this identification, the morphism $\rho_{i,\FF}$ decomposes into 
the morphisms $\res{\rho}_i$, the pairing $\paird_{i,\FF}$ decomposes into the 
pairings $\res{\paird}_i$ and the morphism $\vartheta_{i,\FF}$ decomposes into the 
morphisms $\res{\vartheta}_i$.

Let $R$ be an $\FF$-algebra and let $(t_i)_{i\in\ZZ}$ be a tuple of $\mc 
O_{F}\otimes R$-submodules $t_i\subset \Lambda_{i,R}$. Then 
\eqref{EqLambdaInSpecialFiber} induces decompositions
$t_i=\prod_{\sigma\in\mf S} t_{i,\sigma}$ into $R[u]/u^e$-submodules 
$t_{i,\sigma}\subset\res{\Lambda}_{i,R}$.
\begin{prop}\label{PropDecompositionofLocalModel}
    The morphism $\Mloc_\FF\to \prod_{\sigma\in\mf S} \Me^{e,n}$ given on $R$-valued 
    points by
    \begin{equation}\label{EqDecompLocModelSym}
        \begin{aligned}
            \Mloc_\FF(R)&\to \prod_{\sigma\in\mf S} \Me^{e,n}(R),\\
            (t_i)&\mapsto \left((t_{i,\sigma})_i\right)_\sigma
        \end{aligned}
    \end{equation}
    is an isomorphism of functors on the category of $\FF$-algebras.
\end{prop}
\begin{proof}
    The only point requiring an argument is the transition from $\pairtd_i$ to 
    $\paird_i$. It is warranted by the perfectness of the pairing $\mc O_{\fp}\times\mc O_{\fp}\to \ZZ_p,\ (x,y)\mapsto \tr_{\fp/\QQ_p}(\delta xy)$.
\end{proof}
\subsection{The affine Grassmannian and the affine flag variety for \texorpdfstring{$\GL_n$}{GL(n)}}\label{SecLattices}
Let $R$ be an $\FF$-algebra and let $n\in\NN$.
\begin{defn}\label{DefnLattice}
    A \emph{lattice} in $R(\!(u)\!)^{n}$ is an $R[\![u]\!]$-submodule $L\subset 
    R(\!(u)\!)^{n}$ satisfying the following conditions for some $N\in\NN$.
    \begin{enumerate}
        \item\label{DefnLattice-Inclusion} $u^NR[\![u]\!]^{n}\subset L\subset u^{-N}R[\![u]\!]^{n}$.
        \item\label{DefnLattice-Quotient} $u^{-N}R[\![u]\!]^{n}/L$ is a finite locally free $R$-module.
    \end{enumerate}
\end{defn}
The following statement is well-known. See for example \cite[Proposition 
4.5.5]{diss} for a proof.
\begin{prop}\label{PropLatticesFreeGlobal}
    Let $L$ be a lattice in $R(\!(u)\!)^n$. Then $L$ is a finite locally free 
    $R[\![u]\!]$-module of rank $n$.
\end{prop}
\begin{defn}
    The affine Grassmannian $\mc G$ is the functor on the category of 
    $\FF$-algebras with $\mc G(R)$ the set of lattices in $R(\!(u)\!)^n$.
\end{defn}
Denote by $\aff{\Lambda}_0=R[\![u]\!]^n$ the \emph{standard lattice}.  Clearly 
$\Lf \GL_n(R)$ acts on $\mc G(R)$ by multiplication from the left, and the 
stabilizer of $\aff{\Lambda}_0$ for this action is given by $\Lp \GL_n(R)$.  
Consequently we get an injective map
\begin{align*}
    \phi(R):\Lf \GL_n(R)/\Lp \GL_n(R)&\to \mc G(R)\\
    g&\mapsto g\aff{\Lambda}_0.
\end{align*}
It is equivariant for the left action by $\Lf \GL_n$.
\begin{prop}\label{PropGrassandGLn}
    The map $\phi$ identifies $\mc G$ with both the Zariski and the fpqc 
    sheafification of the presheaf
    $\Lf \GL_n/\Lp \GL_n$.  
\end{prop}
\begin{proof}
    By Proposition \ref{PropLatticesFreeGlobal} it is clear that any lattice lies 
    in the image of $\phi$ Zariski locally on $R$. It follows that 
    $\phi$ is the Zariski sheafification of the presheaf $\Lf \GL_n/\Lp 
    \GL_n$.  The fact that $\mc G$ is already an fpqc sheaf implies formally 
    that $\phi$ is also the fpqc sheafification of the presheaf $\Lf 
    \GL_n/\Lp \GL_n$.
\end{proof}
\begin{defn}
    A (complete, periodic) \emph{lattice chain} in $R(\!(u)\!)^n$ is a tuple 
    $(L_i)_{i\in\ZZ}$ of lattices $L_i$ in $R(\!(u)\!)^{n}$ satisfying the 
    following conditions for each $i\in\ZZ$.
    \begin{enumerate}
        \item $L_i\subset L_{i+1}$.
        \item (completeness) $L_{i+1}/L_i$ is a locally free $R$-module of rank $1$.
        \item (periodicity) $L_{n+i}=u^{-1} L_i$.
    \end{enumerate}
\end{defn}
\begin{defn}
    The \emph{affine flag variety} $\mc F$ is the functor on the category of 
    $\FF$-algebras with $\mc F(R)$ the set of (complete, periodic) lattice chains in 
    $R(\!(u)\!)^{n}$.
\end{defn}
Denote by $(e_1,\ldots,e_{n})$ the standard basis of $R(\!(u)\!)^{n}$ over 
$R(\!(u)\!)$.  For $0\leq i< n$ we denote by $\aff{\Lambda}_i$ the lattice 
in $R(\!(u)\!)^{n}$ with basis
\begin{equation*}
    \aff{\mf 
    E}_i=\left<u^{-1}e_1,\ldots,u^{-1}e_i,e_{i+1},\ldots,e_{n}\right>.
\end{equation*}
For $k\in\ZZ$ we further define $\aff{\Lambda}_{nk+i}=u^{-k}\aff{\Lambda}_i$ 
and we denote by $\aff{\mf E}_{nk+i}$ the corresponding basis obtained from 
$\aff{\mf E_i}$. Then $\aff{\mc L}=(\aff{\Lambda}_i)_i$ is a (complete, 
periodic) lattice chain in $R(\!(u)\!)^n$, called the \emph{standard lattice 
chain}.

In complete analogy with \cite[p. 131]{rz}, we have for an 
$\FF[\![u]\!]$-algebra $R$ the notion of a chain $\mc M=(M_i,\varrho_i:M_i\to M_{i+1},\theta_i:M_{n+i}\too{\sim} 
M_i)_{i\in\ZZ}$ of $R$-modules of type 
$(\aff{\mc L})$ (cf.\ \cite[Definition 7.5.1]{diss}).
The proof of \cite[Proposition A.4]{rz} then carries over without any changes to 
show the following result.
\begin{prop}\label{PropRZbigLin}
    Let $R$ be an $\FF[\![u]\!]$-algebra such that the image of $u$ in $R$
    is nilpotent. Then any two chains $\mc M,\mc N$ of $R$-modules of type 
    $(\aff{\mc L})$ are isomorphic locally for the Zariski topology on $R$. 
    Furthermore the functor $\Isom(\mc M,\mc N)$ is representable by a smooth 
    affine scheme over $R$.
\end{prop}
\begin{prop}\label{PropLiftingIsosLin}
    Let $R$ be an $\FF$-algebra and let $\mc M,\mc N$ be chains of 
    $R[\![u]\!]$-modules of type $(\aff{\mc L})$.
    Then the canonical map $\Isom(\mc M,\mc N)(R[\![u]\!])\to \Isom(\mc M,\mc 
    N)(R[\![u]\!]/u^m)$ is surjective for all $m\in\NN_{\geq 1}$. In particular 
    $\mc M$ and $\mc N$ are isomorphic locally for the Zariski topology on $R$.
\end{prop}
\begin{proof}
    Analogous to the proof of Lemma \ref{LemIsomOverK}.
\end{proof}
\begin{remark}\label{RemarkLatticesAreChainsLin}
    Let $R$ be an $\FF$-algebra and let $(L_i)_i\in \mc F(R)$. For $i\in\ZZ$ 
    denote by $\varrho_i:L_i\to L_{i+1}$ the inclusion and by 
    $\theta_i:L_{n+i}\to L_i$ the isomorphism given by multiplication with $u$.
    Then $(L_i,\varrho_i,\theta_i)$ is a chain of $R[\![u]\!]$-modules of type 
    $(\aff{\mc L})$.
\end{remark}
\begin{remark}\label{RemarkActionOfGLOnF}
    The group $\Lf \GL_n(R)$ acts on $\mc F(R)$ via $g\cdot (L_i)_i=(gL_i)_i$.  
    Denote by $I(R)$ the stabilizer of $\aff{\mc L}$ for this action. One checks 
    that $I(R)\subset \GL_n(R[\![u]\!])$ is equal to the preimage of $B(R)$
    under the reduction map $\GL_n(R[\![u]\!])\to \GL_n(R),\ u\mapsto 0$. Here  
    $B(R)\subset \GL_n(R)$ denotes the subgroup of upper triangular matrices.
\end{remark}
We obtain for each $\FF$-algebra $R$ an injective map
\begin{align*}
    \Lf {\GL_n}(R)/I(R)&\too{\phi(R)} \mc F(R),\\
    g&\xmapsto{\hphantom{\phi(R)}} g\cdot \aff{\mc L}.
\end{align*}
\begin{prop}\label{PropFlagandGL}
    The morphism $\phi$ identifies $\mc F$ with both the Zariski and the fpqc 
    sheafification of the presheaf $\Lf \GL_n/I$.
\end{prop}
\begin{proof}
    Let $R$ be an $\FF$-algebra and let $\mc M\in\mc F(R)$. We consider $\mc M$ 
    as a chain of $R[\![u]\!]$-modules of type $(\aff{\mc L})$
    as in Remark \ref{RemarkLatticesAreChainsLin}. By Proposition \ref{PropLiftingIsosLin}, 
    the chains $\aff{\mc L}$ and $\mc M$ are isomorphic locally for the Zariski 
    topology on $R$. Such an isomorphism $\aff{\mc L} \to \mc M$ is given by 
    multiplication with a single $g\in \GL_n(R[\![u]\!])$. Consequently $\mc M$ 
    lies in the image of $\phi$ Zariski locally on $R$. The fact that $\mc F$ is 
    already an fpqc sheaf implies formally that $\phi$ is also the fpqc 
    sheafification of the presheaf $\Lf \GL_n/I$.
\end{proof}
\subsection{The affine flag variety}\label{SecAffFlagSym}
This section deals with the affine flag variety for the symplectic group. Our 
discussion loosely follows the one in \cite[\textsection 10-11]{pr2}. Note 
though that in loc.\ cit.\ there is a minor problem with the 
definition of the notion of self-duality for lattice chains, see Remark 
\ref{RemWrongDualDefn} below. We have learned the correct formulation of this 
definition from \cite[\textsection 4.2]{smithling_unitary_odd}, which deals with 
the case of a ramified unitary group.

Let $R$ be an $\FF$-algebra. Let $\aff{\paird}$ be the symplectic form on 
$R(\!(u)\!)^{2n}$ described by the matrix $\widetilde{J}_{2n}$ with respect to 
the standard basis of $R(\!(u)\!)^{2n}$ over $R(\!(u)\!)$. We denote by 
$\Sp=\Sp_{2n}$ the symplectic group and by $\GSp=\GSp_{2n}$ the group of 
symplectic similitudes with respect to $\aff{\paird}$.

For a lattice $\Lambda$ in $R(\!(u)\!)^{2n}$ we define 
$\Lambda^\vee:=\{x\in R(\!(u)\!)^{2n}\mid \aff{\pair{x}{\Lambda}}\subset 
R[\![u]\!]\}$. Recall from Section \ref{SecLattices} the standard lattice chain 
$\aff{\mc L}=(\aff{\Lambda}_i)_i$ in $R(\!(u)\!)^{2n}$. Note that 
$(\aff{\Lambda}_i)^\vee=\aff{\Lambda}_{-i}$ for all $i\in\ZZ$. We denote by 
$\aff{\paird}_i:\aff{\Lambda}_i\times \aff{\Lambda}_{-i}\to R[\![u]\!]$ the 
restriction of $\aff{\paird}$.

In complete analogy with \cite[Definition 3.14]{rz}, we have for an 
$\FF[\![u]\!]$-algebra $R$ the notion of a polarized chain $\mc 
M=(M_i,\varrho_i:M_i\to M_{i+1},\theta_i:M_{2n+i}\too{\sim} M_i, \mc 
E_i:M_i\times M_{-i}\to R)_{i\in\ZZ}$ of $R$-modules of type $(\aff{\mc L})$ 
(cf.\ \cite[Definition 5.5.1]{diss}).
The proof of \cite[Proposition A.21]{rz} then carries over without any changes 
the show the following result.
\begin{prop}\label{PropRZbig}
    Let $R$ be an $\FF[\![u]\!]$-algebra such that the image of
    $u$ in $R$ is nilpotent. Then any two polarized chains $\mc M,\mc N$ of 
    $R$-modules of type $(\aff{\mc L})$ are isomorphic locally for the Zariski 
    topology on $R$.  Furthermore the functor $\Isom(\mc M,\mc N)$ is 
    representable by a smooth affine scheme over $R$.
\end{prop}
\begin{prop}\label{PropLiftingIsosSym}
    Let $R$ be an $\FF$-algebra and let $\mc M,\mc N$ be polarized chains of 
    $R[\![u]\!]$-modules of type $(\aff{\mc L})$. Then the canonical map 
    $\Isom(\mc M,\mc N)(R[\![u]\!])\to \Isom(\mc M,\mc N)(R[\![u]\!]/u^m)$ is 
    surjective for all $m\in\NN_{\geq 1}$. In particular $\mc M$ and $\mc N$ are 
    isomorphic locally for the Zariski topology on $R$.
\end{prop}
\begin{proof}
    Analogous to the proof of Lemma \ref{LemIsomOverK}.
\end{proof}
The following definition is a straightforward variant of \cite[\textsection 
4.2]{smithling_unitary_odd}.
\begin{defn}
    Let $R$ be an $\FF$-algebra and let $(L_i)_i$ be a lattice chain in 
    $R(\!(u)\!)^{2n}$.
    \begin{enumerate}
        \item Let $r\in\ZZ$. The chain $(L_i)_i$ is called \emph{$r$-self-dual} 
            if \begin{equation*}
                \forall i\in\ZZ:\ L_i^\vee=u^r L_{-i}.  \end{equation*}
            Denote by $\mc F^{(r)}_\Sp$ the functor on the category of 
            $\FF$-algebras with $\mc F_{\Sp}^{(r)}(R)$ the set of
            $r$-self-dual lattice chains in $R(\!(u)\!)^{2n}$.
        \item The chain $(L_i)_i$ is called \emph{self-dual} if Zariski locally 
            on $R$ there is an $a\in R(\!(u)\!)^\times$ such that 
            \begin{equation}\label{EquSelfdual}
                \forall i\in\ZZ:\ L_i^\vee=a L_{-i}.  \end{equation}
            We denote by $\mc F_{\GSp}$ the functor on the category of 
            $\FF$-algebras with $\mc F_{\GSp}(R)$ the set of self-dual lattice 
            chains in $R(\!(u)\!)^{2n}$.  \end{enumerate}
\end{defn}
Note that $\aff{\mc L}\in \mc F_\Sp^{(0)}(R)$.
\begin{lem}\label{LemUnitsinLaurent}
    Let $R$ be a ring and let $a\in R(\!(u)\!)^\times$. Then Zariski locally on 
    $R$, there are integers $n\leq n_0$, nilpotent elements 
    $a_{n},a_{n+1},\dots,a_{n_0-1}\in R$, a unit $a_{n_0}\in R^\times$ and 
    elements $a_{n_0+1},a_{n_0+2},\ldots\in R$ such that
    $a=\sum_{i=n}^{\infty}a_i u^i$.

    If $\Spec R$ is connected, such integers and elements exist globally on $R$.
\end{lem}
\begin{remark}\label{RemarkLocalModel}
    Let $R$ be a reduced $\FF$-algebra such that $\Spec R$ connected.
    Then
    \begin{equation*}
        \mc F_{\GSp}(R)=\bigcup_{r\in\ZZ} \mc F_{\Sp}^{(r)}(R).
    \end{equation*}
\end{remark}
\begin{proof}
    This follows immediately from Lemma \ref{LemUnitsinLaurent}.
\end{proof}
\begin{remark}\label{RemarkLatticesAreChainsSym}
    Let $R$ be an $\FF$-algebra and let $(L_i)_i\in \mc F_{\Sp}^{(0)}(R)$. For 
    $i\in\ZZ$ denote by $\varrho_i:L_i\to L_{i+1}$ the inclusion, by 
    $\theta_i:L_{2n+i}\to L_i$ the isomorphism given by multiplication with $u$ 
    and by $\mc E_i:L_i\times L_{-i}\to R[\![u]\!]$ the restriction of 
    $\aff{\paird}$. Then $(L_i,\varrho_i,\theta_i,\mc E_i)$ is a polarized chain 
    of $R[\![u]\!]$-modules of type $(\aff{\mc L})$.
\end{remark}
Recall from Remark \ref{RemarkActionOfGLOnF} the subfunctor $I\subset \Lf \GL_{2n}$.
We define a subfunctor $I_{\GSp}=I_{\GSp_{2n}}$ of $\Lf \GSp=\Lf \GSp_{2n}$ by 
$I_\GSp=\Lf \GSp_{2n}\cap I$. We consider all of these functors as functors on 
the category of $\FF$-algebras. 

The proof of the following result is similar to and therefore based on the proof 
of \cite[Theorem 4.1]{pr}.  
\begin{prop}\label{PropFuncDescofFlag}
    The natural action of $\Lf \GL_{2n}$ on $\mc F$ (cf.\ Remark 
    \ref{RemarkActionOfGLOnF}) restricts to an action of $\Lf \GSp$ on $\mc 
    F_\GSp$.
    Consequently we obtain an injective map 
    \begin{align*}
        \Lf {\GSp}(R)/I_{\GSp}(R)&\too{\phi(R)} \mc F_{\GSp}(R),\\
        g&\xmapsto{\hphantom{\phi(R)}} g\cdot \aff{\mc L}
    \end{align*}
    for each $\FF$-algebra $R$.
    The morphism $\phi$ identifies $\mc F_{\GSp}$ with both the Zariski and the 
    fpqc sheafification of the presheaf $\Lf {\GSp}/I_{\GSp}$.
\end{prop}
\begin{proof}
    Let $R$ be an $\FF$-algebra and let $\mc M=(L_i)_i\in\mc F_{\GSp}(R)$. 
    Working Zariski locally on $R$ we may assume that there is an $a\in 
    R(\!(u)\!)^\times$ such that \eqref{EquSelfdual} holds. Choose any $h\in 
    {\GSp}(R(\!(u)\!))$ with factor of similitude $a$, e.g.\ 
    $h=\diag(a^{(n)},1^{(n)})$. An easy computation shows that $h\mc M\in \mc 
    F_{\Sp}^{(0)}(R)$. We see as in the proof of Proposition \ref{PropFlagandGL} 
    that Zariski locally on $R$, there is a $g\in \Sp(R(\!(u)\!))$ with $h\mc 
    M=g\aff{\mc L}$. Consequently $\mc M=h^{-1}g\aff{\mc L}$ lies in the image 
    of $\phi$ Zariski locally on $R$. As $\mc F_{\GSp}$ is clearly a Zariski 
    sheaf, it follows that $\mc F_{\GSp}$ is indeed the Zariski sheafification 
    of the presheaf $\Lf {\GSp}/I_{\GSp}$. 
    
    To see that $\mc F_{\GSp}$ is also the fpqc sheafification of $\Lf 
    {\GSp}/I_{\GSp}$, it suffices to show that $\mc F_{\GSp}$ is an fpqc sheaf.  
    Let $(L_i)_i$ be a lattice chain in $R(\!(u)\!)^{2n}$. Assume that fpqc 
    locally on $R$ there is an $a\in R(\!(u)\!)^\times$ such that 
    \eqref{EquSelfdual} holds. The scalar $a$ gives rise to a well-defined 
    element of $(\Lf \GG_m/\Lp \GG_m)_{\mathrm{fpqc}}(R)$, where $(\Lf \GG_m/\Lp 
    \GG_m)_{\mathrm{fpqc}}$ denotes the fpqc sheafification of the presheaf $\Lf 
    \GG_m/\Lp \GG_m$. By Proposition \ref{PropGrassandGLn} any element of $(\Lf 
    \GG_m/\Lp \GG_m)_{\mathrm{fpqc}}(R)$ can be represented in $\Lf \GG_m(R)$ 
    Zariski locally on $R$, so that the scalar $a$ exists in fact Zariski 
    locally on $R$.
\end{proof}
\begin{remark}\label{RemWrongDualDefn}
    Let us note that there seems to exist a misconception surrounding the notion 
    of self-duality for lattice chains. In the literature one finds the 
    following definition: Let $R$ be an $\FF$-algebra. A lattice chain 
    $(L_i)_i\in \mc F(R)$ is called \emph{(naively) self-dual} if for each 
    $i\in\ZZ$ there is a $j\in\ZZ$ such that $L_{i}^\vee= L_j$. It is then 
    claimed that the fpqc local $\Lf {\GSp}$-orbit of $\aff{\mc L}$ (in the 
    sense of Proposition \ref{PropFuncDescofFlag}) is precisely the set of 
    (naively) self-dual lattice chains. This is wrong in both directions, as 
    shown by the following easy examples.
    \begin{itemize}[leftmargin=*]
        \item Let $n=1$ and $a\in R(\!(u)\!)^\times$. The chain 
            $(L_i)_i=a\aff{\mc L}$ satisfies $L_i^\vee=a^{-2}L_{-i},\ i\in\ZZ$.  
            Assume there is a $j\in\ZZ$ with $L_0^\vee=L_j$.
            Then $a^{-2}L_{0}=L_j$ and hence 
            $a^{-2}\aff{\Lambda}_{0}=\aff{\Lambda}_j$.  Projecting this equality 
            inside $R(\!(u)\!)^2$ to its first components yields the existence 
            of a $k\in\ZZ$ with $a^{-2}R[\![u]\!]=u^{k}R[\![u]\!]$, so that
            $u^ka^2\in R[\![u]\!]^\times$. If for example $R=\FF[x]/x^2$ and 
            $a=1+xu^{-1}$, such a $k$ does not exist.
        \item Conversely one easily sees that for $n\geq 2$, the (naively) 
            self-dual chain $(\aff{\Lambda}_{i+1})_{i\in\ZZ}$ does not lie in 
            the $\Lf \GSp(R)$-orbit of $\aff{\mc L}$ (unless $R=\{0\}$).
    \end{itemize}
\end{remark}
\subsection{Embedding the local model into the affine flag 
variety}\label{SecEmbLocFlagSym}
Let $R$ be an $\FF$-algebra. We consider an $R[u]/u^e$-module as an 
$R[\![u]\!]$-module via the canonical projection $R[\![u]\!]\to R[u]/u^e$.  For 
$i\in\ZZ$ denote by $\alpha_i:\aff{\Lambda}_i\to \res{\Lambda}_{i,R}$ the 
morphism described by the identity matrix with respect to $\aff{\mf E}_i$ and 
$\res{\mf E}_i$. It induces an isomorphism
$\aff{\Lambda}_i/u^e\aff{\Lambda}_i\too{\sim}\res{\Lambda}_{i,R}$.  Clearly the 
following diagrams commute.
\begin{equation*}
    \xymatrix{
        \aff{\Lambda}_i\ar[d]_{\alpha_i}\ar@{}[r]|{\subset}&\aff{\Lambda}_{i+1}\ar[d]^{\alpha_{i+1}}\\
        \res{\Lambda}_{i,R}\ar[r]^{\res{\rho}_{i,R}}&\res{\Lambda}_{i+1,R},
    }\;\;
    \xymatrix{
        \aff{\Lambda}_i\times 
        \aff{\Lambda}_{-i}\ar[r]^-{\aff{\paird}_i}\ar[d]_{\alpha_i\times 
            \alpha_{-i}}& R[\![u]\!]\ar[d]\\
            \res{\Lambda}_{i,R}\times 
            \res{\Lambda}_{-i,R}\ar[r]^-{\res{\paird}_{i,R}}& R[u]/u^e,
    }
    \;\;
    \xymatrix{
        \aff{\Lambda}_i\ar[d]_{\alpha_i}&\aff{\Lambda}_{2n+i}\ar[d]^{\alpha_{2n+i}}\ar[l]_{u\cdot 
        }\\
        \res{\Lambda}_{i,R}&\res{\Lambda}_{2n+i,R}\ar[l]_{\res{\vartheta}_{i,R}}.
    }
\end{equation*}
The following proposition allows us to consider $\Me^{e,n}$ as a subfunctor of 
$\mc F_{\Sp}^{(-e)}$.
\begin{prop}[{\cite[\textsection 11]{pr2}}]\label{PropLocalModelintoFlag}
    There is an embedding $\alpha:\Me^{e,n}\hookrightarrow \mc F_{\Sp}^{(-e)}$ 
    given on
    $R$-valued points by
    \begin{align*}
        \Me^{e,n}(R)&\to \mc F_{\Sp}^{(-e)}(R),\\
        (t_i)_i&\mapsto (\alpha_{i}^{-1}(t_i))_i.
    \end{align*}
    It induces a bijection from $\Me^{e,n}(R)$ onto the set of those $(L_i)_i\in 
    \mc F_{\Sp}^{(-e)}(R)$ satisfying the following conditions for all 
    $i\in\ZZ$.
    \begin{enumerate}
        \item $u^e\aff{\Lambda}_i\subset L_i\subset\aff{\Lambda}_i$.
        \item For all $p\in R[u]/u^e$, we have
            \begin{equation*}
                \chi_R(p|\aff{\Lambda}_i/L_i)=(T-p(0))^{ne}
            \end{equation*}
            in $R[T]$.
            Here $\aff{\Lambda}_i/L_i$ is considered as an $R[u]/u^e$-module 
            using (1).
    \end{enumerate}
\end{prop}
\begin{proof}
    Let $(t_i)_i\in \Me^{e,n}(R)$ and set $(L_i)_i= (\alpha_i^{-1}(t_i))_i$.  It 
    is clear that this defines a periodic lattice chain in $R(\!(u)\!)^{2n}$.  
    Let $i\in\ZZ$. We have \begin{equation}\label{EqRanks}
        \rk_R(\aff{\Lambda}_{i+1}/\aff{\Lambda}_i)+\rk_R(\aff{\Lambda}_i/L_i)=\rk_R(\aff{\Lambda}_{i+1}/L_{i+1})+\rk_R(L_{i+1}/L_i),
    \end{equation}
    as both sides are equal to $\rk_R(\aff{\Lambda}_{i+1}/L_i)$. We conclude 
    from condition 
    \ref{DefnSpecialLocalModel}(\ref{DefnSpecialLocalModel-Determinant})
    that $\rk_R(\aff{\Lambda}_i/L_i)=ne=\rk_R(\aff{\Lambda}_{i+1}/L_{i+1})$.  
    Thus \eqref{EqRanks} amounts to the equation
    $\rk_R(\aff{\Lambda}_{i+1}/\aff{\Lambda}_i)=\rk_R(L_{i+1}/L_i)$, so that the 
    chain $(L_i)_i$ is complete. 
     
    From $\overline{\pair{t_i}{t_{-i}}}_{i,R}=0$ we deduce that 
    $\widetilde{\pair{\mc L_i}{L_{-i}}}\subset u^eR[\![u]\!]$ and hence that 
    $u^{-e}L_{-i}\subset L_i^\vee$.

    From $u^e\aff{\Lambda}_i\subset L_i$ on the other hand we deduce $u^e 
    L_i^\vee\subset \aff{\Lambda}_{-i}$. By definition, we know that 
    $\widetilde{\pair{L_i}{u^e L_i^\vee}}\subset u^eR[\![u]\!]$, which implies 
    $\overline{\pair{t_i}{\alpha_{-i}(u^e L_i^\vee)}}_{i}=0$. Consequently 
    $\alpha_{-i}(u^e L_i^\vee)\subset t_{-i}$, which shows that $u^e L_i^\vee 
    \subset  L_{-i}$. Hence also $L_i^\vee \subset u^{-e} L_{-i}$.
    
    This proves the existence of the map $\alpha$. Its injectivity as well as 
    the characterization of its image are immediate.
\end{proof}
Note that $\res{\mc 
L}=(\res{\Lambda}_{i},\res{\rho}_{i},\res{\vartheta}_{i},\res{\paird}_{i})_i$ is 
a polarized chain of $\FF[u]/u^e$-modules of type $(\aff{\mc L})$. In fact 
$\res{\mc L}=\aff{\mc L}\otimes_{\FF[\![u]\!]}\FF[u]/u^e$. Let $R$ be an 
$\FF$-algebra. There is an obvious action of $\Aut(\res{\mc L})(R[u]/u^e)$ on 
$\Me^{e,n}(R)$, given by
$(\varphi_i)\cdot (t_{i})=(\varphi_i(t_i))$. The canonical morphism 
$R[\![u]\!]\to R[u]/u^e$ induces a morphism $\Aut(\aff{\mc L})(R[\![u]\!])\to 
\Aut(\overline{\mc L})(R[u]/u^e)$ and we thereby extend this $\Aut(\overline{\mc 
L})(R[u]/u^e)$-action on $\Me^{e,n}(R)$ to an $\Aut(\aff{\mc 
L})(R[\![u]\!])$-action.
\begin{lem}\label{LemDifferentAutsSameOrbitsSym}
    Let $R$ be an $\FF$-algebra and let $t\in \Me^{e,n}(R)$. We have 
    \allowbreak$\Aut(\aff{\mc L})(R[\![u]\!])\cdot t=\Aut(\overline{\mc 
    L})(R[u]/u^e)\cdot t$.
\end{lem}
\begin{proof}
    The map $\Aut(\aff{\mc L})(R[\![u]\!])\to \Aut(\overline{\mc L})(R[u]/u^e)$ 
    is surjective by Proposition \ref{PropLiftingIsosSym}.
\end{proof}
Define a subfunctor $I_\Sp=I_{\Sp_{2n}}$ of $\Lf \Sp_{2n}$ by $I_\Sp=\Lf 
\Sp_{2n}\cap I_{\GSp}$.
\begin{lem}\label{LemIvsI0Sym}
    We have $I_\GSp(\FF)=\FF[\![u]\!]^\times I_\Sp(\FF)$.
\end{lem}
\begin{proof}
    Let $g\in I_\GSp(\FF)$. Clearly $c(g)\in \FF[\![u]\!]^\times$. As $\ch 
    \FF\neq 2$, there is an $x\in \FF[\![u]\!]^\times$ with $x^2=c(g)$. Then
    $x^{-1}g\in I_\Sp(\FF)$.
\end{proof}
\begin{lem}\label{LemIvsIprimeSym}
    Let $g\in I_\Sp(\FF)$. Then $g$ restricts to an
    automorphism $g_i:\aff{\Lambda}_i\to\aff{\Lambda}_i$ for each $i\in\ZZ$. The 
    assignment $g\mapsto (g_i)_i$ defines an isomorphism $I_\Sp(\FF)\too{\sim} 
    \Aut(\aff{\mc L})(\FF[\![u]\!])$.
\end{lem}
\begin{proof}
    Analogous to the proof of Lemma \ref{LemIvsIprime}.
\end{proof}
\begin{prop}\label{PropIndicesSym2}
    Let $t\in \Me^{e,n}(\FF)$. Then $\alpha$ induces a
    bijection \begin{equation*}\label{EqOrbits1Sym}
        \Aut(\res{\mc L})(\FF[u]/u^e)\cdot t\too{\sim} I_{\GSp}(\FF)\cdot 
        \alpha(t).
    \end{equation*}
    Consequently we obtain an embedding
    \begin{equation*}
        \Aut(\overline{\mc L})(\FF[u]/u^e)\backslash 
        \Me^{e,n}(\FF)\hookrightarrow I_{\GSp}(\FF)\backslash \mc F_\GSp(\FF).
    \end{equation*}
\end{prop}
\begin{proof}
    The composition $\Me^{e,n}(\FF)\too{\alpha}\mc F^{(-e)}_\Sp(\FF)\subset \mc 
    F_\GSp(\FF)$ is equivariant for the $\Aut(\aff{\mc L})(\FF[\![u]\!])$-action 
    on $\Me^{e,n}(\FF)$, the $I_\Sp(\FF)$-action on $\mc F_\GSp(\FF)$ and the 
    isomorphism $I_\Sp(\FF)\too{\sim} \Aut(\aff{\mc L})(\FF[\![u]\!])$ of Lemma 
    \ref{LemIvsIprimeSym}. It therefore induces a bijection
    $\Aut(\aff{\mc L})(\FF[\![u]\!])\cdot t\too{\sim} I_{\Sp}(\FF)\cdot 
    \alpha(t)$. We conclude by applying Lemmata 
    \ref{LemDifferentAutsSameOrbitsSym}
    and \ref{LemIvsI0Sym}.
\end{proof}
Consider $\alpha':\Me^{e,n}(\FF)\hookrightarrow \mc 
F_\GSp(\FF)\too{\phi(\FF)^{-1}}\Lf {\GSp}(\FF)/I_{\GSp}(\FF)$.
\begin{prop}\label{PropIndicesSym}
    Let $t\in \Me^{e,n}(\FF)$. Then $\alpha'$ induces a
    bijection
    \begin{equation*}
        \Aut(\res{\mc L})(\FF[u]/u^e)\cdot t\too{\sim} I_{\GSp}(\FF)\cdot 
        \alpha'(t).
    \end{equation*}
    Consequently we obtain an embedding
    \begin{equation*}
        \Aut(\overline{\mc L})(\FF[u]/u^e)\backslash 
        \Me^{e,n}(\FF)\hookrightarrow I_{\GSp}(\FF)\backslash 
        \GSp(\FF(\!(u)\!))/I_{\GSp}(\FF).
    \end{equation*}
\end{prop}
\begin{proof}
    Clear from Proposition \ref{PropIndicesSym2}, as the isomorphism $\phi(\FF)$ 
    is in particular $I_{\GSp}(\FF)$-equivariant.
\end{proof}
Let $R$ be an $\FF$-algebra and $(\varphi_i)_i\in \Aut(\mc L)(R)$. The 
decomposition \eqref{EqLambdaInSpecialFiber} induces for each $i$ a 
decomposition of $\varphi_i:\Lambda_{i,R}\too{\sim}\Lambda_{i,R}$ into the 
product of $R[u]/u^e$-linear automorphisms 
$\varphi_{i,\sigma}:\res{\Lambda}_{i,R}\too{\sim}\res{\Lambda}_{i,R}$. The 
following statement is then clear (cf.\ the proof of Proposition 
\ref{PropDecompositionofLocalModel}).
\begin{prop}\label{PropChainMorphismDecomposes}
    Let $R$ be an $\FF$-algebra. The following map is an isomorphism, functorial 
    in $R$.
    \begin{align*}
        \Aut(\mc L)(R)&\to \prod_{\sigma\in\mf S} \Aut(\res{\mc L})(R[u]/u^e),\\
        (\varphi_i)_i&\mapsto ((\varphi_{i,\sigma})_i)_\sigma.
    \end{align*}
\end{prop}
Consider the composition
\begin{equation*}
    \tilde{\alpha}:\Mloc(\FF)\too{\eqref{EqDecompLocModelSym}}
    \prod_{\sigma \in\mf S} \Me^{e,n}(\FF)\too{\prod_\sigma \alpha'}
    \prod_{\sigma\in\mf S} \Lf {\GSp}(\FF)/I_{\GSp}(\FF).
\end{equation*}
For $\sigma\in \mf S$ denote
by $\tilde{\alpha}_\sigma:\Mloc(\FF)\to \Lf {\GSp}(\FF)/I_{\GSp}(\FF)$ the 
corresponding component of $\tilde{\alpha}$.
\begin{thm}\label{ThmIndicesSym}
    Let $t\in \Mloc(\FF)$. Then $\tilde{\alpha}$ induces a bijection
    \begin{equation*}
        \Aut(\mc L)(\FF)\cdot t\too{\sim} \prod_{\sigma\in\mf S} 
        I_{\GSp}(\FF)\cdot \tilde{\alpha}_\sigma(t).
    \end{equation*}
    Consequently we obtain an embedding
    \begin{equation*}
        \Aut(\mc L)(\FF)\backslash \Mloc(\FF)\hookrightarrow \prod_{\sigma\in\mf 
        S} I_{\GSp}(\FF)\backslash \mc \GSp(F(\!(u)\!))/I_{\GSp}(\FF).
    \end{equation*}
\end{thm}
\begin{proof}
    The isomorphism $\Mloc(\FF)\too{\eqref{EqDecompLocModelSym}} 
    \prod_{\sigma\in\mf S} \Me^{e,n}(\FF)$  is equivariant for the $\Aut(\mc 
    L)(\FF)$ action on $\Mloc(\FF)$, the $\prod_{\sigma\in\mf S} \Aut(\res{\mc 
    L})(\FF[u]/u^e)$ action on $\prod_{\sigma\in\mf S} \Me^{e,n}(\FF)$ and the 
    isomorphism of Lemma \ref{PropChainMorphismDecomposes}. The statement thus 
    follows from
    Proposition \ref{PropIndicesSym}.
\end{proof}
\subsection{The extended affine Weyl group}\label{SecExtendedAffineWeyl}
Let $T$ be the maximal torus of diagonal matrices in $\GSp_{2n}$ and let $N$ be 
its normalizer.
We denote by $\widetilde{W}=N(\FF(\!(u)\!))/T(\FF[\![u]\!])$ the extended affine 
Weyl group of $\GSp$ with respect to $T$. Setting
\begin{equation*}
    W=\{w\in S_{2n}\mid \forall i\in\{1,\dots,2n\}:\ w(i)+w(2n+1-i)=2n+1\}
\end{equation*}
and
\begin{equation*}
    X=\{(a_1,\dots,a_{2n})\in \ZZ^{2n}\mid 
    a_1+a_{2n}=a_2+a_{2n-1}=\dots=a_n+a_{n+1}\},
\end{equation*}
the group homomorphism $\upsilon:W\ltimes X\to N(\FF(\!(u)\!)),\ 
(w,\lambda)\mapsto A_wu^\lambda$ induces an isomorphism $W\ltimes X\too{\sim} 
\widetilde{W}$. We use it to identify $\widetilde{W}$ with $W\ltimes X$ and 
consider $\widetilde{W}$ as a subgroup of $\GSp(\FF(\!(u)\!)\!)$ via $\upsilon$.

To avoid any confusion of the product inside $\widetilde{W}$ and the canonical 
action of $S_{2n}$ on $\ZZ^{2n}$, we will always denote the element of 
$\widetilde{W}$ corresponding to $\lambda\in X$ by $u^\lambda$.

Recall from \cite[\textsection 2.5-2.6]{gy1} the notion of an extended alcove 
$(x_i)_{i=0}^{2n-1}$ for $\GSp_{2n}$. Also recall the standard alcove 
$(\omega_i)_{i=0}^{2n-1}$. As in loc.\ cit.\ we identify $\widetilde{W}$ with the set of
extended alcoves by using the standard alcove as a base point.

Write $\mathbf{e}=(e^{(2n)})$.
\begin{defn}[{Cf.\ \cite{kottwitz_rapoport}, \cite[Definition 2.4]{gy1}}]\label{DefnPermSym}
    An extended alcove $(x_i)_{i=0}^{2n-1}$ is called \emph{permissible} 
    if it satisfies the following conditions for all $i\in\{0,\dots,2n-1\}$.
    \begin{enumerate}
        \item $\omega_i\leq x_i \leq \omega_i+\mathbf{e}$, where $\leq$ is to be 
            understood componentwise.
        \item $\sum_{j=1}^{2n} x_i(j)=ne-i$.
    \end{enumerate}
    Denote by $\Perm$ the set of all permissible extended alcoves.
\end{defn}
\begin{prop}
    The inclusion $N(\FF(\!(u)\!))\subset \GSp(\FF(\!(u)\!))$ induces a bijection 
    $\widetilde{W}\too{\sim} I_{\GSp}(\FF)\backslash \mc 
    \GSp(\FF(\!(u)\!))/I_{\GSp}(\FF)$. In other words,
    \begin{equation*}
        \GSp(\FF(\!(u)\!))=\coprod_{x\in\widetilde{W}} I_{\GSp}(\FF)xI_{\GSp}(\FF).
    \end{equation*}
    Under this bijection, the subset 
    \begin{equation*}
        \Aut(\overline{\mc L})(\FF[u]/u^e)\backslash \Me^{e,n}(\FF)\subset 
        I_{\GSp}(\FF)\backslash \GSp(\FF(\!(u)\!))/I_{\GSp}(\FF)
    \end{equation*}
    of Proposition \ref{PropIndicesSym} corresponds to the subset $\Perm\subset 
    \widetilde{W}$.
\end{prop}
\begin{proof}
    The first statement is the well-known Iwahori decomposition. The second 
    statement follows easily from the explicit description of the image of 
    $\alpha$ in Proposition \ref{PropLocalModelintoFlag}.
\end{proof}
\begin{cor}\label{CorIndexSetSym}
    Under the identifications of  Theorem \ref{ThmIndicesSym}, the set 
    \allowbreak$\prod_{\sigma\in\mf S} \Perm$ constitutes a set of representatives of $\Aut(\mc 
    L)(\FF)\backslash \Mloc(\FF)$.
\end{cor}
\begin{remark}\label{RemVvsFSym}
    In the normalization of Corollary \ref{CorIndexSetSym} we have indexed the 
    KR stratification by the relative position of $\mc L\otimes \FF$ 
    to its image under \emph{Verschiebung}, compare Remark \ref{RemarkVvsF}. 
    The normalization of Corollary \ref{CorIndexSetSym} therefore differs from 
    the one of Definition \ref{DefnIndexKR} by the automorphism
    $\prod_{\sigma\in\mf S} \Perm\to \prod_{\sigma\in\mf S} \Perm,\ 
    (x_\sigma)\mapsto (u^ex_\sigma^{-1})$.
\end{remark}
\subsection{The \texorpdfstring{$p$}{p}-rank on a KR stratum}
Recall from Section \ref{SecLocalModelGen} the scheme 
$\mc A/\ZZ_p$ associated with our choice of PEL datum, and the KR stratification 
\begin{equation*}
    \mc A(\FF)=\coprod_{x\in \Aut(\mc L)(\FF)\backslash \Mloc(\FF)}\mc A_x.
\end{equation*}
We have identified the occurring index set with $\prod_{\sigma\in\mf 
S}\Perm$ in Corollary \ref{CorIndexSetSym}. We can then state the following 
result.
\begin{thm}\label{ThmPrank}
    Let $x=(x_\sigma)_\sigma\in \prod_{\sigma\in\mf S} \Perm$. Write 
    $x_\sigma=w_\sigma u^{\lambda_\sigma}$ with $w_\sigma\in W,\ 
    \lambda_\sigma\in X$. Then the $p$-rank on $\mc A_x$ is constant with value
    \begin{equation*}
        g\cdot |\{1\leq i\leq 2n\mid \forall \sigma\in\mf S(w_\sigma(i)=i\wedge \lambda_\sigma(i)=0)\}|.
    \end{equation*}
\end{thm}
\begin{remark}
    For $F=\QQ$, we recover the result \cite[Th\'eor\`eme 4.1]{genestier_ngo} 
    of Ng\^o and Genestier.
\end{remark}
\begin{proof}[Proof of Theorem \ref{ThmPrank}]
    Let $1\leq i\leq 2n$ and $x\in\Perm$. Write 
    $x=wu^\lambda$ with $w\in W,\lambda\in X$.
    By Propositions \ref{ProppRankFirstVersionGen} and
    \ref{PropEquivCondGen} it suffices to show the following equivalence.
    \begin{equation*}
        \aff{\Lambda}_i=x(\aff{\Lambda}_i)+\aff{\Lambda}_{i-1} \Leftrightarrow 
        (w(i)=i\wedge \lambda(i)=0).
    \end{equation*}
    Consider the subset $\mc S=\{u^k e_j\mid k\in\ZZ,\ 1\leq j\leq 2n\}$ of
    $\FF(\!(u)\!)^{2n}$. Then $x$ induces a permutation of $\mc S$, namely 
    $x(u^k e_j)=u^{\lambda(j)+k} e_{w(j)}$. We have $\aff{\Lambda}_{i}\cap \mc 
    S=\aff{\Lambda}_{i-1}\cap \mc S\amalg \{u^{-1}e_i\}$, and $x\in\Perm$ 
    implies
    $x(\aff{\Lambda}_{i-1}\cap \mc S)\subset \aff{\Lambda}_{i-1}\cap \mc S$.  
    Consequently $u^{-1}e_i\in x(\aff{\Lambda}_{i}\cap \mc S)$ if and only if 
    $x(u^{-1}e_i)=u^{-1}e_i$, which in turn is equivalent to $w(i)=i\wedge 
    \lambda(i)=0$, as desired.
\end{proof}
\subsection{The density of the ordinary locus}\label{SecDensOrd}
Denote by $\leq$ and $\ell$ the partial order and the length function on 
$\widetilde{W}$ defined in \cite[\textsection 2.1]{gy1}, respectively. We extend 
$\leq$ and $\ell$ to $\prod_{\sigma\in\mf S} \widetilde{W}$ by setting
$(x_\sigma)_\sigma\leq (x'_\sigma)_\sigma:\Leftrightarrow (\forall \sigma\in\mf 
S:\ x_\sigma\leq x'_\sigma)$ and $\ell((x_\sigma)_\sigma)=\sum_\sigma\ell(x_\sigma)$.
\begin{lem}\label{LemPropsOfKR}
    Let $x\in \prod_{\sigma\in\mf S}\Perm$. The smooth $\FF$-variety $\mc A_{x}$ is
    equidimensional of dimension $\ell(x)$.  Furthermore the closure 
    $\overline{\mc A}_{x}$ of $\mc A_{x}$ in $\mc A_{\FF}$ is given 
    by 
    \begin{equation}\label{EqClosure}
        \overline{\mc A}_{x}=\coprod_{\substack{y\leq x}} \mc A_{y}.
    \end{equation}
\end{lem}
\begin{proof}
    As in Remark \ref{RemarkKRasScheme} there is, \'etale locally on $\mc A_\FF$, 
    an \'etale morphism $\beta:\mc A_\FF\to \Mloc_\FF$ with $\mc 
    A_{x}=\beta^{-1}(\Mloc_x)$. In Theorem \ref{ThmIndicesSym} we have identified $\Mloc_x$ with the 
    Schubert cell $\mc C_x\subset \prod_{\sigma\in\mf 
    S}\GSp_{2n}(\FF(\!(u)\!))/I_{\GSp}(\FF)$ corresponding to $x$. The 
    statements therefore follow from well-known properties of Schubert cells 
    once we know that all KR strata are
    non-empty. This is true in the Siegel case by Genestier's result \cite[Proposition 
    1.3.2]{genestier_semi}, in the case that $p$ is unramified in $F$ by the 
    result \cite[Theorem 2.5.2(1)]{goren} of Goren and Kassaei, and in the 
    ramified case by a yet to be published result of Yu \cite{fakeyu}.
\end{proof}
Our next goal is to generalize the result \cite[Corollaire 4.3]{genestier_ngo} 
of Ng\^o and Genestier on the density of the ordinary locus. Denote by $\mc 
A^{(ng)} \subset \mc A(\FF)$ the subset where the $p$-rank of the underlying 
abelian variety is equal to $ng$. By the determinant condition imposed in the 
definition of $\mc A$ this is precisely the ordinary locus in $\mc A(\FF)$.
\begin{cor}\label{CorDense}
    The ordinary locus $\mc A^{(ng)}$ is dense in $\mc A(\FF)$ if and only if 
    $p$ is totally ramified in $F$.
\end{cor}
\begin{proof}
    Let $\mu=(e^{(g)},0^{(g)})\in X$.
    Our subset $\Perm\subset \widetilde{W}$ is precisely the set denoted by 
    $\Perm(\mu)$ in \cite{haines_ngo}. By \cite[Theorem 10.1]{haines_ngo}, we 
    have \begin{equation}\label{EqPermAdm}
        \Perm(\mu)=\Adm,
    \end{equation}
    where $\Adm:=\{x\in \widetilde{W}\mid \exists w\in W:\ x\leq u^{w(\mu)}\}$. 

    Write $\mf M=\{x\in \widetilde{W}\mid \exists w\in W:\ x= u^{w(\mu)}\}$.  
    Then \eqref{EqPermAdm} implies that $\prod_{\sigma\in\mf S}\mf M$ is 
    precisely the subset of maximal elements for $\leq$ in $\prod_{\sigma\in\mf 
    S}\Perm$. Denote by $\Delta_{\mf M}\subset \prod_{\sigma\in\mf S}\mf M$ the 
    diagonal. By Theorem \ref{ThmPrank} we have $\mc 
    A^{(ng)}=\coprod_{x\in\Delta_{\mf M}}\mc A_x$. 
    The statement therefore follows from \eqref{EqClosure} by noting that
    $\Delta_\mf M=\prod_{\sigma\in\mf S}\mf M$ if and only if $p$ is totally 
    ramified in $F$.
\end{proof}
\subsection{An explicit example: Hilbert-Blumenthal modular 
varieties}\label{SecHilbBlum}
In this section we use the explicit case of the Hilbert-Blumenthal modular 
varieties to illustrate how Theorem \ref{ThmPrank} and the KR
stratification in general yield results about the geometry of the moduli spaces 
$\mc A$. We also compare these results to some of those obtained by Stamm in 
\cite{stamm}.

Assume from now on that $p$ is \emph{inert} in $\mc O_F$, so that we have $e=1$ 
and $f=g$. Assume also that $\dim_F V=2$, so that $n=1$.

Let us start with a discussion of the index set $\prod_{\sigma\in \mf S} \Perm$ 
of the KR stratification. From
Definition \ref{DefnPermSym} one immediately obtains that the subset 
$\Perm\subset\widetilde{W}$ is given by 
$\Perm=\{u^{(1,0)},u^{(0,1)},(1,2)u^{(1,0)}\}$. To put this set into a group 
theoretic perspective, we recall the setup described in
\cite[\textsection 2.1]{gy1} in this easy special case. Consider the elements 
$\tau=(1,2)u^{(1,0)}, s_1=(1,2)$ and $s_0=(1,2)u^{(1,-1)}$ of $\widetilde{W}$.  
The subgroup $W_a$ of $\widetilde{W}$ generated by $s_0$ and $s_1$ is a Coxeter 
group on the generators $s_0$ and $s_1$, and we denote by $\leq$ and $\ell$ the 
corresponding Bruhat order and length function on $W_a$, respectively. Denoting 
by $\Omega$ the cyclic subgroup of $\widetilde{W}$ generated by $\tau$, we have 
$\widetilde{W}=W_a\rtimes \Omega$. The extension of $\leq$ and $\ell$ to 
$\widetilde{W}$ is given by $w'\tau'\leq w''\tau'':\Leftrightarrow (w'\leq 
w''\wedge \tau'=\tau'')$ and $\ell(w'\tau'):=\ell(w')$, for $w',w''\in W_a$ and 
$\tau',\tau''\in \Omega$.  We extend $\leq$ and $\ell$ to $\prod_{\sigma\in\mf 
S} \widetilde{W}$ as in Section \ref{SecDensOrd}.

We see that 
\begin{equation*}
    \Perm=\{s_1\tau, s_0\tau, \tau\}\subset W_a\tau.
\end{equation*}
The Bruhat 
order on $\Perm$ is determined by the non-trivial relations $\tau\leq 
s_1\tau$ and $\tau\leq s_0\tau$, while the length function on $\Perm$ is 
given by $\ell(\tau)=0$ and $\ell(s_1\tau)=\ell(s_0\tau)=1$.

Let us state Theorem \ref{ThmPrank} in this special case. Denote by $\mc 
A^{(0)} \subset \mc A(\FF)$ and $\mc A^{(g)} \subset \mc 
A(\FF)$ the subsets where the $p$-rank of the underlying abelian variety 
is equal to $0$ and $g$, respectively.
\begin{prop}\label{PropPrankHilbert}
    We have 
    \begin{equation*}
        \mc A(\FF)=\mc A^{(0)} \amalg \mc A^{(g)}. 
    \end{equation*}
    The ordinary locus $\mc A^{(g)}$ is the union of only two KR strata, namely
    those corresponding to the elements $((s_1\tau)^{(g)})=(s_1\tau,s_1\tau,\dots,s_1\tau)$ and 
    $((s_0\tau)^{(g)})=(s_0\tau,s_0\tau,\dots,s_0\tau)$
    of $\prod_{\sigma\in\mf S}\Perm$. The $p$-rank on all other KR strata is 
    equal to 0.
\end{prop}
\begin{lem}
    The maximal elements in $\prod_{\sigma\in\mf S}\Perm$ for the Bruhat order 
    are precisely the elements  of length $2^g$ in $\prod_{\sigma\in\mf 
    S}\Perm$. The set of these maximal elements is given by $\prod_{\sigma\in\mf 
    S}\{s_1\tau,s_0\tau\}$.\qed
\end{lem}
From the preceding results, we obtain without any additional work the following 
theorem. 
\begin{thm}\label{TheoremHB}
    Let $g\geq 2$. Then
    \begin{equation*}
        \mc A_{\FF}=\overline{\mc A}_{((s_1\tau)^{(g)})}\cup \overline{\mc A}_{((s_0\tau)^{(g)})}\cup \mc A^{(0)}.
    \end{equation*}
    Each of $\overline{\mc A}_{((s_1\tau)^{(g)})}, 
    \overline{\mc A}_{((s_0\tau)^{(g)})}$ and $\mc A^{(0)}$
    is equidimensional of dimension $2^g$, and hence so is $\mc A_{\FF}$.

    More precisely, $\mc A^{(0)}$ is the union
    \begin{equation*}
        \mc A^{(0)}=\bigcup_{\substack{x\in \prod_{\sigma\in\mf S}\{s_1\tau,s_0\tau\}\\
        x\neq ((s_1\tau)^{(g)}),((s_0\tau)^{(g)})}}
        \overline{\mc A}_{x}
    \end{equation*}
    of $2^g-2$ closed subsets, all equidimensional of dimension $2^g$.

    Furthermore, we have
    \begin{equation*}
        \overline{\mc A}_{((s_1\tau)^{(g)})}\cap \overline{\mc A}_{((s_0\tau)^{(g)})}\subset \mc A^{(0)}.
    \end{equation*}
\end{thm}
Taking $g=2$, we recover \cite[Theorem 2 (p.\ 408)]{stamm}. Note that for $g=2$, 
the set $\mc A^{(0)}$ is precisely the supersingular locus in $\mc 
A_{\FF}$, because a 2-dimensional abelian variety is supersingular if and only if 
its $p$-rank is equal to zero.
\section{The unitary PEL datum}\label{SecPELUni}
Let $n\in\NN_{\geq 1}$. In Sections \ref{SecUni} through \ref{SecUni3} we will be concerned with the 
PEL datum consisting of the following objects.
\begin{enumerate}
	\item An imaginary quadratic extension $F/F_0$ of a totally real extension $F_0/\QQ$. Let
        $g_0=[F_0:\QQ]$ and $g=[F:\QQ]$, so that $g=2g_0$.	
	\item The non-trivial element $\inv$ of $\Gal(F/F_0)$.
	\item An $n$-dimensional $F$-vector space $V$.
    \item The symplectic form $\pairtd:V\times V\to \QQ$ on the underlying 
        $\QQ$-vector space of $V$ constructed as follows: 
        Fix once and for all a $\inv$-skew-hermitian form $\pairtd':V\times V\to 
        F$ (i.e.\ $\pairt{av}{bw}'=ab^\inv\pairt{v}{w}'$ and 
        $\pairt{v}{w}'=-\pairt{w}{v}'^\inv$ for $v,w\in V,\ a,b\in F$).
        Define $\pairtd=\tr_{F/\QQ}\circ \pairtd'$.
    \item The element $J\in \End_{B\otimes \RR}(V\otimes \RR)$ to be 
        defined separately in each case, see Sections \ref{SecPELUni1}, 
        \ref{SecPELUni2} and \ref{SecPELUni3}.
\end{enumerate}
\begin{remark}\label{RemGUni1}
    Denote by $\GU_{\pairtd'}$ the $F_0$-group given on $R$-valued points by
        $\GU_{\pairtd'}(R)=\{g\in \GL_{F\otimes_{F_0} R}(V\otimes_{F_0} R)\mid \exists c=c(g)\in R^\times
        \forall x,y\in V\otimes_{F_0} R:\ \pairt{gx}{gy}'_{R}=c\pairt{x}{y}'_{R}
\}$.
    Then the reductive $\QQ$-group $G$ associated with the above PEL datum fits 
    into the following cartesian diagram.
    \begin{equation*}
        \xymatrix{
            G\ar[d]_c \xyhookrightarrow & \Res_{F_0/\QQ} \GU_{\pairtd'}\ar[d]^c\\ 
            \GG_{m,\QQ}\xyhookrightarrow& \Res_{F_0/\QQ}\GG_{m,F_0}.
        }
    \end{equation*}
\end{remark}
\section{The ramified unitary case}\label{SecUni}
\subsection{The PEL datum}\label{SecPELUni1}
We start with the PEL datum defined in Section \ref{SecPELUni}. We assume that 
$p\mc O_{F_0}=(\mc P_0)^{e_0}$ for a single prime $\mc P_0$ of $\mc O_{F_0}$ and 
that $\mc P_0\mc O_F=\mc P^2$ for a prime $\mc P$ of $\mc O_F$.  Write $e=2e_0$, 
so that $p\mc O_F=\mc P^e$. Denote by $f=[k_{\mc P_0}:\FF_p]$ the corresponding 
inertia degree, so that $g=ef$ and $g_0=fe_0$.  Fix once and for all 
uniformizers $\pi_0$ of $\mc O_{F_0}\otimes\ZZ_{(p)}$ and $\pi$ of $\mc 
O_{F}\otimes\ZZ_{(p)}$, satisfying $\pi^2=\pi_0$. We have
$\pi^\inv=-\pi$.

For typographical reasons, we denote the ring of integers in $F_\mc P$ by 
$\mc O_{\mc P}$ and the ring of integers in $(F_0)_{\mc P_0}$ by $\mc O_{\mc 
P_0}$. Denote by $\mf C=\mf C_{\mc O_{\mc P}\mid \ZZ_p},\ \mf C_0=\mf C_{\mc O_{\mc P_0}\mid \ZZ_p}$ and 
$\mf C'=\mf C_{\mc O_{\mc P}\mid \mc O_{\mc P_0}}$ the corresponding inverse 
differents. Then $\mf C_0=(\pi_0^{-k})$ for some $k\in\NN$.  The extension 
$\fp/\fpn$ is tamely ramified, so that $\mf C'=(\pi^{-1})$. The equality $\mf 
C=\mf C'\cdot \mf C_0$ then implies that $\mf C=(\pi^{-2k-1})$ and we denote by 
$\delta=\pi^{-2k-1}$ the corresponding generator of $\mf C$. It satisfies 
$\delta^\inv=-\delta$.  Consequently the form 
$\delta^{-1}\pairtd'_{\QQ_p}:V_{\QQ_p}\times V_{\QQ_p}\to \fp$ is $\inv$-hermitian and 
we assume that it \emph{splits}, i.e.\ that there is a basis $(e_1,\dots,e_n)$ 
of $V_{\QQ_p}$ over $\fp$ such that 
$\pairt{e_i}{e_{n+1-j}}_{\QQ_p}'=\delta\delta_{ij}$ for $1\leq i,j\leq n$.

Let $0\leq i<n$. We denote by $\Lambda_i$ the $\mc O_{\mc P}$-lattice in 
$V_{\QQ_p}$ with basis
\begin{equation*}
    \mf E_i=(\pi^{-1}e_1,\ldots,\pi^{-1}e_i,e_{i+1},\ldots,e_{n}).
\end{equation*}
For $k\in\ZZ$ we further define $\Lambda_{nk+i}=\pi^{-k}\Lambda_i$ and we denote 
by $\mf E_{nk+i}$ the corresponding basis obtained from $\mf E_i$.
Then $\mc L=(\Lambda_i)_i$ is a complete chain of $\mc O_{\mc P}$-lattices in 
$V_{\QQ_p}$. For $i\in\ZZ$, the dual lattice 
$\Lambda_i^\vee:=\{x\in V_{\QQ_p}\mid \pairt{x}{\Lambda_i}_{\QQ_p}\subset\ZZ_p\}$ of $\Lambda_i$ is given by 
$\Lambda_{-i}$.  Consequently the chain $\mc L$ is self-dual. 
    
Let $i\in\ZZ$. We denote by $\rho_i:\Lambda_{i}\to \Lambda_{i+1}$ the 
inclusion, by $\vartheta_i:\Lambda_{n+i}\to \Lambda_i$ the isomorphism given by 
multiplication with $\pi$ and by $\pairtd_i:\Lambda_i\times \Lambda_{-i}\to 
\ZZ_p$ the restriction of $\pairtd_{\QQ_p}$. Then
$(\Lambda_{i},\rho_{i},\vartheta_{i},\pairtd_{i})_i$ is a polarized chain of $\mc 
O_{\fp}$-modules of type $(\mc L)$, which, by abuse of notation, we also 
denote by $\mc L$.

Denote by $\paird_i:\Lambda_i\times \Lambda_{-i}\to \mc O_{\mc P}$ the restriction of the $\inv$-hermitian form 
$\delta^{-1}\pairtd'_{\QQ_p}$, and by $H_i$
the matrix describing $\paird_i$ with respect to $\mf E_i$ and $\mf E_{-i}$. We 
have
\begin{equation}
    \label{EqDescribingMatrixUni}
    H_i=\text{anti-diag}((-1)^{a_{i,1}},\dots,(-1)^{a_{i,n}})
\end{equation}
for some $a_{i,1},\dots,a_{i,n}\in\ZZ/2\ZZ$.

Denote by $\Sigma_0$ the set of all embeddings $F_0\hookrightarrow\RR$ and by 
$\Sigma$ the set of all embeddings $F\hookrightarrow \CC$. For each
$\sigma\in\Sigma_0$, we denote by 
$\tau_{\sigma,1},\tau_{\sigma,2}\in\Sigma$ the two embeddings with 
$\tau_{\sigma,j}\big|_{F_0}=\sigma$. Of course we have 
$\tau_{\sigma,2}=\tau_{\sigma,1}\circ \inv$.

We obtain isomorphisms
\begin{align}
	\label{FRdecomp}
	F\otimes_\QQ \RR &= \prod_{\sigma\in\Sigma_0}\CC, &F\ni x&\mapsto (\tau_{\sigma,1}(x))_\sigma,\\
	\label{FCdecomp}
	F\otimes_\QQ \CC &= \prod_{\sigma\in\Sigma_0}\CC\times\CC,&F\ni x&\mapsto (\tau_{\sigma,1}(x),\tau_{\sigma,2}(x))_\sigma
\end{align} 
of $\RR$- and $\CC$-algebras, respectively.

The isomorphism \eqref{FRdecomp} induces a decomposition $V\otimes\RR = 
\prod_{\sigma\in\Sigma_0} V_\sigma$ into $\CC$-vector spaces $V_\sigma$ and $\pairtd'_\RR$ decomposes 
into the product of skew-hermitian forms $\pairtd'_\sigma:V_\sigma\times 
V_\sigma\to \CC,\ \sigma\in \Sigma_0$. For each $\sigma\in \Sigma_0$, there are 
$r_\sigma,s_\sigma\in\NN$ with $r_\sigma+s_\sigma=n$ and a basis $\mf B_\sigma$ 
of $V_\sigma$ over $\CC$ such that $\pairtd'_\sigma$ is described by the matrix 
$D_\sigma=\diag(i^{(r_\sigma}),(-i)^{(s_\sigma)})$ with respect to $\mf 
B_\sigma$. Denote by $J_\sigma$ the endomorphism of $V_\sigma$ described by the 
matrix $D_\sigma$ with respect to $\mf B_\sigma$. We complete the description of 
the PEL datum by defining $J:=\prod_{\sigma\in \Sigma_0}J_\sigma\in 
\End_{B\otimes \RR}(V\otimes \RR)$.

\subsection{The determinant morphism}\label{SecDetMorUni}
The isomorphism \eqref{FCdecomp} induces a decomposition $V\otimes\CC = 
\prod_{\sigma\in\Sigma_0} (V_{\tau_{\sigma,1}}\times V_{\tau_{\sigma,2}})$ into $\CC$-vector 
spaces $V_{\tau_{\sigma,j}}$. The basis $\mf B_\sigma$ of $V_\sigma$ induces 
bases $\mf B_{\tau_{\sigma,j}}$ of $V_{\tau_{\sigma,j}}$ over $\CC$, and the 
endomorphism $J_{\sigma,\CC}$ decomposes into the product of endomorphisms 
$J_{\tau_{\sigma,j}}$ of $V_{\tau_{\sigma,j}}$. We find that 
$J_{\tau_{\sigma,1}}$ is described by the matrix $D_\sigma$ with respect to 
$\mf B_{\tau_{\sigma,1}}$, while $J_{\tau_{\sigma,2}}$ is described by the matrix 
$-D_\sigma$ with respect to $\mf B_{\tau_{\sigma,2}}$.

Denote by $V_{-i}$ the $(-i)$-eigenspace of $J_\CC$. From the explicit 
description of the $J_{\tau_{\sigma,j}}$ with respect to the $\mf 
B_{\tau_{\sigma,j}}$, one concludes that $V_{-i}$ is the $\mc 
O_F\otimes\CC$-module corresponding to the
$\prod_{\sigma\in\Sigma_0}\CC\times\CC$-module 
$\prod_{\sigma\in\Sigma_0}\CC^{s_\sigma}\times \CC^{r_\sigma}$ under 
\eqref{FCdecomp}.

Let $E'$ be the Galois closure of $F$ inside $\CC$ and choose a prime $\mc Q'$ of $E'$ over $\mc 
P$. In absolute analogy to \eqref{FCdecomp}, we have a decomposition
\begin{equation}\label{EqFEDecomp}
	F\otimes_\QQ E'=\prod_{\sigma\in \Sigma_0}E'\times E'.
\end{equation}
Let $M$ be the $\mc O_F\otimes E'$-module corresponding to the 
$\prod_{\sigma\in \Sigma_0}E'\times E'$-module 
$\prod_{\sigma\in\Sigma_0}(E')^{s_\sigma}\times (E')^{r_\sigma}$ 
under \eqref{EqFEDecomp}. From the present discussion we obtain 
an identification $M\otimes_{E'}\CC=V_{-i}$ of $\mc O_F\otimes\CC$-modules. 
Let $\mf B$ be a basis of $M$ over $E'$ and denote by $M_0$ the $\mc O_F\otimes 
\mc O_{E'}$-module generated by $\mf B$.
Then $M_0$ is an $\mc O_F\otimes 
\mc O_{E'}$-stable $\mc O_{E'}$-lattice $M_0$ in $M$.  In particular, the 
morphism $\det_{V_{-i}}:V_{\mc O_F\otimes \CC}\to \AF^1_\CC$ descends to the 
morphism $\det_{M_0}:V_{\mc O_F\otimes \mc O_{E'}}\to \AF^1_{\mc O_{E'}}$.
\subsection{The special fiber of the determinant morphism}
We write $\mf S=\Gal(k_{\mc P}/\FF_p)=\Gal(k_{\mc P_0}/\FF_p)$. We fix once and 
for all an embedding $\iota_{\mc Q'}:k_{\mc Q'}\hookrightarrow \FF$. We consider 
$\FF$ as an $\mc O_{E'}$-algebra via the composition 
$\mc O_{E'}\too{\rho_{\mc Q'}}k_{\mc Q'}\overset{\iota_{\mc Q'}}{\hookrightarrow} \FF$. Also $\iota_{\mc Q'}$ 
induces an embedding $\iota_{\mc P}:k_{\mc \mc P}\hookrightarrow \FF$ and thereby an 
identification of the set of all embeddings $k_{\mc P}\hookrightarrow \FF$ with 
$\mf S$. Our choice of uniformizer $\pi$ induces a canonical isomorphism
\begin{align}
    \label{decompspecialfiberfirstUni}
    \mc O_{F}\otimes\FF&=\prod_{\sigma\in\mf S} \FF[u]/(u^e).
\end{align}
\begin{prop}\label{PropDeterminantSpecialFiberUni}
    Let $x\in \mc O_F$ and let $(p_\sigma)_\sigma\in \prod_{\sigma\in\mf S} \FF[u]/(u^e)$ be the element corresponding to $x\otimes 1$
    under \eqref{decompspecialfiberfirstUni}.
    Then
    \begin{equation*}
        \chi_\FF(x| M_0\otimes_{\mc O_{E'}}\FF)=\prod_{\sigma\in\mf S}\bigl(T-p_\sigma(0)\bigr)^{ne_0}
    \end{equation*}
    in $\FF[T]$.
\end{prop}
\begin{proof}
    Reduce $\chi_{\mc O_{E'}}(x|M_0)$ modulo $\mc Q'$, using Lemma \ref{LemResidueMapinNonGaloisCase}.
\end{proof}
Denote by $E=\QQ(\tr_\CC(x\otimes 1|V_{-i});\ x\in F)$ the reflex field and define $\mc Q=\mc Q'\cap \mc O_E$.
The morphism $\det_{V_{-i}}$ is defined over $\mc O_E$.
\subsection{The local model}
For the chosen PEL datum, Definition \ref{DefnLocalModelGen} amounts to the 
following.
\begin{defn}\label{DefnLocalModelUni}
    The local model $\Mloc$ is the functor on the category of
    $\mc O_{E_\mc Q}$-algebras with $\Mloc(R)$ the set of tuples $(t_i)_{i\in 
    \ZZ}$ of $\mc O_{F}\otimes R$-submodules $t_i\subset \Lambda_{i,R}$
    satisfying the following conditions for all $i\in\ZZ$.
    \renewcommand\theenumi {\alph{enumi}}
    \begin{enumerate}
        \item\label{DefnLocalModelUni-Functoriality}
            $\rho_{i,R}(t_i)\subset t_{i+1}$.
        \item\label{DefnLocalModelUni-Projectivity}
            The quotient $\Lambda_{i,R}/t_i$ is a finite locally free 
            $R$-module.
        \item\label{DefnLocalModelUni-Determinant} We have an equality
            \begin{equation*}
                \det_{\Lambda_{i,R}/t_i}=\det_{V_{-i}}\otimes_{\mc O_E}R
            \end{equation*}
            of morphisms $V_{\mc O_{F}\otimes R}\to \AF^1_R$.
        \item\label{DefnLocalModelUni-DualityCondition} Under the pairing 
            $\pairtd_{i,R}:\Lambda_{i,R}\times \Lambda_{-i,R}\to R$, the
            submodules $t_i$ and $t_{-i}$ pair to zero.
        \item\label{DefnLocalModelUni-Periodicity} $\vartheta_i(t_{n+i})=t_i$.
    \end{enumerate}
\end{defn}
\subsection{The special fiber of the local model}
For $i\in\ZZ$, denote by $\res{\Lambda}_i$ the free $\FF[u]/u^e$-module 
$(\FF[u]/u^e)^{n}$ and by $\res{\mf E}_i$ its canonical basis. Consider the 
$\FF$-automorphism $\res{\inv}:\FF[u]/u^e\to \FF[u]/u^e,\ u\mapsto -u$. Denote 
by $\res{\paird}_i:\res{\Lambda}_i\times \res{\Lambda}_{-i}\to \FF[u]/u^e$ the 
$\res{\inv}$-sesquilinear form described by the matrix $H_i$ of 
\eqref{EqDescribingMatrixUni} with respect to $\res{\mf E}_i$ and $\res{\mf 
E}_{-i}$. Denote by $\res{\vartheta}_i:\res{\Lambda}_{n+i}\to \res{\Lambda}_i$ 
the identity morphism. For $k\in\ZZ$ and $0\leq i<n$, let 
$\res{\rho}_{nk+i}:\res{\Lambda}_{nk+i}\to \res{\Lambda}_{nk+i+1}$ be the 
morphism described by the matrix $\mathrm{diag}(1^{(i)},u,1^{(n-i-1)})$ with 
respect to $\res{\mf E}_{nk+i}$ and $\res{\mf E}_{nk+i+1}$.
\begin{defn}\label{DefnSpecialLocalModelUni}
    Define a functor $\Me^{e,n}$ on the category of
    $\FF$-algebras with $\Me^{e,n}(R)$ the set of tuples $(t_i)_{i\in \ZZ}$ of
    $R[u]/u^e$-submodules $t_i\subset \res{\Lambda}_{i,R}$
    satisfying the following conditions for all $i\in\ZZ$.
    \renewcommand\theenumi {\alph{enumi}}
    \begin{enumerate}
        \item $\res{\rho}_{i,R}(t_i)\subset t_{i+1}$.
        \item The quotient $\res{\Lambda}_{i,R}/t_i$ is finite locally free over 
            $R$.
        \item\label{DefnSpecialLocalModelUni-Determinant}
            For all $p\in R[u]/u^e$, we have
            \begin{equation*}
                \chi_R(p|\res{\Lambda}_{i,R}/t_i)=\bigl(T-p(0)\bigr)^{ne_0}
            \end{equation*}
            in $R[T]$.
        \item $t_i^{\perp,\res{\paird}_{i,R}}=t_{-i}$.
        \item $\res{\vartheta}_i(t_{n+i})=t_i$.
    \end{enumerate}
\end{defn}
Let $i\in\ZZ$. From \eqref{decompspecialfiberfirstUni}
we obtain an isomorphism
\begin{equation}\label{LambdaInSpecialFiberUni}
    \Lambda_{i,\FF}=\prod_{\sigma\in\mf S}\res{\Lambda}_{i}
\end{equation}
by identifying the basis $\mf E_{i,\FF}$ with the product of the bases $\res{\mf 
E}_i$. Under this identification, the morphism $\rho_{i,\FF}$ decomposes into 
the morphisms $\res{\rho}_i$, the pairing $\paird_{i,\FF}$ decomposes into the 
pairings $\res{\paird}_i$ and the morphism $\vartheta_{i,\FF}$ decomposes into 
the morphisms $\res{\vartheta}_i$.

Let $R$ be an $\FF$-algebra and let $(t_i)_{i\in\ZZ}$ be a tuple of $\mc O_F\otimes R$-submodules 
$t_i\subset \Lambda_{i,R}$.  Then \eqref{LambdaInSpecialFiberUni} induces 
decompositions
$t_i=\prod_{\sigma\in\mf S} t_{i,\sigma}$ into $R[u]/u^e$-submodules 
$t_{i,\sigma}\subset\res{\Lambda}_{i,R}$. The following statement is then clear 
(cf.\ the proof of Proposition \ref{PropDecompositionofLocalModel}).
\begin{prop}\label{PropDecompositionofLocalModelUni}
    The morphism $\Mloc_\FF\to \prod_{\sigma\in\mf S} \Me^{e,n}$ given on 
    $R$-valued points by
    \begin{equation}
        \label{EqDecompLocModelUni}
        \begin{aligned}
            \Mloc_\FF(R)&\to \prod_{\sigma\in\mf S} \Me^{e,n}(R),\\
            (t_i)&\mapsto \left((t_{i,\sigma})_i\right)_\sigma
        \end{aligned}
    \end{equation}
    is an isomorphism of functors on the category of $\FF$-algebras.
\end{prop}
\subsection{The affine flag variety}
This section deals with the affine flag variety for the ramified unitary group. 
Our discussion is based on and has greatly profited from \cite{pr}, \cite{pr3}
and \cite{smithling_unitary_odd}, \cite{smithling_unitary_even}.

Let $R$ be an $\FF$-algebra. Consider the extension $R[\![u]\!]/R[\![u_0]\!]$ with 
$u_0=u^2$. Also consider the $R(\!(u_0)\!)$-automorphism 
$\aff{\inv}:R(\!(u)\!)\to R(\!(u)\!),\ u\mapsto -u$. Let $\aff{\paird}$ be the 
$\aff{\inv}$-hermitian form on $R(\!(u)\!)^{n}$ described by the matrix 
$\widetilde{I}_n$ with respect to the standard basis of $R(\!(u)\!)^{n}$ over 
$R(\!(u)\!)$. 
For a lattice $\Lambda$ in $R(\!(u)\!)^n$ we define 
$\Lambda^\vee:=\{x\in R(\!(u)\!)^{n}\mid \aff{\pair{x}{\Lambda}}\subset R[\![u]\!]\}$. Recall 
from Section \ref{SecLattices} the standard lattice chain $\aff{\mc 
L}=(\aff{\Lambda}_i)_i$ in 
$R(\!(u)\!)^{n}$. Note that
$(\aff{\Lambda}_i)^\vee=\aff{\Lambda}_{-i}$ for all $i\in\ZZ$. We denote by 
$\aff{\paird}_i:\aff{\Lambda}_i\times \aff{\Lambda}_{-i}\to R[\![u]\!]$ the 
restriction of $\aff{\paird}$. It is the perfect $\aff{\inv}$-sesquilinear 
pairing described by the matrix $H_i$ of \eqref{EqDescribingMatrixUni} with 
respect to $\aff{\mf E}_i$ and $\aff{\mf E}_{-i}$.

In complete analogy with \cite[Definition A.41]{rz}, we have for an 
$\FF[\![u_0]\!]$-algebra $R$ the notion of a polarized chain 
$\mc M=(M_i,\varrho_i:M_i\to M_{i+1},\theta_i:M_{n+i}\too{\sim} M_i,\mc E_i:M_i\times M_{-i}\to\FF[\![u]\!]\otimes_{\FF[\![u_0]\!]} R)_{i\in\ZZ}$ of $\FF[\![u]\!]\otimes_{\FF[\![u_0]\!]} 
R$-modules of type $(\aff{\mc L})$ (cf.\ \cite[Definition 6.6.1]{diss}).
The proof of \cite[Proposition A.43]{rz} then carries over without any changes 
to show the following result.
\begin{prop}\label{PropRZbigUni}
    Let $R$ be an $\FF[\![u_0]\!]$-algebra such that the image of $u_0$ in $R$ is 
    nilpotent. Then any two polarized chains $\mc M,\mc N$ of $\FF[\![u]\!]\otimes_{\FF[\![u_0]\!]} 
    R$-modules of type $(\aff{\mc L})$ are isomorphic locally for the \'etale 
    topology on $R$. Furthermore the functor $\Isom(\mc M,\mc N)$ is representable by 
    a smooth affine scheme over $R$.
\end{prop}
\begin{prop}\label{PropLiftingIsosUni}
    Let $R$ be an $\FF$-algebra and let $\mc M,\mc N$ be polarized chains of 
    $\FF[\![u]\!]\otimes_{\FF[\![u_0]\!]} R[\![u_0]\!]$-modules of type 
    $(\aff{\mc L})$.  Then the canonical map $\Isom(\mc M,\mc N)(R[\![u_0]\!])\to
    \Isom(\mc M,\mc N)(R[\![u_0]\!]/u_0^m)$ is surjective for all $m\in\NN_{\geq 1}$. In 
    particular $\mc M$ and $\mc N$ are isomorphic locally for the \'etale 
    topology on $R$.
\end{prop}
\begin{proof}
    Analogous to the proof of Proposition \ref{PropLiftingIsosSym}.
\end{proof}
Denote by $\U=\U_n$ and $\GU=\GU_n$ the $\FF(\!(u_0)\!)$-groups with
\begin{equation*}
    \U(R)=\{g\in \GL_{n}(K\otimes_{K_0} R)\mid g^t\widetilde{I}_n g^{\aff{\inv}}=\widetilde{I}_n\}
\end{equation*}
and
\begin{align*}
    \GU(R)=\left\{g\in \GL_{n}(K\otimes_{K_0} R)\mid \exists c=c(g)\in R^\times:\ g^t\widetilde{I}_n g^{\aff{\inv}}=c\widetilde{I}_n\right\}.
\end{align*}
\begin{defn}[{\cite[\textsection 4.2]{smithling_unitary_odd}, \cite[\textsection 6.2]{smithling_unitary_even}}]
    Let $R$ be an $\FF$-algebra and let $(L_i)_i$ be a lattice chain in
    $R(\!(u)\!)^{n}$.
    \begin{enumerate}
        \item Let $r\in\ZZ$. The chain $(L_i)_i$ is called \emph{$r$-self-dual} 
            if 
            \begin{equation*}
                \forall i\in\ZZ:\ L_i^\vee=u_0^r L_{-i}.  
            \end{equation*}
            Denote by $\mc F^{(r)}_\U$ the functor on the 
            category of $\FF$-algebras with $\mc F_{\U}^{(r)}(R)$ the set of
            $r$-self-dual lattice chains in $R(\!(u)\!)^{n}$.
        \item The chain $(L_i)_i$ is called \emph{self-dual} if Zariski locally on $R$ there is an $a\in 
            R(\!(u_0)\!)^\times$ such that 
            \begin{equation}\label{EquSelfdualUni}
                \forall i\in\ZZ:\ L_i^\vee=a L_{-i}.  
            \end{equation}
            Denote by $\mc F_{\GU}$ the functor on the 
            category of $\FF$-algebras with $\mc F_{\GU}(R)$ the set of self-dual 
            lattice chains in $R(\!(u)\!)^{n}$.  
    \end{enumerate}
\end{defn}
Note that $\aff{\mc L}\in \mc F_\U^{(0)}(R)$.
\begin{remark}
    Let $R$ be a reduced $\FF$-algebra such that $\Spec R$ is connected.
    Then
    \begin{equation*}
        \mc F_\GU(R)=\bigcup_{r\in\ZZ} \mc F_\U^{(r)}(R).
    \end{equation*}
\end{remark}
\begin{proof}
    This follows directly from Lemma \ref{LemUnitsinLaurent}.
\end{proof}
\begin{remark}\label{RemarkLatticesAreChainsUni}
    Let $R$ be an $\FF$-algebra and let $(L_i)_i\in \mc F_{\U}^{(0)}(R)$. For $i\in\ZZ$ denote by 
    $\varrho_i:L_i\to L_{i+1}$ the inclusion, by $\theta_i:L_{n+i}\to L_i$ the 
    isomorphism given by multiplication with $u$ and by $\mc E_i:L_i\times 
    L_{-i}\to R[\![u]\!]$ the restriction of $\aff{\paird}$. Then
    $(L_i,\varrho_i,\theta_i,\mc E_i)$ is a polarized chain of
    $\FF[\![u]\!]\otimes_{\FF[\![u_0]\!]}R[\![u_0]\!]$-modules of type 
    $(\aff{\mc L})$.
\end{remark}
Note that for an $\FF$-algebra $R$, the canonical maps
\begin{align*}
    \FF[\![u]\!]\otimes_{\FF[\![u_0]\!]}R[\![u_0]\!]&\to R[\![u]\!],\\
    \FF(\!(u)\!)\otimes_{\FF(\!(u_0)\!)}R(\!(u_0)\!)&\to R(\!(u)\!)
\end{align*}
are isomorphisms.
Consequently we can consider $\Lf_{u_0} \GU$ and $\Lf_{u_0} \U$ as subfunctors 
of $\Lf_u \GL_n$. Recall from Remark \ref{RemarkActionOfGLOnF} the subfunctor 
$I\subset \Lf \GL_{n}$.
We define a subfunctor $I_\GU$ of $\Lf_{u_0} \GU$ by $I_\GU=\Lf_{u_0} \GU\cap 
I$.
\begin{prop}
    The natural action of $\Lf_u \GL_{n}$ on $\mc F$ (cf.\ Remark 
    \ref{RemarkActionOfGLOnF}) restricts to an action of $\Lf_{u_0} \GU$ on $\mc 
    F_\GU$.
    Consequently we obtain an injective map 
    \begin{equation*}
        \begin{aligned}
            \Lf_{u_0} \GU(R)/I_\GU(R)&\too{\phi(R)} \mc F_\GU(R),\\
            g&\xmapsto{\hphantom{\phi(R)}} g\cdot \aff{\mc L}
        \end{aligned}
    \end{equation*}
    for each $\FF$-algebra $R$. The morphism $\phi$ identifies $\mc F_\GU$ 
    with both the \'etale and the fpqc sheafification of the presheaf $\Lf_{u_0} 
    \GU/I_\GU$.
\end{prop}
\begin{proof}
    Let $R$ be an $\FF$-algebra. We claim that every $a\in R(\!(u_0)\!)^\times$ 
    lies in the image of the map $c:\GU(R(\!(u_0)\!))\to R(\!(u_0)\!)^\times$ 
    \'etale locally on $R$. Assuming this, one can proceed exactly as in the 
    proof of Proposition \ref{PropFuncDescofFlag}.
    
    To prove the claim, first note that for $b\in R(\!(u)\!)$, the matrix $bI_n\in 
    \GU(R(\!(u_0)\!))$ satisfies $c(bI_n)=bb^{\aff{\inv}}$.  Lemma 
    \ref{LemUnitsinLaurent} implies that Zariski locally on $R$, the element
    $a$ is of the form $a=u_0^k\upsilon(1+n)$ for some $k\in\ZZ$, a unit 
    $\upsilon\in R[\![u_0]\!]^\times$ and a nilpotent element $n\in 
    R(\!(u_0)\!)$. Consequently it suffices to show that each of
    $u_0^k,\upsilon$ and $1+n$ is of the form $bb^{\aff{\inv}}$ for some $b\in 
    R(\!(u)\!)$ \'etale locally on $R$.
    
    As $2\in R^\times$, one easily sees that $\upsilon$ is a square in 
    $R[\![u_0]\!]^\times$ whenever $\upsilon(0)$ is a square in $R^\times$, 
    which is the case \'etale locally on $R$. For $b=\sqrt{-1}u^k$, one has 
    $bb^{\aff{\inv}}=u_0^k$.
    Finally, $1+n$ is a square in $R(\!(u_0)\!)^\times$; this follows from the 
    Taylor expansion of $\sqrt{1+x}$ if one notes that $\binom{1/2}{l}\in 
    \ZZ[\frac{1}{2}]$ for all $l\in\NN$.
\end{proof}
\subsection{Embedding the local model into the affine flag variety}
Let $R$ be an $\FF$-algebra. We consider an $R[u]/u^e$-module as an 
$R[\![u]\!]$-module via the canonical projection $R[\![u]\!]\to R[u]/u^e$.  
For $i\in\ZZ$ denote by $\alpha_i:\aff{\Lambda}_i\to \res{\Lambda}_{i,R}$ the morphism described 
by the identity matrix with respect to $\aff{\mf E}_i$ and $\res{\mf E}_i$. It 
induces an isomorphism
$\aff{\Lambda}_i/u^e\aff{\Lambda}_i\too{\sim}\res{\Lambda}_{i,R}$.  Clearly the 
following diagrams commute.
\begin{equation*}
    \xymatrix{
    \aff{\Lambda}_i\ar[d]_{\alpha_i}\ar@{}[r]|{\subset}&\aff{\Lambda}_{i+1}\ar[d]^{\alpha_{i+1}}\\
    \res{\Lambda}_{i,R}\ar[r]^-{\res{\rho}_{i,R}}&\res{\Lambda}_{i+1,R},
    }
    \quad
    \xymatrix{
    \aff{\Lambda}_i\times 
    \aff{\Lambda}_{-i}\ar[r]^-{\aff{\paird}_i}\ar[d]_{\alpha_i\times 
        \alpha_{-i}}& R[\![u]\!]\ar[d]\\
        \res{\Lambda}_{i,R}\times \res{\Lambda}_{-i,R}\ar[r]^-{\res{\paird}_{i,R}}& R[u]/u^e,
    }
    \quad
    \xymatrix{
        \aff{\Lambda}_i\ar[d]_{\alpha_i}&\aff{\Lambda}_{n+i}\ar[d]^{\alpha_{n+i}}\ar[l]_{u \cdot}\\
        \res{\Lambda}_{i,R}&\res{\Lambda}_{n+i,R}\ar[l]_{\res{\vartheta}_{i,R}}.
    }
\end{equation*}
    The following proposition allows us to consider $\Me^{e,n}$ as a subfunctor of $\mc 
F_\U^{(-e_0)}$.
\begin{prop}[{\cite[\textsection 3.3]{pr3},\cite[\textsection 
    4.4-5.1]{smithling_unitary_odd}, \cite[\textsection 
    6.4-7.1]{smithling_unitary_even}}]\label{PropLocalModelintoFlagUni}
    There is an embedding $\alpha:\Me^{e,n}\hookrightarrow \mc F_\U^{(-e_0)}$ given on $R$-valued by
    \begin{align*}
        \Me^{e,n}(R)&\to \mc F_\U^{(-e_0)}(R),\\
        (t_i)_i&\mapsto (\alpha_i^{-1}(t_i))_i.
    \end{align*}
    It induces a bijection from $\Me^{e,n}(R)$ onto the set of those $(L_i)_i\in \mc F_\U^{(-e_0)}(R)$ 
    satisfying the following conditions for all $i\in\ZZ$.
    \begin{enumerate}
        \item $u^e\aff{\Lambda}_i\subset L_i\subset\aff{\Lambda}_i$.
        \item For all $p\in R[u]/u^e$, we have
            \begin{equation*}
                \chi_R(p|\aff{\Lambda}_i/L_i)=(T-p(0))^{ne_0}
            \end{equation*}
            in $R[T]$. Here $\aff{\Lambda}_i/L_i$ is considered as an $R[u]/u^e$-module 
            using (1).
    \end{enumerate}
\end{prop}
\begin{proof}
    Identical to the proof of Proposition \ref{PropLocalModelintoFlag}.
\end{proof}
Note that $\res{\mc L}=(\res{\Lambda}_{i},\res{\rho}_{i},\res{\vartheta}_{i},\res{\paird}_{i})_i$ is a polarized chain of 
$\FF[\![u]\!]\otimes_{\FF[\![u_0]\!]}\FF[u_0]/u_0^{e_0}$-modules of type $(\aff{\mc L})$. In fact $\res{\mc L}=\aff{\mc 
L}\otimes_{\FF[\![u_0]\!]}\FF[u_0]/u_0^{e_0}$. Let $R$ be an $\FF$-algebra. There is an 
obvious action of $\Aut(\overline{\mc L})(R[u_0]/u_0^{e_0})$ on $\Me^{e,n}(R)$, 
given by $(\varphi_i)\cdot (t_{i})=(\varphi_i(t_i))$.
The canonical morphism $R[\![u_0]\!]\to R[u_0]/u_0^{e_0}$ induces a morphism $\Aut(\aff{\mc 
L})(R[\![u_0]\!])\to \Aut(\overline{\mc L})(R[u_0]/u_0^{e_0})$ and we thereby extend this 
$\Aut(\overline{\mc L})(R[u_0]/u_0^{e_0})$-action on $\Me^{e,n}(R)$ to an
$\Aut(\aff{\mc L})(R[\![u_0]\!])$-action.
\begin{lem}\label{LemDifferentAutsSameOrbitsUni}
    Let $R$ be an $\FF$-algebra and let $t\in \Me^{e,n}(R)$. We have $\Aut(\aff{\mc L})(R[\![u_0]\!])\cdot 
    t=\Aut(\overline{\mc L})(R[u_0]/u_0^e)\cdot t$.
\end{lem}
\begin{proof}
    The map $\Aut(\aff{\mc L})(R[\![u_0]\!])\to \Aut(\overline{\mc 
    L})(R[u_0]/u_0^{e_0})$ is surjective by Proposition \ref{PropLiftingIsosUni}.
\end{proof}
Define a subfunctor $I_\U$ of $\Lf_{u_0} \U$ by $I_{\U}=\Lf_{u_0} \U\cap I_\GU$.
\begin{lem}\label{LemIvsI0Uni}
    We have $I_\GU(\FF)=\FF[\![u_0]\!]^\times I_\U(\FF)$.
\end{lem}
\begin{proof}
    Analogous to the proof of Lemma \ref{LemIvsI0Sym}, noting that for $g\in 
    I_\GU(\FF)$ one has $c(g)\in \FF[\![u_0]\!]^\times$.  
\end{proof}
\begin{lem}\label{LemIvsIprimeUni}
    Let $g\in I_\U(\FF)$. Then $g$ restricts to an
    automorphism $g_i:\aff{\Lambda}_i\too{\sim}\aff{\Lambda}_i$ for each 
    $i\in\ZZ$. The assignment $g\mapsto (g_i)_i$ defines an isomorphism 
    $I_\U(\FF)\too{\sim} \Aut(\aff{\mc L})(\FF[\![u_0]\!])$.
\end{lem}
\begin{proof}
    Analogous to the proof of Lemma \ref{LemIvsIprime}.
\end{proof}
\begin{prop}\label{PropIndicesUni_2}
    Let $t\in \Me^{e,n}(\FF)$. 
    Then $\alpha$ induces a bijection
    \begin{equation*}\label{EqOrbits1Uni}
        \Aut(\res{\mc L})(\FF[u_0]/u_0^{e_0})\cdot t\too{\sim} I_{\GU}(\FF)\cdot \alpha(t).
    \end{equation*}
    Consequently we obtain an embedding
    \begin{equation*}
        \Aut(\overline{\mc L})(\FF[u_0]/u_0^{e_0})\backslash \Me^{e,n}(\FF)\hookrightarrow I_{\GU}(\FF)\backslash \mc F_\GU(\FF).
    \end{equation*}
\end{prop}
\begin{proof}
    Analogous to the proof of Proposition \ref{PropIndicesSym2}.
\end{proof}
Consider $\alpha':\Me^{e,n}(\FF)\hookrightarrow \mc F_\GU(\FF)\too{\phi(\FF)^{-1}}\Lf_{u_0} {\GU}(\FF)/I_{\GU}(\FF)$.
\begin{prop}\label{PropIndicesUni}
    Let $t\in \Me^{e,n}(\FF)$. Then $\alpha'$ induces a
    bijection
    \begin{equation*}
        \Aut(\res{\mc L})(\FF[u_0]/u_0^{e_0})\cdot t\too{\sim} I_{\GU}(\FF)\cdot \alpha'(t).
    \end{equation*}
    Consequently we obtain an embedding
    \begin{equation*}
        \Aut(\overline{\mc L})(\FF[u_0]/u_0^{e_0})\backslash \Me^{e,n}(\FF)\hookrightarrow I_{\GU}(\FF)\backslash \GU(\FF(\!(u_0)\!))/I_{\GU}(\FF).
    \end{equation*}
\end{prop}
\begin{proof}
    Clear from Proposition \ref{PropIndicesUni_2}, as the isomorphism 
    $\phi(\FF)$ is in particular $I_{\GU}(\FF)$-equivariant.
\end{proof}
Let $R$ be an $\FF$-algebra and $(\varphi_i)_i\in \Aut(\mc L)(R)$. The 
decomposition \eqref{LambdaInSpecialFiberUni} induces for each $i$ a 
decomposition of $\varphi_i:\Lambda_{i,R}\too{\sim}\Lambda_{i,R}$ into the product of $R[u]/u^e$-linear automorphisms 
$\varphi_{i,\sigma}:\res{\Lambda}_{i,R}\too{\sim}\res{\Lambda}_{i,R}$. The 
following statement is then clear (cf.\ the proof of Proposition 
\ref{PropDecompositionofLocalModel}).
\begin{prop}\label{PropChainMorphismDecomposesUni}
    Let $R$ be an $\FF$-algebra. The following map is an isomorphism, functorial 
    in $R$.
    \begin{align*}
        \Aut(\mc L)(R)&\to \prod_{\sigma\in\mf S} \Aut(\res{\mc 
        L})(R[u_0]/u_0^{e_0}),\\
        (\varphi_i)_i&\mapsto ((\varphi_{i,\sigma})_i)_{\sigma\in\mf S}.
    \end{align*}
\end{prop}
Consider the composition
\begin{equation*}
    \tilde{\alpha}:\Mloc(\FF)\too{\eqref{EqDecompLocModelUni}}
    \prod_{\sigma\in\mf S} \Me^{e,n}(\FF)\too{\prod_\sigma \alpha'}
    \prod_{\sigma\in\mf S} \Lf_{u_0} {\GU}(\FF)/I_{\GU}(\FF).  \end{equation*}
For $\sigma\in \mf S$ denote
by $\tilde{\alpha}_\sigma:\Mloc(\FF)\to \Lf_{u_0} {\GU}(\FF)/I_{\GU}(\FF)$ the 
corresponding component of $\tilde{\alpha}$.
\begin{thm}\label{ThmIndicesUni}
    Let $t\in \Mloc(\FF)$. Then $\tilde{\alpha}$ induces a bijection
    \begin{equation*}
        \Aut(\mc L)(\FF)\cdot t\too{\sim} \prod_{\sigma\in\mf S} 
        I_{\GU}(\FF)\cdot \tilde{\alpha}_\sigma(t).
    \end{equation*}
    Consequently we obtain an embedding
    \begin{equation*}
        \Aut(\mc L)(\FF)\backslash \Mloc(\FF)\hookrightarrow \prod_{\sigma\in\mf 
        S} I_{\GU}(\FF)\backslash \mc \GU(F(\!(u_0)\!))/I_{\GU}(\FF).
    \end{equation*}
\end{thm}
\begin{proof}
    Identical to the proof of Theorem \ref{ThmIndicesSym}.
\end{proof}
\subsection{The extended affine Weyl group}
As in \cite[3.2]{smithling_unitary_odd},\cite[3.2]{smithling_unitary_even}, we 
denote by $S$ the standard diagonal maximal split torus in $\GU$. Denote by $T$
the centralizer and by $N$ the normalizer of $S$ in $\GU$. By the discussion in 
\cite[3.4]{smithling_unitary_odd}, \cite[5.4]{smithling_unitary_even},
the Kottwitz homomorphism for $T$ is given by
\begin{equation*}
        \kappa_T:T(\FF(\!(u_0)\!))\to \ZZ^n,\quad \diag(x_1,\dots,x_n)\mapsto 
        (\val_u(x_1),\dots,\val_u(x_n)).
\end{equation*}
Consequently the kernel $T\bigl(\FF(\!(u_0)\!)\bigr)_1$ of $\kappa_T$ is equal 
to $T(\FF(\!(u_0)\!))\cap D_n(\FF[\![u]\!])$, with the intersection taking place 
in $\GL_n(\FF(\!(u)\!))$. Here $D_n\subset \GL_n$ denotes the subgroup of 
diagonal matrices. By definition, the extended affine Weyl group of $\GU$ 
with respect to $S$ is given by 
$\widetilde{W}:=N(\FF(\!(u_0)\!))/T(\FF(\!(u_0)\!))_1$.

Set
\begin{equation*}
	W=\{w\in S_{n}\mid \forall i\in\{1,\dots,n\}:\ w(i)+w(n+1-i)=n+1\}
\end{equation*}
and
\begin{equation*}
	X=\{(x_1,\dots,x_n)\in\ZZ^n\mid \exists r\in\ZZ\forall i\in\{1,\dots,n\}:\ x_i+x_{n+1-i}=2r\}.
\end{equation*}
We identify $W$ with a subgroup of $\U(\FF(\!(u_0)\!))$ via $W\ni w\mapsto A_w$.  
One easily sees that $N(\FF(\!(u_0)\!))=W\ltimes T(\FF(\!(u_0)\!))$. The 
Kottwitz homomorphism $\kappa_T$ induces an isomorphism 
$T(\FF(\!(u_0)\!))/T(\FF(\!(u_0)\!))_1\too{\sim} X$  and we thereby identify 
$\widetilde{W}$ with $W\ltimes X$.

To avoid any confusion of the product inside $\widetilde{W}$ and the canonical 
action of $S_n$ on $\ZZ^n$, we will always denote the element of 
$\widetilde{W}$ corresponding to $\lambda\in X$ by $u^\lambda$.

Recall from \cite[\textsection 2.5]{gy1} the notion of an extended alcove 
$(x_i)_{i=0}^{n-1}$ for $\GL_{n}$. An \emph{extended alcove for $\GU$} is an 
extended alcove $(x_i)_{i=0}^{n-1}$ for $\GL_{n}$ such that
\begin{equation*}
    \exists r\in\ZZ\forall i\in\{0,\dots,n\}\forall j\in\{1,\dots,n\}:\ x_i(j)+x_{n-i}(n+1-j)=2r-1.
\end{equation*}
Here $x_n=x_0+(1^{(n)})$. 

Also recall the standard alcove $(\omega_i)_{i=0}^{n-1}$. As in the linear case
treated in loc.\ cit., we identify $\widetilde{W}$ with the set of
extended alcoves for $\GU$ by using the standard alcove as a base point.

Write $\mathbf{e}=(e^{(n)})$.
\begin{defn}[{Cf.\ \cite{kottwitz_rapoport}}]
    An extended alcove $(x_i)_{i=0}^{n-1}$ for $\GU$ is called \emph{permissible} 
    if it satisfies the following conditions for all $i\in\{0,\dots,n-1\}$.
    \begin{enumerate}
        \item $\omega_i\leq x_i \leq \omega_i+\mathbf{e}$, where $\leq$ is to be 
            understood componentwise.
        \item $\sum_{j=1}^{n} x_i(j)=ne_0-i$.
    \end{enumerate}
    Denote by $\Perm$ the set of all permissible extended alcoves for 
    $\GU$.
\end{defn}
\begin{prop}
    The inclusion $N(\FF(\!(u_0)\!))\subset \GU(\FF(\!(u_0)\!))$ induces a bijection 
    $\widetilde{W}\too{\sim} I_{\GU}(\FF)\backslash \mc 
    \GU(\FF(\!(u_0)\!))/I_{\GU}(\FF)$. In other words,
    \begin{equation*}
        \GU(\FF(\!(u_0)\!))=\coprod_{x\in\widetilde{W}} I_{\GU}(\FF)xI_{\GU}(\FF).
    \end{equation*}
    Under this bijection, the subset 
    \begin{equation*}
        \Aut(\overline{\mc L})(\FF[u_0]/u_0^{e_0})\backslash \Me^{e,n}(\FF)\subset 
        I_{\GU}(\FF)\backslash \GU(\FF(\!(u_0)\!))/I_{\GU}(\FF)
    \end{equation*}
    of Proposition \ref{PropIndicesUni} corresponds to the subset 
    $\Perm\subset \widetilde{W}$.
\end{prop}
\begin{proof}
    The first statement is discussed in \cite[4.4]{smithling_unitary_odd}, \cite[6.4]{smithling_unitary_even}. The second 
    statement follows easily from the explicit description of the image of
    $\alpha$ in Proposition \ref{PropLocalModelintoFlagUni}.
\end{proof}
\begin{cor}\label{CorIndexSetUni}
    Under the identifications of  Theorem \ref{ThmIndicesUni}, the 
    set \allowbreak$\prod_{\sigma\in\mf S} \Perm$ constitutes a set of representatives of $\Aut(\mc 
    L)(\FF)\backslash \Mloc(\FF)$.
\end{cor}
\begin{remark}
    As explained in Remark \ref{RemVvsFSym}, the normalization of Corollary 
    \ref{CorIndexSetUni} differs from the one of Definition \ref{DefnIndexKR} by 
    the automorphism $\prod_{\sigma\in\mf S} \Perm\to \prod_{\sigma\in\mf S} \Perm,\ 
    (x_\sigma)\mapsto (u^ex_\sigma^{-1})$
\end{remark}
\subsection{The \texorpdfstring{$p$}{p}-rank on a KR stratum}
Recall from Section \ref{SecLocalModelGen} the scheme 
$\mc A/\mc O_{E_\mc Q}$ associated with our choice of PEL datum, and the KR stratification 
\begin{equation*}
    \mc A(\FF)=\coprod_{x\in \Aut(\mc L)(\FF)\backslash \Mloc(\FF)}\mc A_x.
\end{equation*}
We have identified the occurring index set with 
$\prod_{\sigma\in\mf S}\Perm$ in Corollary \ref{CorIndexSetUni} . We can then 
state the following result.
\begin{thm}\label{ThmPrankUni}
    Let $x=(x_\sigma)_\sigma\in \prod_{\sigma\in\mf S} \Perm$. Write 
    $x_\sigma=w_\sigma u^{\lambda_\sigma}$ with $w_\sigma\in W,\ 
    \lambda_\sigma\in X$. Then the $p$-rank on $\mc A_x$ is constant with value
    \begin{equation*}
        g\cdot |\{1\leq i\leq n\mid \forall \sigma\in\mf S(w_\sigma(i)=i\wedge 
        \lambda_\sigma(i)=0)\}|.
    \end{equation*}
\end{thm}
\begin{proof}
    The proof is identical to the one of Theorem \ref{ThmPrank}.
\end{proof}
\section{The inert unitary case}\label{SecUni2}
\subsection{The PEL datum}\label{SecPELUni2}
We start with the PEL datum defined in Section \ref{SecPELUni}.
We assume that $p\mc O_{F_0}=(\mc P_0)^{e}$ for a single prime $\mc P_0$ of $\mc 
O_{F_0}$ and that $\mc P_0\mc O_F=\mc P$ for a single prime $\mc P$ of $\mc O_F$. 
Denote by $f_0=[k_{\mc P_0}:\FF_p]$ and $f=[k_{\mc P}:\FF_p]$ the corresponding 
inertia degrees, so that $f=2f_0$. We fix once and for all a uniformizer $\pi$ 
of $\mc O_{F_0}\otimes\ZZ_{(p)}$. Then $\pi$ is also a
uniformizer of $\mc O_{F}\otimes\ZZ_{(p)}$.

Denote by $\mf C=\mf C_{\mc O_{F_\mc P}\mid \ZZ_p}$ the corresponding inverse 
different. Choose a generator $\delta$ of $\mf C$ satisfying 
$\delta^\inv=-\delta$. Consequently the form $\delta^{-1}\pairtd'_{\QQ_p}:V_{\QQ_p}\times V_{\QQ_p}\to \fp$ is $\inv$-hermitian 
and we assume that it \emph{splits}, i.e.\ that there is a basis 
$(e_1,\dots,e_n)$ of $V_{\QQ_p}$ over $\fp$ such that 
$\pairt{e_i}{e_{n+1-j}}_{\QQ_p}'=\delta\delta_{ij}$ for $1\leq i,j\leq n$.

Let $0\leq i<n$. We denote by $\Lambda_i$ the $\mc O_{F_\mc P}$-lattice in 
$V_{\QQ_p}$ with basis
\begin{equation*}
	\mf E_i=(\pi^{-1}e_1,\ldots,\pi^{-1}e_i,e_{i+1},\ldots,e_{n}).
\end{equation*}
For $k\in\ZZ$ we further define $\Lambda_{nk+i}=\pi^{-k}\Lambda_i$ and we denote 
by $\mf E_{nk+i}$ the corresponding basis obtained from $\mf E_i$.
Then $\mc L=(\Lambda_i)_i$ is a complete chain of $\mc O_{F_\mc P}$-lattices in 
$V_{\QQ_p}$. For $i\in\ZZ$, the dual lattice 
$\Lambda_i^\vee:=\{x\in V_{\QQ_p}\mid \pairt{x}{\Lambda_i}_{\QQ_p}\subset\ZZ_p\}$ of $\Lambda_i$ is given by 
$\Lambda_{-i}$.  Consequently the chain $\mc L$ is self-dual. 

Let $i\in\ZZ$. We denote by $\rho_i:\Lambda_{i}\to \Lambda_{i+1}$ the 
inclusion, by $\vartheta_i:\Lambda_{n+i}\to \Lambda_i$ the isomorphism given by 
multiplication with $\pi$ and by $\pairtd_i:\Lambda_i\times \Lambda_{-i}\to 
\ZZ_p$ the restriction of $\pairtd_{\QQ_p}$. Then
$(\Lambda_{i},\rho_{i},\vartheta_{i},\pairtd_{i})_i$ is a polarized chain of $\mc 
O_{\fp}$-modules of type $(\mc L)$, which, by abuse of notation, we also 
denote by $\mc L=\mc L^{\mathrm{inert}}$.

Denote by $\paird_i:\Lambda_i\times \Lambda_{-i}\to \mc O_{F_\mc P}$ the 
restriction of the $\inv$-hermitian form $\delta^{-1}\pairtd'_{\QQ_p}$. It is the 
$\inv$-sesquilinear form described by the matrix
$\widetilde{I}_n$ with respect to $\mf E_i$ and $\mf E_{-i}$.

Denote by $\Sigma_0$ the set of all embeddings $F_0\hookrightarrow\RR$ and by 
$\Sigma$ the set of all embeddings $F\hookrightarrow \CC$. Also write $\mf 
S=\Gal(k_\mc P/\FF_p)$ and $\mf S_0=\Gal(k_{\mc P_0}/\FF_p)$. 

Let $E'$ be the Galois closure of $F$ inside $\CC$ and choose a prime $\mc Q'$ 
of $E'$ over $\mc P$.
Consider the maps $\gamma:\Sigma\to \mf S$ and 
$\gamma_0:\Sigma_0\to \mf S_0$ of Lemma 
\ref{LemResidueMapinNonGaloisCase}. For each $\sigma\in\mf S_0$ 
we denote by $\tau_{\sigma,1},\tau_{\sigma,2}\in\mf S$ the two elements with 
$\tau_{\sigma,j}\big|_{k_{\mc P_0}}=\sigma$. 

Let $\sigma\in\Sigma_0$ and $j\in\{1,2\}$.  
There is a unique $\tau_{\sigma,j}\in\Sigma$ with $\tau_{\sigma,j}\big|_{F_0}=\sigma$ satisfying
\begin{equation}\label{EqInterplayOfGammas}
        \gamma(\tau_{\sigma,j})=\tau_{\gamma_0(\sigma),j}.
\end{equation}
Exactly as in Section \ref{SecUni}, we define for each $\sigma\in\Sigma_0$ 
integers $r_\sigma,s_\sigma$ with $r_\sigma+s_\sigma=n$,
and using these the element $J\in \End_{B\otimes \RR}(V\otimes \RR)$. Denote by 
$V_{-i}$ the $(-i)$-eigenspace of
$J_\CC$. As before, we construct an $\mc O_F\otimes\mc O_{E'}$-module $M_0$
which is finite locally free over $\mc O_{E'}$, such that
$M_0\otimes_{\mc O_{E'}}\CC=V_{-i}$ as $\mc O_F\otimes\CC$-modules.
\subsection{The special fiber of the determinant morphism}
Let $\sigma\in\mf S_0$. We define
\begin{equation*}
    \overline{r}_\sigma=\sum_{\sigma'\in \gamma_0^{-1}(\sigma)} r_{\sigma'}\quad
    \text{and}\quad \overline{s}_\sigma=\sum_{\sigma'\in \gamma_0^{-1}(\sigma)} 
    s_{\sigma'}.
\end{equation*}
As the fibers of $\gamma_0$ have cardinality $e$, it follows that 
$\overline{r}_\sigma+\overline{s}_\sigma=ne$.

We fix once and for all an embedding $\iota_{\mc Q'}:k_{\mc Q'}\hookrightarrow 
\FF$. We consider $\FF$ as an $\mc O_{E'}$-algebra with respect to the 
composition $\mc O_{E'}\too{\rho_{\mc Q'}}k_{\mc Q'}\overset{\iota_{\mc Q'}}{\hookrightarrow} \FF$. Also 
$\iota_{\mc Q'}$ induces an embedding $\iota_\mc P:k_{\mc \mc P}\hookrightarrow \FF$ and 
thereby an identification of the set of all embeddings $k_{\mc P}\hookrightarrow
\FF$ with $\mf S$.

Our choice of uniformizer $\pi$ induces a canonical isomorphism
\begin{equation}
    \label{EqdecompspecialfiberUni2}
    \mc O_{F}\otimes \FF=\prod_{\sigma\in \mf S_0} \FF[u]/(u^{e})\times \FF[u]/(u^{e}).
\end{equation}
Here in the component 
$\FF[u]/(u^{e})\times \FF[u]/(u^{e})$ corresponding to $\sigma\in \mf S_0$, the
first factor is supposed to correspond to $\tau_{\sigma,1}$ and the second 
factor is supposed to correspond to $\tau_{\sigma,2}$.

\begin{prop}\label{PropDeterminantSpecialFiberUni2}
    Let $x\in \mc O_F$ and let $\bigl((q_{\tau_{\sigma,1}},q_{\tau_{\sigma,2}})\bigr)_\sigma\in \prod_{\sigma\in \mf S_0} \FF[u]/(u^{e})\times \FF[u]/(u^{e})$ be the element corresponding to $x\otimes 1$
    under \eqref{EqdecompspecialfiberUni2}. Then
    \begin{equation*}
        \chi_\FF(x| M_0\otimes_{\mc O_{E'}}\FF)=\prod_{\sigma\in \mf S_0}\bigl(T-q_{\tau_{\sigma,1}}(0)\bigr)^{\overline{s}_\sigma}\bigl(T-q_{\tau_{\sigma,2}}(0)\bigr)^{\overline{r}_\sigma}
    \end{equation*}
    in $\FF[T]$.
\end{prop}
\begin{proof}
    Reduce $\chi_{\mc O_{E'}}(x|M_0)$ modulo $\mc Q'$, using \eqref{EqInterplayOfGammas}.
\end{proof}
Denote by $E=\QQ(\tr_\CC(x\otimes 1|V_{-i});\ x\in F)$ the reflex field and define $\mc Q=\mc Q'\cap \mc O_E$.
The morphism $\det_{V_{-i}}$ is defined over $\mc O_E$.
\subsection{The 
local model}
For the chosen PEL datum, Definition \ref{DefnLocalModelGen} amounts to the 
following.
\begin{defn}\label{DefnLocalModelUni2}
    The local model $\Mloc=\Me^{\mathrm{loc,inert}}$ is the functor on the 
    category of
    $\mc O_{E_\mc Q}$-algebras with $\Mloc(R)$ the set of tuples $(t_i)_{i\in 
    \ZZ}$ of
    $\mc O_{F}\otimes R$-submodules $t_i\subset \Lambda_{i,R}$,
    satisfying the following conditions for all $i\in\ZZ$.
    \renewcommand\theenumi{\alph{enumi}}
    \begin{enumerate}
        \item\label{DefnLocalModelUni2-Functoriality}
            $\rho_{i,R}(t_i)\subset t_{i+1}$.
        \item\label{DefnLocalModelUni2-Projectivity} The quotient 
            $\Lambda_{i,R}/t_i$ is a finite locally free $R$-module.
        \item\label{DefnLocalModelUni2-Determinant} We have an equality
            \begin{equation*}
                \det_{\Lambda_{i,R}/t_i}=\det_{V_{-i}}\otimes_{\mc O_E}R
            \end{equation*}
            of morphisms $V_{\mc O_F\otimes R}\to \AF^1_R$.
        \item\label{DefnLocalModelUni2-DualityCondition} Under the pairing 
            $\pairtd_{i,R}:\Lambda_{i,R}\times \Lambda_{-i,R}\to R$, the 
            submodules $t_i$ and $t_{-i}$ pair to zero.
        \item\label{DefnLocalModelUni2-Periodicity} $\vartheta_i(t_{n+i})=t_i$.
    \end{enumerate}
\end{defn}
\subsection{The special fiber of the local model}\label{SecSpecialFibLocModUni2}
For $i\in\ZZ$, denote by $\res{\Lambda}_{i}$ the free $\FF[u]/u^{e}$-module 
$(\FF[u]/u^{e})^{n}$ and by $\res{\mf E}_{i}$ its canonical basis. Denote by 
$\res{\vartheta}_{i}:\res{\Lambda}_{n+i}\to \res{\Lambda}_{i}$ the identity 
morphism.
Consider the map $\res{\ast}:\FF[u]/u^{e}\times \FF[u]/u^{e}\to 
\FF[u]/u^{e}\times \FF[u]/u^{e},\ (a,b)\mapsto (b,a)$. Let $\res{\Lambda}_{i,1}$ 
and $\res{\Lambda}_{i,2}$ be two copies of $\res{\Lambda}_i$ and denote by 
$\res{\paird}_{i,1}:\res{\Lambda}_{i,1}\times \res{\Lambda}_{-i,2}\to 
\FF[u]/u^{e}$ (resp.\ $\res{\paird}_{i,2}:\res{\Lambda}_{i,2}\times 
\res{\Lambda}_{-i,1}\to \FF[u]/u^{e}$) the perfect bilinear map described by the 
matrix $\widetilde{I}_n$ with respect to $\res{\mf E}_{i,1}$ and $\res{\mf 
E}_{-i,2}$ (resp.\ $\res{\mf E}_{i,2}$ and $\res{\mf E}_{-i,1}$). Consider the 
pairing
\begin{align*}
    \overline{\paird}_{i}:(\res{\Lambda}_{i,1}\times \res{\Lambda}_{i,2})\times 
    (\res{\Lambda}_{-i,1}\times \res{\Lambda}_{-i,2})&\to
    \FF[u]/u^{e}\times \FF[u]/u^{e},\\
    \bigl((x_1,x_2),(y_1,y_2)\bigr)&\mapsto 
    \bigl(\pair{x_1}{y_2}_{i,1},\pair{x_2}{y_1}_{i,2}\bigr).
\end{align*}
It is a perfect $\res{\ast}$-sesquilinear pairing.

For $k\in\ZZ$ and $0\leq i<n$, let $\res{\rho}_{nk+i}:\res{\Lambda}_{nk+i}\to 
\res{\Lambda}_{nk+i+1}$ be the morphism described by the matrix 
$\mathrm{diag}(1^{(i)},u,1^{(n-i-1)})$ with respect to $\res{\mf E}_{nk+i}$ and 
$\res{\mf E}_{nk+i+1}$.  \begin{defn}\label{DefnSpecialLocalModelUni2}
    Let $r,s\in\NN$ with $r+s=ne$. Define a functor $\Me^{e,n,r}$ on the 
    category of $\FF$-algebras with $\Me^{e,n,r}(R)$ the set of tuples 
    $(t_i)_{i\in \ZZ}$ of $R[u]/u^{e}$-submodules $t_i\subset 
    \res{\Lambda}_{i,R}$
    satisfying the following conditions for all $i\in\ZZ$.
    \renewcommand\theenumi{\alph{enumi}}
    \begin{enumerate}
        \item $\res{\rho}_{i,R}(t_i)\subset t_{i+1}$.
        \item The quotient $\res{\Lambda}_{i,R}/t_i$ is a finite locally free 
            $R$-module.
        \item\label{DefnSpecialLocalModelUni2-Determinant} For all $p\in 
            R[u]/u^e$, we have
            \begin{equation*}
                \chi_R(p|\res{\Lambda}_{i,R}/t_i)=\bigl(T-p(0)\bigr)^{s}
            \end{equation*}
            in $R[T]$.
            \item $\res{\vartheta}_i(t_{n+i})=t_i$.
    \end{enumerate}
\end{defn}
Let $i\in\ZZ$. From \eqref{EqdecompspecialfiberUni2} we obtain an isomorphism
\begin{equation}\label{LambdaInSpecialFiberUni2}
    \Lambda_{i,\FF}=\prod_{\sigma\in \mf S_0}\res{\Lambda}_{i,1}\times 
    \res{\Lambda}_{i,2}
\end{equation}
by identifying the basis $\mf E_{i,\FF}$ with the product of the bases $\res{\mf 
E}_i$. Under this identification, the morphism $\rho_{i,\FF}$ decomposes into 
the morphisms $\res{\rho}_i$, the pairing $\paird_{i,\FF}$ decomposes into the 
pairings $\res{\paird}_i$ and the morphism $\vartheta_{i,\FF}$ decomposes into 
the morphisms $\res{\vartheta}_i$.

Let $R$ be an $\FF$-algebra and let $(t_i)_{i\in\ZZ}$ be a tuple of $\mc 
O_F\otimes R$-submodules $t_i\subset \Lambda_{i,R}$.  Then 
\eqref{LambdaInSpecialFiberUni2} induces decompositions $t_i=\prod_{\sigma\in 
\mf S_0} t_{i,\tau_{\sigma,1}}\times t_{i,\tau_{\sigma,2}}$ into 
$R[u]/u^{e}$-submodules $t_{i,\tau_{\sigma,j}}\subset\res{\Lambda}_{i,j,R}$. The 
following statement is then clear (cf.\ the proof of Proposition 
\ref{PropDecompositionofLocalModel}).
\begin{prop}\label{PropDecompositionofLocalModelUni2}
    The morphism $\Phi_1:\Mloc_\FF\to \prod_{\sigma\in \mf S_0} 
    \Me^{e,n,\overline{r}_\sigma}$ given on $R$-valued points by
    \begin{align*}
        \Mloc_\FF(R)&\to \prod_{\sigma\in \mf S_0} 
        \Me^{e,n,\overline{r}_\sigma}(R),\\
        (t_i)&\mapsto \left((t_{i,\tau_{\sigma,1}})_i\right)_\sigma
    \end{align*}
    is an isomorphism of functors on the category of $\FF$-algebras.
\end{prop}
\begin{remark}\label{RemarkDifferentDecompofLocalModel}
    For symmetry reasons, also the morphism $\Phi_2:\Mloc_\FF\to 
    \prod_{\sigma\in \mf S_0} \Me^{e,n,\overline{s}_\sigma}$ given on $R$-valued 
    points by
    \begin{align*}
        \Mloc_\FF(R)&\to \prod_{\sigma\in \mf S_0} 
        \Me^{e,n,\overline{s}_\sigma}(R),\\
        (t_i)&\mapsto \left((t_{i,\tau_{\sigma,2}})_i\right)_\sigma
    \end{align*}
    is an isomorphism of functors on the category of $\FF$-algebras.

    The morphism $\prod_{\sigma\in \mf S_0} \Me^{e,n,\overline{r}_\sigma}\to 
    \prod_{\sigma\in \mf S_0} \Me^{e,n,\overline{s}_\sigma}$ making commutative 
    the diagram
    \begin{equation*}
        \xymatrix{
        &\prod_{\sigma\in \mf S_0} \Me^{e,n,\overline{r}_\sigma}\ar[dd]\\
        \Mloc_\FF\ar[ur]^{\Phi_1}\ar[dr]_{\Phi_2}& \\
        &\prod_{\sigma\in \mf S_0} \Me^{e,n,\overline{s}_\sigma}\\
        }
    \end{equation*}
    is given by on $R$-valued points by
    \begin{equation}
        \begin{aligned}
            \prod_{\sigma\in \mf S_0} \Me^{e,n,\overline{r}_\sigma}(R)&\to 
            \prod_{\sigma\in \mf S_0} \Me^{e,n,\overline{s}_\sigma}(R),\\
            ((t_{i,\sigma})_i)_{\sigma}&\mapsto 
            ((t_{-i,\sigma}^{\perp,\res{\paird}_{-i,1,R}})_i)_{\sigma}.
        \end{aligned}
        \label{EqRelationBetweenDifferentMs}
    \end{equation}
\end{remark}
\subsection{Embedding the local model into the affine flag variety}
Recall from Section \ref{SecLattices} the affine flag variety $\mc F$. Let $R$ 
be an $\FF$-algebra. We consider an $R[u]/u^e$-module as an $R[\![u]\!]$-module 
via the canonical projection $R[\![u]\!]\to R[u]/u^e$. For $i\in\ZZ$, denote by 
$\alpha_i:\aff{\Lambda}_i\to \res{\Lambda}_{i,R}$ the morphism described by the 
identity matrix with respect to $\aff{\mf E}_i$ and $\res{\mf E}_i$. It induces 
an isomorphism
$\aff{\Lambda}_i/u^e\aff{\Lambda}_i\too{\sim}\res{\Lambda}_{i,R}$.  Clearly the 
following diagrams commute.
\begin{equation*}
    \xymatrix{
    \aff{\Lambda}_i\ar[d]_{\alpha_i}\ar@{}[r]|{\subset}&\aff{\Lambda}_{i+1}\ar[d]^{\alpha_{i+1}}\\
    \res{\Lambda}_{i,R}\ar[r]^-{\res{\rho}_{i,R}}&\res{\Lambda}_{i+1,R},
    }
    \quad
    \xymatrix{
    \aff{\Lambda}_i\ar[d]_{\alpha_i}&\aff{\Lambda}_{n+i}\ar[d]^{\alpha_{n+i}}\ar[l]_{u\cdot}\\
    \res{\Lambda}_{i,R}&\res{\Lambda}_{n+i,R}\ar[l]_{\res{\vartheta}_{i,R}}.
    }
\end{equation*}
Let $r,s\in\NN$ with $r+s=ne$. The following proposition allows us to consider 
$\Me^{e,n,r}$ as a subfunctor of $\mc F$.
\begin{prop}[{\cite[\textsection 4]{pr2}}]\label{PropLocalModelintoFlagUni2}
    There is an embedding $\alpha:\Me^{e,n,r}\hookrightarrow \mc F$ given on 
    $R$-valued points by
    \begin{align*}
        \Me^{e,n,r}(R)&\to \mc F(R),\\
        (t_i)_i&\mapsto (\alpha_i^{-1}(t_i))_i.
    \end{align*}
    It induces a bijection from $\Me^{e,n,r}(R)$ onto the set of those 
    $(L_i)_i\in \mc F(R)$ satisfying the following conditions for all $i\in\ZZ$.
    \begin{enumerate}
        \item $u^e\aff{\Lambda}_i\subset L_i\subset\aff{\Lambda}_i$.
        \item For all $p\in R[u]/u^e$, we have
            \begin{equation*}
                \chi_R(p|\aff{\Lambda}_{i}/L_i)=\bigl(T-p(0)\bigr)^{s}
            \end{equation*}
            in $R[T]$. Here $\aff{\Lambda}_i/L_i$ is considered as an 
            $R[u]/u^e$-module using (1).
    \end{enumerate}
\end{prop}
\begin{proof}
    Analogous to the proof of Proposition \ref{PropLocalModelintoFlag}.
\end{proof}
Let $R$ be an $\FF$-algebra. Denote by $\aff{\paird}:R(\!(u)\!)^n\times 
R(\!(u)\!)^n\to R(\!(u)\!)$ the bilinear form described by the matrix 
$\widetilde{I}_n$ with respect to the standard basis of $R(\!(u)\!)^n$ over 
$R(\!(u)\!)$. Further denote by
$\aff{\paird}_i:\aff{\Lambda}_i\times \aff{\Lambda}_{-i}\to R[\![u]\!]$ the 
restriction of $\aff{\paird}$.
Note that the diagram
 \begin{equation*}
     \xymatrix{
         \aff{\Lambda}_{i}\times 
         \aff{\Lambda}_{-i}\ar[rr]^{\aff{\paird}_{i}}\ar[d]_{\alpha_{i}\times 
         \alpha_{-i}}&& R[\![u]\!]\ar[d]\\
         \res{\Lambda}_{i,1,R}\times 
         \res{\Lambda}_{-i,2,R}\ar[rr]^-{\res{\paird}_{i,1,R}}&& R[u]/u^e
     }
\end{equation*}
commutes. For a lattice $\Lambda$ in $R(\!(u)\!)^{n}$ we define 
$\Lambda^\vee:=\{x\in R(\!(u)\!)^{n}\mid \aff{\pair{x}{\Lambda}}\subset 
R[\![u]\!]\}$. 

As in Remark \ref{RemarkTwoEmbeddingsUni2}, the morphism $\Psi:\Me^{e,n,r}\to 
\Me^{e,n,s}$ given on $R$-valued points by
\begin{align*}
    \Me^{e,n,r}(R)&\to \Me^{e,n,s}(R),\\
    (t_i)_i&\mapsto (t_{-i}^{\perp,\res{\paird}_{-i,1,R}})_i
\end{align*}
is an isomorphism.
\begin{prop}\label{PropDifferentEmbeddingsIntoFlag}
    The following diagram commutes.
    \begin{equation*}
        \xymatrix{
            \Me^{e,n,r}\xyhookrightarrow^-\alpha\ar[d]_-\Psi& \mc F\ar[d]^{(L_i)_i \mapsto (u^eL_{-i}^\vee)_i}\\
            \Me^{e,n,s}\xyhookrightarrow^-\alpha& \mc F.
        }
    \end{equation*}
\end{prop}
\begin{proof}
    Similar to the proof of the duality statement in the proof of Proposition \ref{PropLocalModelintoFlag}.
\end{proof}
Note that $\res{\mc L}=(\res{\Lambda}_{i},\res{\rho}_{i},\res{\vartheta}_{i})$ is a chain of 
$\FF[u]/u^e$-modules of type $(\aff{\mc L})$. In fact $\res{\mc L}=\aff{\mc 
L}\otimes_{\FF[\![u]\!]}\FF[u]/u^e$. Let $R$ be an $\FF$-algebra. There is an obvious action of 
$\Aut(\res{\mc L})(R[u]/u^e)$ on $\Me^{e,n,r}(R)$, given by
$(\varphi_i)\cdot (t_{i})=(\varphi_i(t_i))$. The canonical morphism $R[\![u]\!]\to R[u]/u^e$ induces a map $\Aut(\aff{\mc 
L})(R[\![u]\!])\to \Aut(\overline{\mc L})(R[u]/u^e)$ and we thereby extend this 
$\Aut(\overline{\mc L})(R[u]/u^e)$-action on $\Me^{e,n,r}(R)$ to an
$\Aut(\aff{\mc L})(R[\![u]\!])$-action.
\begin{lem}\label{LemDifferentAutsSameOrbitsUni2}
    Let $R$ be an $\FF$-algebra and let $t\in \Me^{e,n,r}(R)$. We have $\Aut(\aff{\mc L})(R[\![u]\!])\cdot 
    t=\Aut(\overline{\mc L})(R[u]/u^e)\cdot t$.
\end{lem}
\begin{proof}
    The map $\Aut(\aff{\mc L})(R[\![u]\!])\to \Aut(\overline{\mc 
    L})(R[u]/u^e)$ is surjective by Proposition \ref{PropLiftingIsosLin}.
\end{proof}
\begin{lem}\label{LemIvsIprimeUni2}
    Let $g\in I(\FF)$. Then $g$ restricts to an automorphism 
    $g_i:\aff{\Lambda}_i\too{\sim}\aff{\Lambda}_i$ for each $i\in\ZZ$. The 
    assignment $g\mapsto (g_i)_i$ defines an isomorphism $I(\FF)\too{\sim} 
    \Aut(\aff{\mc L})(\FF[\![u]\!])$.
\end{lem}
\begin{proof}
    Clear (cf.\ the proof of Lemma \ref{LemIvsIprime}).
\end{proof}
\begin{prop}\label{PropIndicesUni2_2}
    Let $t\in \Me^{e,n,r}(\FF)$. Then $\alpha$ induces a
    bijection 
    \begin{equation*}
        \Aut(\res{\mc L})(\FF[u]/u^e)\cdot t\too{\sim} I(\FF)\cdot \alpha(t).
    \end{equation*}
    Consequently we obtain an embedding
    \begin{equation*}
        \Aut(\overline{\mc L})(\FF[u]/u^e)\backslash \Me^{e,n,r}(\FF)\hookrightarrow I(\FF)\backslash \mc F(\FF).
    \end{equation*}
\end{prop}
\begin{proof}
    Analogous to the proof of Proposition \ref{PropIndicesSym}.
\end{proof}
Consider $\alpha':\Me^{e,n,r}(\FF)\hookrightarrow \mc F(\FF)\too{\phi(\FF)^{-1}}\Lf \GL_n(\FF)/I(\FF)$.
\begin{prop}\label{PropIndicesUni2}
    Let $t\in \Me^{e,n,r}(\FF)$. Then $\alpha'$ induces a
    bijection
    \begin{equation*}
        \Aut(\res{\mc L})(\FF[u]/u^e)\cdot t\too{\sim} I(\FF)\cdot \alpha'(t).
    \end{equation*}
    Consequently we obtain an embedding
    \begin{equation}\label{Eqalphabar}
        \Aut(\overline{\mc L})(\FF[u]/u^e)\backslash \Me^{e,n,r}(\FF)\hookrightarrow I(\FF)\backslash \GL_n(\FF(\!(u)\!))/I(\FF).
    \end{equation}
\end{prop}
\begin{proof}
    Clear from Proposition \ref{PropIndicesUni2_2}, as the isomorphism 
    $\phi(\FF)$ is in particular $I(\FF)$-equivariant.
\end{proof}
Denote by $\tau$ the adjoint involution for $\aff{\paird}$ on $\GL_n(\FF(\!(u)\!))$, 
so that for $g\in\GL_n(\FF(\!(u)\!))$ we have
$\aff{\pair{gx}{y}}=\aff{\pair{x}{g^\tau}},\ x,y\in\FF(\!(u)\!)^n$.
\begin{prop}\label{PropDifferentEmbeddingsIntoFlag2}
    The vertical maps in the following diagram are well-defined bijections and 
    the diagram commutes.
    \begin{equation*}
        \xymatrix{
            \Aut(\overline{\mc L})(\FF[u]/u^e)\backslash \Me^{e,n,r}(\FF)\xyhookrightarrow^-{\eqref{Eqalphabar}}\ar[d]_\Psi& I(\FF)\backslash \GL_n(\FF(\!(u)\!))/I(\FF)\ar[d]^{g\mapsto u^e(g^\tau)^{-1}}\\
            \Aut(\overline{\mc L})(\FF[u]/u^e)\backslash \Me^{e,n,s}(\FF)\xyhookrightarrow^-{\eqref{Eqalphabar}}& I(\FF)\backslash \GL_n(\FF(\!(u)\!))/I(\FF).
        }
    \end{equation*}
\end{prop}
\begin{proof}
    In view of Proposition \ref{PropDifferentEmbeddingsIntoFlag} it suffices to 
    note the following statement, which follows from a short computation:
    Let $\Lambda$ be a lattice in $\FF(\!(u)\!)^n$ and let $g\in 
    \GL_n(\FF(\!(u)\!))$. Then $(g\Lambda)^\vee=(g^\tau)^{-1}(\Lambda^\vee)$.
\end{proof}
Let $R$ be an $\FF$-algebra and $\varphi=(\varphi_i)_i\in \Aut(\mc L)(R)$. The 
decomposition \eqref{LambdaInSpecialFiberUni2} induces for each $i$ a 
decomposition of $\varphi_i:\Lambda_{i,R}\too{\sim}\Lambda_{i,R}$ into the 
product of $R[u]/u^e$-linear automorphisms 
$\varphi_{i,\tau_{\sigma,j}}:\res{\Lambda}_{i,j,R}\too{\sim}\res{\Lambda}_{i,j,R}$.
The following statement is then clear (cf.\ the proof of Proposition 
\ref{PropDecompositionofLocalModel}).
\begin{prop}\label{PropChainMorphismDecomposesUni2}
    Let $R$ be an $\FF$-algebra. The following map is an isomorphism, functorial 
    in $R$.
    \begin{align*}
        \Aut(\mc L)(R)&\to \prod_{\sigma \in\mf S_0}
        \Aut(\res{\mc L})(R[u]/u^e),\\
        (\varphi_i)_i&\mapsto ((\varphi_{i,\tau_{\sigma,1}})_i)_\sigma.
    \end{align*}
\end{prop}
Consider the composition 
\begin{equation*}
    \tilde{\alpha}_1:\Mloc(\FF)\too{\Phi_1}
    \prod_{\sigma\in\mf S_0} \Me^{e,n,\res{r}_\sigma}(\FF)\too{\prod_\sigma \alpha'}
    \prod_{\sigma\in\mf S_0} \Lf \GL_n(\FF)/I(\FF). 
\end{equation*}
For $\sigma\in \mf S_0$ denote
by $\tilde{\alpha}_{1,\sigma}:\Mloc(\FF)\to \Lf \GL_n(\FF)/I(\FF)$ the corresponding component 
of $\tilde{\alpha}_1$.
\begin{thm}\label{ThmIndicesUni2}
    Let $t\in \Mloc(\FF)$. Then $\tilde{\alpha}_1$ induces a bijection
    \begin{equation*}
        \Aut(\mc L)(\FF)\cdot t\too{\sim} \prod_{\sigma\in\mf S_0} I(\FF)\cdot \tilde{\alpha}_{1,\sigma}(t).
    \end{equation*}
    Consequently we obtain an embedding
    \begin{equation*}
        \iota_1:\Aut(\mc L)(\FF)\backslash \Mloc(\FF)\hookrightarrow \prod_{\sigma\in\mf S_0} I(\FF)\backslash \mc \GL_n(\FF(\!(u)\!))/I(\FF).
    \end{equation*}
\end{thm}
\begin{proof}
    Identical to the proof of Theorem \ref{ThmIndicesSym}.
\end{proof}
\begin{remark}\label{RemarkTwoEmbeddingsUni2}
    In the same way, the composition 
    \begin{equation*}
        \tilde{\alpha}_2:\Mloc(\FF)\too{\Phi_2} 
        \prod_{\sigma\in\mf S_0} \Me^{e,n,\res{s}_\sigma}(\FF)\too{\prod_\sigma \alpha'}
        \prod_{\sigma\in\mf S_0} \Lf \GL_n(\FF)/I(\FF)
    \end{equation*}
    induces an embedding
    \begin{equation*}
        \iota_2:\Aut(\mc L)(\FF)\backslash \Mloc(\FF)\hookrightarrow 
        \prod_{\sigma\in\mf S_0} I(\FF)\backslash \mc \GL_n(\FF(\!(u)\!))/I(\FF).
    \end{equation*}
    By Proposition \ref{PropDifferentEmbeddingsIntoFlag2} the following diagram 
    commutes.
    \begin{equation*}
        \xymatrix{
            & \prod_{\sigma\in\mf S_0} I(\FF)\backslash \mc 
            \GL_n(\FF(\!(u)\!))/I(\FF)\ar[dd]^{(g_\sigma)_\sigma\mapsto 
                (u^e(g_\sigma^\tau)^{-1})_\sigma}\\
        \Aut(\mc L)(\FF)\backslash 
        \Mloc(\FF)\ar[ur]^{\iota_1}\ar[dr]_{\iota_2}& \\
        & \prod_{\sigma\in\mf S_0} I(\FF)\backslash \mc \GL_n(\FF(\!(u)\!))/I(\FF)
        }
    \end{equation*}
\end{remark}
\subsection{The extended affine Weyl group}\label{SecExtAffWeylUni2}
Let $T$ be the maximal torus of diagonal matrices in $\GL_n$ and let $N$ be its 
normalizer.  We denote by $\widetilde{W}=N(\FF(\!(u)\!))/T(\FF[\![u]\!])$ the 
extended affine Weyl group of $\GL_n$ with respect to $T$.
Setting $W=S_n$ and $X=\ZZ^n$, the group homomorphism $\upsilon:W\ltimes X\to 
N(\FF(\!(u)\!)),\ (w,\lambda)\mapsto A_wu^\lambda$ induces an isomorphism 
$W\ltimes X\too{\sim} \widetilde{W}$. We use it to identify $\widetilde{W}$ with 
$W\ltimes X$ and
consider $\widetilde{W}$ as a subgroup of $\GL_n(\FF(\!(u)\!)\!)$ via 
$\upsilon$.

To avoid any confusion of the product inside $\widetilde{W}$ and the canonical 
action of $S_n$ on $\ZZ^n$, we will always denote the element of 
$\widetilde{W}$ corresponding to $\lambda\in X$ by $u^\lambda$.

Recall from \cite[\textsection 2.5]{gy1} the notion of an extended alcove 
$(x_i)_{i=0}^{n-1}$ for $\GL_{n}$. Also recall the standard alcove 
$(\omega_i)_i$. As in loc.\ cit.\ we identify $\widetilde{W}$ with the set of 
extended alcoves by using the standard alcove as a base point.

Let $r,s\in\NN$ with $r+s=ne$ and write $\mathbf{e}=(e^{(n)})$.
\begin{defn}[{Cf.\ \cite{kottwitz_rapoport}, \cite[Definition 2.4]{gy1}}]
    An extended alcove $(x_i)_{i=0}^{n-1}$ is called \emph{$r$-permissible} if 
    it satisfies the following conditions for all $i\in\{0,\dots,n-1\}$.
    \begin{enumerate}
        \item $\omega_i\leq x_i \leq \omega_i+\mathbf{e}$, where $\leq$ is to be 
            understood componentwise.
        \item $\sum_{j=1}^{n} x_i(j)=s-i$.
    \end{enumerate}
    Denote by $\Perm_{r}$ the set of all $r$-permissible extended alcoves.
\end{defn}
\begin{prop}
    The inclusion $N(\FF(\!(u)\!))\subset \GL_n(\FF(\!(u)\!))$ induces a bijection 
    $\widetilde{W}\too{\sim} I(\FF)\backslash \mc 
    \GL_n(\FF(\!(u)\!))/I(\FF)$. In other words,
    \begin{equation*}
        \GL_n(\FF(\!(u)\!))=\coprod_{x\in\widetilde{W}} I(\FF)xI(\FF).
    \end{equation*}
    Under this bijection, the subset
    \begin{equation*}
        \Aut(\overline{\mc L})(\FF[u]/u^e)\backslash \Me^{e,n,r}(\FF)\subset I(\FF)\backslash \GL_n(\FF(\!(u)\!))/I(\FF)
    \end{equation*}
    of \eqref{Eqalphabar} corresponds to the subset $\Perm_{r}\subset \widetilde{W}$.
\end{prop}
\begin{proof}
    The first statement is the well-known Iwahori decomposition. The second 
    statement follows easily from the explicit description of the image of 
    $\alpha$ in Proposition \ref{PropLocalModelintoFlagUni2}.
\end{proof}
\begin{cor}\label{CorIndexSetUni2}
    With respect to the embedding $\iota_1$ of  Theorem \ref{ThmIndicesUni2}, 
    the set $\prod_{\sigma\in\mf S_0} \Perm_{\bar{r}_\sigma}$ constitutes a set of representatives of $\Aut(\mc 
    L)(\FF)\backslash \Mloc(\FF)$.
\end{cor}
The following lemma will be used below.
\begin{lem}\label{Lemxprime}
    Let $x\in\widetilde{W}$. Write $x=wu^\lambda$ with $w\in W,\ \lambda\in X$. Define 
    $w'\in W$ and $\lambda'\in X$ by 
    \begin{equation*}
        w'(i)=n+1-w(n+1-i),\quad 1\leq i\leq n
    \end{equation*}
    and
    \begin{equation*}
        \lambda'(i)=e-\lambda(n+1-i),\quad 1\leq i\leq n.
    \end{equation*}
    Let $x'=w'u^{\lambda'}$. Then $x'=u^e(x^\tau)^{-1}$.
\end{lem}
\begin{proof}
    This is an easy computation.
\end{proof}
\subsection{The \texorpdfstring{$p$}{p}-rank on a KR stratum}
Recall from Section \ref{SecLocalModelGen} the scheme 
$\mc A/\mc O_{E_\mc Q}$ associated with our choice of PEL datum, and the KR stratification 
\begin{equation*}
    \mc A(\FF)=\coprod_{x\in \Aut(\mc L)(\FF)\backslash \Mloc(\FF)}\mc A_x.
\end{equation*}
We have identified the occurring index set with
    $\prod_{\sigma\in\mf S_0}\Perm_{\overline{r}_\sigma}$ in Corollary 
    \ref{CorIndexSetUni2}.
    We can then state the following result.
\begin{thm}\label{ThmPrankUni2}
    Let $x=(x_\sigma)_\sigma\in \prod_{\sigma\in\mf S_0}\Perm_{\overline{r}_\sigma}$. Write $x_\sigma=w_\sigma 
    u^{\lambda_\sigma}$ with $w_\sigma\in W,\ \lambda_\sigma\in X$ and define 
    elements $w'_\sigma\in W$ and $\lambda_\sigma'\in X$ as in Lemma 
    \ref{Lemxprime}.
    Then the $p$-rank on $\mc A_x$ is 
    constant with value
    \begin{equation*}
        g\cdot \left|\left\{1\leq i\leq n\middle\vert \forall \sigma\in\mf S_0\left(
            \begin{aligned}
                w_\sigma(i)=w'_\sigma(i)=i\ \wedge\\
                \lambda_\sigma(i)=\lambda'_\sigma(i)=0
            \end{aligned}\right)\right\}\right|.
    \end{equation*}
\end{thm}
\begin{proof}
    Follows from Proposition \ref{PropDifferentEmbeddingsIntoFlag2} and Lemma \ref{Lemxprime}
    by the arguments of the proof of Theorem \ref{ThmPrank}.
\end{proof}
\section{The split unitary case}\label{SecUni3}
\subsection{The PEL datum}\label{SecPELUni3}
We start with the PEL datum defined in Section \ref{SecPELUni}. 
We assume that $p\mc O_{F_0}=(\mc P_0)^{e}$ for a single prime $\mc P_0$ of $\mc 
O_{F_0}$ and that $\mc P_0\mc O_F=\mc P_+\mc P_-$ for two distinct primes $\mc 
P_\pm$ of $\mc O_F$. Consequently $\mc P_-=(\mc P_+)^\inv$. Denote by 
$f_0=[k_{\mc P_0}:\FF_p]$ the corresponding inertia degree. We fix once and for 
all a uniformizer $\pi_0$ of $\mc O_{F_0}\otimes\ZZ_{(p)}$.

For typographical reasons, we denote the ring of integers in $(F_0)_{\mc P_0}$ by 
$\mc O_{\mc P_0}$. The inclusion $\mc O_{F_0}\hookrightarrow \mc O_F$ induces identifications
\begin{equation}
	\label{EqSplittingOfCompletion}
    \begin{aligned}
        \mc O_{F}\otimes\ZZ_p&= \mc O_{\mc P_0}\times \mc O_{\mc P_0},\\
        F\otimes\QQ_p&= (F_0)_{\mc P_0}\times (F_0)_{\mc P_0}.
    \end{aligned}
\end{equation}
Here the first (resp.\ second) factor is always supposed to correspond to $\mc 
P_+$ (resp.\ $\mc P_-$). Under \eqref{EqSplittingOfCompletion}, the base-change
$F\otimes\QQ_p\to F\otimes\QQ_p$ of $\inv$ takes the simple form 
$(F_0)_{\mc P_0}\times (F_0)_{\mc P_0}\to (F_0)_{\mc P_0}\times (F_0)_{\mc P_0},\ (a,b)\mapsto (b,a)$.

The identification \eqref{EqSplittingOfCompletion} further induces a 
decomposition $V\otimes\QQ_p=V_+\times V_-$ into $(F_0)_{\mc P_0}$-vector spaces 
$V_\pm$.  The pairing $\pairtd'_{\QQ_p}$ decomposes into its 
restrictions $\pairtd_\pm:V_\pm\times V_\mp\to (F_0)_{\mc P_0}$.
Both $\pairtd_+$ and $\pairtd_-$ are perfect $(F_0)_{\mc P_0}$-bilinear pairings
and they are related by the equation $\pairt{v}{w}_+=-\pairt{w}{v}_-,\ 
v\in V_+, w\in V_-$.

Denote by $\mf C_0=\mf C_{\mc O_{\mc P_0}\mid \ZZ_p}$ the corresponding inverse 
different and fix a generator $\delta_0$ of $\mf C_0$. 
We fix bases $(e_{1,\pm},\dots,e_{n,\pm})$ of $V_\pm$ over 
$\fpn$ such that $\pairt{e_{i,+}}{e_{n+1-j,-}}_+=\delta_0 \delta_{ij}$ for 
$1\leq i,j\leq n$.

Let $0\leq i< n$. We denote by $\Lambda_{i,\pm}$ the 
$\mc O_{\mc P_0}$-lattice in $V_{\pm}$ with basis 
\begin{gather}
    \mf E_{i,\pm}=(\pi_0^{-1}e_{1,\pm},\ldots,\pi_0^{-1}e_{i,\pm},e_{i+1,\pm},\ldots,e_{n,\pm}).
\end{gather}
For $k\in\ZZ$ we further define $\Lambda_{nk+i,\pm}=\pi_0^{-k}\Lambda_{i,\pm}$ and we denote 
by $\mf E_{nk+i,\pm}$ the corresponding basis obtained from $\mf E_{i,\pm}$.
Then $\mc L_\pm=(\Lambda_{i,\pm})_i$ is a complete chain of $\mc O_{\mc P_0}$-lattices in 
$V_\pm$. 

Let $i\in\ZZ$. We denote by $\rho_{i,\pm}:\Lambda_{i,\pm}\to \Lambda_{i+1,\pm}$ 
the inclusion and by $\vartheta_{i,\pm}:\Lambda_{n+i,\pm}\to \Lambda_{i,\pm}$ the 
isomorphism given by multiplication with $\pi_0$. Then
$(\Lambda_{i,\pm},\rho_{i,\pm},\vartheta_{i,\pm})$ is a chain of $\mc O_{\mc 
P_0}$-modules of type $(\mc L_\pm)$ which, by abuse of notation, we also denote 
by $\mc L_\pm$.

For $(i,j)\in\ZZ\times\ZZ$ we define $\Lambda_{(i,j)}:=\Lambda_{i,+}\times 
\Lambda_{j,-}$. Then $\Lambda_{(i,j)}$ is an $\mc O_F\otimes\ZZ_p$-lattice in $V_{\QQ_p}$. A basis $\mf E_{(i,j)}$ of $\Lambda_{(i,j)}$ 
over $\mc O_F\otimes\ZZ_p$ is given by the diagonal in $\mf E_{i,+}\times \mf E_{j,-}$.
Then $\mc L=(\Lambda_{i,j})_{(i,j)}$ is a complete multichain of $\mc O_{F}\otimes 
\ZZ_p$-lattices in $V_{\QQ_p}$. For $(i,j)\in\ZZ\times\ZZ$
the dual lattice 
$\Lambda_{(i,j)}^\vee:=\{x\in V_{\QQ_p}\mid 
    \pairt{x}{\Lambda_{(i,j)}}_{\QQ_p}\subset\ZZ_p\}$
of $\Lambda_{(i,j)}$ is given by $\Lambda_{(-j,-i)}$. Consequently the 
multichain $\mc L$ is a self-dual. 

Let $(i,j)\in\ZZ\times\ZZ$. We denote by $\rho_{(i,j),+}:\Lambda_{(i,j)}\to \Lambda_{(i+1,j)},\ 
\rho_{(i,j),-}:\Lambda_{(i,j)}\to \Lambda_{(i,j+1)}$ and 
$\rho_{(i,j)}:\Lambda_{(i,j)}\to \Lambda_{(i+1,j+1)}$ the inclusions. We denote 
by
$\vartheta_{(i,j),+}:\Lambda_{(n+i,j)}\to \Lambda_{(i,j)}$ 
(resp.\ $\vartheta_{(i,j),-}:\Lambda_{(i,n+j)}\to \Lambda_{(i,j)}$,
resp.\ $\vartheta_{(i,j)}:\Lambda_{(n+i,n+j)}\to \Lambda_{(i,j)}$) the isomorphism given by 
multiplication with $\pi_0$ in the first (resp.\ second, resp.\ first and second) 
component. We further denote by $\pairtd_{(i,j)}:\Lambda_{(i,j)}\times \Lambda_{(-j,-i)}\to \ZZ_p$ the 
restriction of $\pairtd_{\QQ_p}$.

We find that $(\mc L_+,\mc L_-)$, equipped with $(\pairtd_{(i,j)})_{(i,j)}$, is a 
polarized multichain of $\mc O_F\otimes \ZZ_p$-modules of type $(\mc L)$, which, by 
abuse of notation, we also denote by $\mc L=\mc L^{\mathrm{split}}$.

Denote by $\paird_{(i,j)}:\Lambda_{(i,j)}\times \Lambda_{(-j,-i)}\to \mc O_{\mc 
P_0}\times \mc O_{\mc P_0}$ the restriction of the $\inv$-hermitian form 
$(\delta_0^{-1},-\delta_0^{-1})\pairtd'_{\QQ_p}$. It is the $\inv$-sesquilinear form described by 
the matrix $\widetilde{I}_n$ with respect to $\mf E_{(i,j)}$ and $\mf E_{(-j,-i)}$.

Denote by $\Sigma_0$ the set of all embeddings $F_0\hookrightarrow\RR$ and by 
$\Sigma$ the set of all embedding $F\hookrightarrow \CC$. The inclusion 
$\mc O_{F_0}\hookrightarrow \mc O_F$ induces an identification of $k_{\mc 
P_\pm}/\FF_p$ with $k_{\mc P_0}/\FF_p$. We write $\mf S_0=\Gal(k_{\mc 
P_0}/\FF_p)$  and also identify $\Gal(k_{\mc P_\pm}/\FF_p)$ with $\mf S_0$.
Let $E'$ be the Galois closure of $F$ inside $\CC$ and choose a prime $\mc Q'$ 
of $E'$ over $\mc P_+$. Consider the decomposition $\Sigma=\Sigma_+\amalg 
\Sigma_-$ and the maps $\gamma_0:\Sigma_0\to \mf S_0,\ 
\gamma_{\pm}:\Sigma_{\pm}\to \mf S_0$ of Lemma \ref{LemGammaSplit}. For $\sigma\in \Sigma_0$ we denote by 
$\tau_{\sigma,\pm}$ the unique lift of $\sigma$ to $\Sigma_\pm$.
Exactly as in Section \ref{SecUni}, we define for each $\sigma\in\Sigma_0$ 
integers $r_\sigma,s_\sigma$ with $r_\sigma+s_\sigma=n$,\footnote
{In Section \ref{SecUni} we have written $\tau_{\sigma,1}$ and $\tau_{\sigma,2}$ 
instead of $\tau_{\sigma,+}$ and $\tau_{\sigma,-}$, respectively.} and using 
these the element $J\in\End_{B\otimes\RR}(V\otimes\RR)$. Denote by $V_{\CC,-i}$ 
the $(-i)$-eigenspace of $J_\CC$. As before, we construct an $\mc O_F\otimes\mc 
O_{E'}$-module $M_0$ which is finite locally free over $\mc O_{E'}$, such that
$M_0\otimes_{\mc O_{E'}}\CC=V_{\CC,-i}$ as $\mc O_F\otimes\CC$-modules.
\subsection{The special fiber of the determinant morphism}
For $\sigma\in\mf S_0$ we write
\begin{equation*}
    \overline{r}_\sigma=\sum_{\sigma'\in \gamma_0^{-1}(\sigma)} 
    r_{\sigma'}\quad\text{and}\quad \overline{s}_\sigma=\sum_{\sigma'\in 
        \gamma_0^{-1}(\sigma)} s_{\sigma'}.
\end{equation*}
As the fibers of $\gamma_0$ have cardinality $e$, it follows that 
$\overline{r}_\sigma+\overline{s}_\sigma=ne$.

We fix once and for all an embedding $\iota_{\mc Q'}:k_{\mc Q'}\hookrightarrow 
\FF$. We consider $\FF$ as an $\mc O_{E'}$-algebra with respect to the 
composition $\mc O_{E'}\too{\rho_{\mc Q'}}k_{\mc Q'}\overset{\iota_{\mc Q'}}{\hookrightarrow} \FF$. Also 
$\iota_{\mc Q'}$ induces an embedding $\iota_{\mc P_0}:k_{\mc P_0}\hookrightarrow \FF$ and 
thereby an identification of the set of all embeddings $k_{\mc P_0}\hookrightarrow \FF$ with $\mf S_0$.

Consider the isomorphism
\begin{equation}
    \label{EqdecompspecialfiberUni3}
    \mc O_{F}\otimes \FF=\prod_{\sigma\in \mf S_0} \FF[u]/(u^{e})\times \FF[u]/(u^{e})
\end{equation}
obtained from \eqref{EqSplittingOfCompletion} and our choice of uniformizer 
$\pi_0$.

\begin{prop}\label{PropDeterminantSpecialFiberUni3}
    Let $x\in \mc O_F$ and let $\bigl((q_{\sigma,+},q_{\sigma,-})\bigr)_\sigma\in \prod_{\sigma\in \mf S_0} \FF[u]/(u^{e})\times \FF[u]/(u^{e})$ be the element corresponding to $x\otimes 1$
    under \eqref{EqdecompspecialfiberUni3}. Then
    \begin{equation*}
        \chi_\FF(x| M_0\otimes_{\mc O_{E'}}\FF)=\prod_{\sigma\in \mf S_0}\bigl(T-q_{\sigma,+}(0)\bigr)^{\overline{s}_\sigma}\bigl(T-q_{\sigma,-}(0)\bigr)^{\overline{r}_\sigma}
    \end{equation*}
    in $\FF[T]$.
\end{prop}
\begin{proof}
    Reduce $\chi_{\mc O_{E'}}(x|M_0)$ modulo $\mc Q'$, using \eqref{EqRedOfSigmaUnify}.
\end{proof}
Denote by $E=\QQ(\tr_\CC(x\otimes 1|V_{-i});\ x\in F)$ the reflex field and define $\mc Q=\mc Q'\cap \mc O_E$.
The morphism $\det_{V_{-i}}$ is defined over 
$\mc O_E$.
\subsection{The local model}
For the chosen PEL datum, Definition \ref{DefnLocalModelGen} amounts to the following.
\begin{defn}\label{DefnLocalModelUni3}
	The local model $\Mloc=\Me^{\mathrm{loc,split}}$ is the functor on the category of
	$\mc O_{E_\mc Q}$-algebras with $\Mloc(R)$ the set of tuples
	$(t_{(i,j)})_{(i,j)\in \ZZ\times\ZZ}$ of
	$\mc O_{F}\otimes R$-submodules $t_{(i,j)}\subset \Lambda_{(i,j),R}$
	satisfying the following conditions for all $(i,j)\in\ZZ\times 
	\ZZ$.
	\renewcommand\theenumi{\alph{enumi}}
	\begin{enumerate}
		\item\label{DefnLocalModelUni3-Functoriality}
			$\rho_{(i,j),+,R}(t_{(i,j)})\subset t_{(i+1,j)}$ and 
			$\rho_{(i,j),-,R}(t_{(i,j)})\subset t_{(i,j+1)}$.
		\item\label{DefnLocalModelUni3-Projectivity} The quotient 
			$\Lambda_{(i,j),R}/t_{(i,j)}$ is a finite locally free $R$-module.
		\item\label{DefnLocalModelUni3-Determinant} We have an equality
			\begin{equation*}
				\det_{\Lambda_{(i,j),R}/t_{(i,j)}}=\det_{V_{-i}}\otimes_{\mc O_E}R
			\end{equation*}
			of morphisms $V_{\mc O_F\otimes R}\to \AF^1_R$.
		\item\label{DefnLocalModelUni3-DualityCondition} Under the pairing 
			$\pairtd_{(i,j),R}:\Lambda_{(i,j),R}\times \Lambda_{(-j,-i),R}\to R$, 
			the submodules $t_{(i,j)}$ and $t_{(-j,-i)}$ 
			pair to zero.
\item\label{DefnLocalModelUni3-Periodicity}
	$\vartheta_{(i,j),+,R}(t_{(n+i,j)})=t_{(i,j)}$ and 
	$\vartheta_{(i,j),-,R}(t_{(i,n+j)})=t_{(i,j)}$.
	\end{enumerate}
\end{defn}
\begin{remark}\label{RemarkLocModDecompUni3}
	Let $R$ be an $\mc O_{E_\mc Q}$-algebra and let $(t_{(i,j)})_{(i,j)}\in 
    \Mloc(R)$. For $(i,j)\in\ZZ\times\ZZ$, the decomposition \eqref{EqSplittingOfCompletion} induces a decomposition
    $t_{(i,j)}=t_{(i,j),+}\times t_{(i,j),-}$ into $\mc O_{\mc P_0}\otimes_{\ZZ_p} R$-submodules $t_{(i,j),+}\subset \Lambda_{i,+,R}$ and 
    $t_{(i,j),-}\subset \Lambda_{j,-,R}$. As in Remark \ref{RemarkDecompLocalModel} one sees that 
    $t_{(i,j),+}$ (resp.\ $t_{(i,j),-}$) is independent of $j$ (resp.\ $i$).
    Writing $t_{i,+}=t_{(i,j),+}$ and $t_{j,-}=t_{(i,j),-}$, the tuple
    $(t_{(i,j)})_{(i,j)}$ is determined by the pair of tuples $((t_{i,+})_i, (t_{j,-})_j)$.
\end{remark}
Recall from Section \ref{SecUni2} the chain $\mc L^{\mathrm{inert}}$ and the functor $\Me^{\mathrm{loc,inert}}$.
The identifications \eqref{EqdecompspecialfiberUni2} and 
\eqref{EqdecompspecialfiberUni3}, together with our choices of bases, give rise to a canonical identification of the 
tuple $(\Lambda_{(i,i),\FF},\rho_{(i,i),\FF},\vartheta_{(i,i),\FF},\pairtd_{(i,i),\FF})_i$
with the chain $\mc L^{\mathrm{inert}}\otimes_{\ZZ_p}\FF$.

We can then state the following result.
\begin{prop}\label{PropMSpinvsMLoc}
	\begin{enumerate}
		\item 
			The morphism $\Me^{\mathrm{loc,split}}_\FF\to \Me^{\mathrm{loc,inert}}_\FF$ given on $R$-valued 
			points by
			\begin{align*}
				\Me^{\mathrm{loc,split}}_\FF(R)&\to \Me^{\mathrm{loc,inert}}_\FF(R),\\
				(t_{(i,j)})_{(i,j)}&\mapsto (t_{(i,i)})_i
			\end{align*}
			is an isomorphism.
        \item The morphism $\Aut(\mc L^{\mathrm{split}})_\FF\to \Aut(\mc L^{\mathrm{inert}})_\FF$ given on $R$-valued 
			points by
			\begin{align*}
                \Aut(\mc L^{\mathrm{split}})_\FF(R)&\to \Aut(\mc L^{\mathrm{inert}})_\FF(R),\\
				(\varphi_{(i,j)})_{(i,j)}&\mapsto (\varphi_{(i,i)})_i
			\end{align*}
			is an isomorphism.
	\end{enumerate}
\end{prop}
\begin{proof}
    Clear in view of Remark \ref{RemarkLocModDecompUni3} and Propositions 
    \ref{PropDeterminantSpecialFiberUni2}, \ref{PropDeterminantSpecialFiberUni3}.
\end{proof}
Consequently all the statements about $\Me^{\mathrm{loc,inert}}_\FF$ from 
Section \ref{SecUni2} are also valid for $\Me^{\mathrm{loc,split}}_\FF$.
\subsection{The \texorpdfstring{$p$}{p}-rank on a KR stratum}
Recall from Section \ref{SecLocalModelGen} the scheme 
$\mc A/\mc O_{E_\mc Q}$ associated with our choice of PEL datum, and the KR stratification 
\begin{equation*}
    \mc A(\FF)=\coprod_{x\in \Aut(\mc L)(\FF)\backslash \Mloc(\FF)}\mc A_x.
\end{equation*}
We have identified the occurring index set with
    $\prod_{\sigma\in\mf S_0}\Perm_{\overline{r}_\sigma}$ in Corollary 
    \ref{CorIndexSetUni2}. We can then state the following result.
    \begin{thm}\label{ThmPrankUni3}
    Let $x=(x_\sigma)_\sigma\in \prod_{\sigma\in\mf 
    S_0}\Perm_{\overline{r}_\sigma}$. Write $x_\sigma=w_\sigma 
    u^{\lambda_\sigma}$ with $w_\sigma\in W,\ \lambda_\sigma\in X$.
    Then the $p$-rank on $\mc A_{x}$ is constant with value
    \begin{align*}
        &\phantom{{}+{}}g_0\cdot |\{1\leq i\leq n\mid \forall \sigma\in\mf 
        S_0(w_\sigma(i)=i \wedge \lambda_\sigma(i)=0)\}|\\
        &+g_0\cdot |\{1\leq i\leq n\mid \forall \sigma\in\mf S_0(w_\sigma(i)=i 
        \wedge \lambda_\sigma(i)=e)\}|.
    \end{align*}
\end{thm}
\begin{proof}
    Define elements $w'_\sigma\in W$ and $\lambda_\sigma'\in X$ as in Lemma 
    \ref{Lemxprime}.
    Let $t=(t_{(i,j)})_{(i,j)}\in \Mloc(\FF)$ and let $(t_{i,\pm})_i$ be the two 
    associated tuples of Remark \ref{RemarkLocModDecompUni3}. Let 
    $(i,j)\in\ZZ\times\ZZ$.
    We have the following equivalences.
    \begin{align*}
        \Lambda_{(i,j),\FF}=\im \rho_{(i-1,j),+,\FF}+t_{(i,j)}&\Leftrightarrow 
        \Lambda_{i,+,\FF}=\im \rho_{i-1,+,\FF}+t_{i,+},\\
        \Lambda_{(i,j),\FF}=\im \rho_{(i,j-1),-,\FF}+t_{(i,j)}&\Leftrightarrow 
        \Lambda_{j,-,\FF}=\im \rho_{j-1,-,\FF}+t_{j,-}.
    \end{align*}
    Assume now that $t$ lies in the $\Aut(\mc L^{\mathrm{inert}})(\FF)$-orbit 
    corresponding to $x$ under the identifications of
    Corollary \ref{CorIndexSetUni2}.
    Consider the chain of neighbors
    \begin{equation*}
        \Lambda_{(0,0)}\subset \Lambda_{(1,0)}\subset \dots\subset 
        \Lambda_{(n,0)}\subset \Lambda_{(n,1)}\subset \dots\subset 
        \Lambda_{(n,n)}=\pi_0^{-1}\Lambda_{(0,0)},
    \end{equation*}
    and let $1\leq i\leq n$. By Propositions \ref{ProppRankFirstVersionGen} and 
    \ref{PropEquivCondGen}, the claim of the theorem follows once we can show 
    the following equivalences.
    \begin{align*}
        \Lambda_{i,+,\FF}=\im \rho_{i-1,+,\FF}+t_{i,+}&\Leftrightarrow 
        \forall\sigma\in\mf S_0(w_\sigma(i)=i\wedge \lambda_\sigma(i)=0),\\
        \Lambda_{i,-,\FF}=\im \rho_{i-1,-,\FF}+t_{i,-}&\Leftrightarrow 
        \forall\sigma\in\mf S_0(w'_\sigma(i)=i\wedge \lambda'_\sigma(i)=0).
    \end{align*}
    These equivalences follow from Proposition 
    \ref{PropDifferentEmbeddingsIntoFlag2} and Lemma \ref{Lemxprime}
    by the arguments of the proof of Theorem \ref{ThmPrank}.
\end{proof}
\subsection{An application to the dimension of the \texorpdfstring{$p$}{p}-rank 0 locus}
Assume from now on that $F_0=\QQ$, so that $F/\QQ$ is an imaginary quadratic 
extension in which $p$ splits. We write $r=r_{\id_\QQ}$ and $s=s_{\id_\QQ}$, so 
that $n=r+s$. Also write $I_n=\{1,\dots,n\}$.

Note that the moduli problem $\mc A$ is a special case of the ``fake unitary 
case'' considered in \cite{haines}. Concretely, the moduli problem defined in 
\cite[\textsection 5.2]{haines} specializes to $\mc A$ for $D=F$.

Denote by $\ell:\widetilde{W}\to \NN$ the length function defined in 
\cite[\textsection 2.1]{gy1}.
\begin{lem}
    Let $x\in\Perm_r$. The smooth $\FF$-variety $\mc A_{x}$ is equidimensional 
    of dimension $\ell(x)$.
\end{lem}
\begin{proof}
    We know from \cite[Lemma 13.1]{haines} that $\mc A_{x}$ is non-empty.
    The rest of the proof is identical to the one of Lemma \ref{LemPropsOfKR}.
\end{proof}
Let us state Theorem \ref{ThmPrankUni3} in this special case.
\begin{thm}\label{ThmPrankSpecialCase}
    Let $x\in\Perm_r$. Write $x=wu^\lambda$ with $w\in W,\lambda\in X$. Then the 
    $p$-rank on $\mc A_x$ is constant with value $|\mathrm{Fix}(w)|$, where 
    $\mathrm{Fix}(w)=\{i\in I_n\mid w(i)=i\}$.
\end{thm}
We want to use this result to compute the dimension of the $p$-rank 0 locus in 
$\mc A_{\FF}$. We do this by copying the approach of \cite[\textsection 8]{gy2}.  

Denote by $\Perm_r^{(0)}$ the subset of those $x\in \Perm_r$ such that the 
$p$-rank on $\mc A_{x}$ is equal to 0. Denote by $W_{n,r}$ the subset of 
those $w \in W$ satisfying $\mathrm{Fix}(w)=\emptyset$ and
\begin{equation}\label{EqDefnWnr}
    |\{i\in I_n\mid w(i)<i\}|=r.
\end{equation}
\begin{lem}[{Cf.\ \cite[Lemma 8.1]{gy2}}]\label{LemPerm0}
    The canonical projection $\widetilde{W}\to W$ induces a bijection
    $\Perm_r^{(0)}\to W_{n,r}$. Its inverse is given by $w\mapsto 
    u^{\lambda(w)}w$ with \begin{equation*}
        \lambda(w)(i) = \begin{cases}
            0, & \text{if } w^{-1}(i) > i \\
                1, & \text{if } w^{-1}(i) < i
        \end{cases}, \quad i\in I_n.
    \end{equation*}
\end{lem}
\begin{proof}
    This is an easy combinatorial consequence of Theorem 
    \ref{ThmPrankSpecialCase} and the interpretation of $\Perm_r$ in terms of
    extended alcoves, see Section \ref{SecExtAffWeylUni2}.
\end{proof}
Define for $\sigma\in S_n$ the following sets and natural numbers.
\begin{align*}
    A_\sigma&=\{(i,j)\in (I_n)^2\mid 
    i<j<\sigma(j)<\sigma(i)\},&a_\sigma&=|A_\sigma|,\\
    \tilde{A}_\sigma&=\{(i,j)\in (I_n)^2\mid 
    \sigma(j)<\sigma(i)<i<j\},&\tilde{a}_\sigma&=|\tilde{A}_\sigma|,\\
    B_\sigma&=\{(i,j)\in (I_n)^2\mid 
    \sigma(i)<i<j<\sigma(j)\},&b_\sigma&=|B_\sigma|,\\
    \tilde{B}_\sigma&=\{(i,j)\in (I_n)^2\mid 
    i<\sigma(i)<\sigma(j)<j\},&\tilde{b}_\sigma&=|\tilde{B}_\sigma|,\\
    N_\sigma&=a_\sigma+\tilde{a}_\sigma+b_\sigma+\tilde{b}_\sigma.
\end{align*}
Note that $N_\sigma=N_{\sigma^{-1}}$ in view of the obvious identities 
$a_\sigma=\tilde{a}_{\sigma^{-1}}$ and $b_\sigma=\tilde{b}_{\sigma^{-1}}$.
\begin{prop}
    Let $u^{\lambda}w\in \Perm_r^{(0)},\ w\in W,\lambda\in X$. Then 
    $\ell(u^{\lambda}w)=N_w$.
\end{prop}
\begin{proof}
    Denote by $e_i$ the $i$-th standard basis vector of $\ZZ^n$.
    The positive roots $\beta>0$ of $\GL_{n}$ are given by $\beta_{ij}=e_i-e_j,\ 
    1\leq i<j\leq n$. Denote by $\paird$ the standard symmetric pairing on 
    $\ZZ^n$, determined by $\pair{e_i}{e_j}=\delta_{ij}$.
    By \cite[(8.1)]{gy2} the following Iwahori-Matsumoto formula holds.
    \begin{equation}\label{EqIwahoriMatsumoto}
        \ell(u^\lambda 
        w)=\sum_{\substack{\beta>0\\w^{-1}\beta>0}}|\pair{\beta}{\lambda}|+ 
        \sum_{\substack{\beta>0\\w^{-1}\beta<0}}|\pair{\beta}{\lambda}+1|.  
    \end{equation}
    Using Lemma \ref{LemPerm0}, the equality $\ell(u^{\lambda}w)=N_{w^{-1}}$ 
    readily follows.
\end{proof}
Define
\begin{equation*}
    N_{n,r}:=\min\bigl((r-1)(n-r),r(n-r-1)\bigr)=
    \begin{cases}
        (r-1)(n-r),&\text{if } r\leq n/2,\\
        r(n-r-1),&\text{if } r\geq n/2.
    \end{cases}
\end{equation*}
\begin{prop}
    Let $\sigma\in W_{n,r}$. Then $N_\sigma\leq N_{n,r}$.
\end{prop}
\begin{proof}
    Consider the set $M=\{(n,r,i_0)\in\NN^3\mid 1\leq r\leq n-1,\ 2\leq i_0\leq 
    n\}$ and equip it with the lexicographical ordering $<$, which is a 
    well-ordering on $M$. For $(n,r,i_0)\in M$ we define 
    $W_{n,r,i_0}=\{\sigma\in W_{n,r}\mid \min\{2\leq i\leq n\mid 
    \sigma(i)<i\}=i_0\}$. Denote by $\mc P(n,r,i_0)$ the following statement.
        \begin{equation*}
            \forall \sigma\in W_{n,r,i_0}:\ N_\sigma\leq N_{n,r}.
    \end{equation*}
    We will prove it by induction on $(n,r,i_0)$.

    Let $(n,r,i_0)\in M,\ \sigma\in W_{(n,r,i_0)}$ and assume that $\mc 
    P(n',r',i'_0)$ is true for all $(n',r',i'_0)\in M$ with 
    $(n',r',i'_0)<(n,r,i_0)$. Set $\sigma'=\sigma\circ (i_0-1,i_0)$. We 
    distinguish the following four cases.

    \textbf{Case 1:} $\sigma(i_0)<i_0-1$ and $\sigma(i_0-1)>i_0$.

    \textbf{Case 2:} $\sigma(i_0)=i_0-1$ and $\sigma(i_0-1)>i_0$.
   
    \textbf{Case 3:} $\sigma(i_0)<i_0-1$ and $\sigma(i_0-1)=i_0$.

    \textbf{Case 4:} $\sigma(i_0)=i_0-1$ and $\sigma(i_0-1)=i_0$.

    We use the example of \textbf{Case 2} to illustrate how to proceed.
    So assume that $\sigma(i_0)=i_0-1$ and $\sigma(i_0-1)>i_0$.

    We read off the following identities.
    \begin{align*}
        a_\sigma&=a_{\sigma'},\\
        \tilde{a}_\sigma&=\tilde{a}_{\sigma'}+|\{j\in I_n\mid 
        \sigma(j)<i_0-1<i_0<j\}|,\\
        b_\sigma&=b_{\sigma'}+|\{j\in I_n\mid i_0-1<i_0<j<\sigma(j)\}|,\\ 
        \tilde{b}_\sigma&=\tilde{b}_{\sigma'}+|\{i\in I_n\mid 
        i<\sigma(i)<i_0-1<i_0\}|.
    \end{align*}

    Identifying $\{1,\ldots,\widehat{i_0-1},\ldots,n\}$ with $\{1,\ldots,n-1\}$, 
    we consider the restriction 
    $\sigma'\big|_{\{1,\ldots,\widehat{i_0-1},\ldots,n\}}$ as an element of 
    $W_{n-1,r-1,j_0}$ for some $j_0$. By induction hypothesis we know that
    $N_{\sigma'}=N_{\sigma'|_{\{1,\ldots,\widehat{i_0-1},\ldots,n\}}}\leq 
    N_{n-1,r-1}$.
    In view of $N_{n,r}-N_{n-1,r-1} \geq n-r-1$
    it therefore suffices to show the following two inequalities.
    \begin{align*}
            &i_0-2\geq |\{j\in I_n\mid \sigma(j)<i_0-1<i_0<j\}|+|\{i\in I_n\mid 
            i<\sigma(i)<i_0-1\}|,\\
            &n-r-(i_0-1)\geq |\{j\in I_n\mid i_0<j<\sigma(j)\}|.
    \end{align*}
    For the first inequality, it suffices to note that $\sigma$ maps both sets 
    in question into $I_{i_0-2}$.  On the other hand, by the definition of $i_0$ 
    we have $I_{i_0-1}\subset \{i\in I_n\mid i<\sigma(i)\}$, so that 
    \eqref{EqDefnWnr} implies the second inequality.
\end{proof}
\begin{prop}
    We have
    \begin{equation*}
        \max_{\sigma\in W_{n,r}}N_\sigma=N_{n,r}.
    \end{equation*}
\end{prop}
\begin{proof}
    It suffices to show that there is a $\sigma\in W_{n,r}$ satisfying
    $N_\sigma=N_{n,r}$. As $W_{n,r}\to W_{n,n-r},\ \sigma\mapsto \sigma^{-1}$ is 
    a bijection and as $N_\sigma=N_{\sigma^{-1}}$, we may assume that $r\leq 
    n/2$. One easily checks that
    \begin{equation*}
        \sigma=(1,2)(3,4)\cdots(2(r-1)-1,2(r-1))(2r-1,2r,2r+1,\ldots,n)\in 
        W_{n,r}
    \end{equation*}
    satisfies $N_\sigma=(r-1)(n-r)=N_{n,r}$.
\end{proof}
Denote by $\mc A^{(0)} \subset \mc A(\FF)$ the subset where the 
$p$-rank of the underlying abelian variety is equal to $0$. It is a closed 
subset and we equip it with the reduced scheme structure. From the discussion 
above we obtain the following result.
\begin{thm}\label{ThmPRank0Locus}
    $\dim \mc A^{(0)}=\min\bigl((r-1)(n-r),r(n-r-1)\bigr)$.
\end{thm}

\providecommand{\bysame}{\leavevmode\hbox to3em{\hrulefill}\thinspace}
\providecommand{\MR}{\relax\ifhmode\unskip\space\fi MR }
\providecommand{\MRhref}[2]{%
  \href{http://www.ams.org/mathscinet-getitem?mr=#1}{#2}
}
\providecommand{\href}[2]{#2}

\end{document}